\documentclass[11pt]{amsart}

\usepackage[paperheight=11in, paperwidth=8.5in, left=1in, top=1in, right=1in, bottom=1in]{geometry}
\usepackage[linktocpage=true, linktoc=all, colorlinks=true, linkcolor=black, citecolor=black, filecolor=black, urlcolor=black, pagebackref=false, pdfstartpage={1}, pdfstartview={FitH}, pdftitle={}, pdfauthor={}, pdfsubject={}, pdfcreator={}, pdfproducer={}, pdfkeywords={}]{hyperref}
\usepackage{graphicx} 
\usepackage{amsthm,amsfonts, amssymb, amscd,bm}

\usepackage{subcaption}
\usepackage{enumerate}
\usepackage{youngtab}
\usepackage{ytableau}
\usepackage{euscript}
\usepackage{booktabs,lscape}% tables
\usepackage{tabularx} 
\usepackage{array}    % For vertical centering in cells
\usepackage{nccmath, mathtools} %%%To fit large extended matrices
\usepackage{tikz}
\usetikzlibrary{matrix,arrows}
\tikzset{partition/.style={fill,circle,inner sep=1pt}}
\usetikzlibrary{positioning}
\usetikzlibrary{decorations.pathreplacing}
\usetikzlibrary{matrix,arrows}
\usetikzlibrary{calc}
\tikzset{partition/.style={fill,circle,inner sep=1pt},
         part/.style={baseline=0,scale=0.5,bend left=45},
         partlabel/.style={below}}
\usetikzlibrary{shapes,arrows}
\usetikzlibrary{snakes}
\tikzstyle{pnt}=[draw,ellipse,fill,inner sep=1pt]
\tikzstyle{opnt}=[draw,ellipse,inner sep=1pt]
\tikzstyle{opnt}=[ ]
\tikzstyle{pntt}=[draw,ellipse,fill,inner sep=0.5pt]
\tikzstyle{point}=[draw,ellipse,fill,inner sep=2pt]

\usepackage{multicol}

\newtheorem{theorem}{Theorem}[section]
\newtheorem{lemma}[theorem]{Lemma}
\newtheorem{corollary}[theorem]{Corollary}
\newtheorem{prop}[theorem]{Proposition}

\newtheorem{conj}[theorem]{Conjecture}
\theoremstyle{remark}
\newtheorem{remark}[theorem]{Remark}
\theoremstyle{remark}

\theoremstyle{definition}

\theoremstyle{definition}

\theoremstyle{definition}

\newtheorem{definition}[theorem]{Definition}

\newtheorem{eg_no_qed}[theorem]{Example}
\newenvironment{ex}[1][]{\begin{eg_no_qed}[#1]\pushQED{\qed}}{\popQED\end{eg_no_qed}}
\AfterEndEnvironment{ex}{\noindent\ignorespaces}

\newtheorem{innercustomgeneric}{\customgenericname}
\providecommand{\customgenericname}{}
\newcommand{\newcustomtheorem}[2]{%
  \newenvironment{#1}[1]
  {%
   \renewcommand\customgenericname{#2}%
   \renewcommand\theinnercustomgeneric{##1}%
   \innercustomgeneric
  }
  {\endinnercustomgeneric}
}

\newcustomtheorem{customthm}{Theorem}
\newcustomtheorem{customconj}{Conjecture}
\newcustomtheorem{customprop}{Proposition}
\numberwithin{equation}{section}

\newcommand{\C}{\mathbb{C}}           % Use for complex numbers.
\newcommand{\Z}{\mathbb{ Z}}           % Use for integers.
 
\newcommand{\Q}{\mathbb{ Q}}           % Use for quaternions
\newcommand{\spec}{\operatorname{Spec}}

\renewcommand{\ker}{\operatorname{ker }}

\newcommand{\im}{\operatorname{im }}

\newcommand{\rk}{\operatorname{rk}}

\newcommand{\GL}{\operatorname{GL}}

\newcommand{\ga}{\alpha}

\newcommand{\gd}{\delta}

\newcommand{\gs}{\sigma}

 \newcommand{\cb}{\mathcal{B}}
\newcommand{\cc}{\mathcal{C}}

 \newcommand{\ci}{\mathcal{I}}
 \newcommand{\cj}{\mathcal{J}}
 
 \newcommand{\cm}{\mathcal{M}}

  \newcommand{\cp}{\mathcal{P}} %Collection of cells
  
 \newcommand{\cs}{\mathcal{S}}

 \newcommand{\cv}{\mathcal{V}}

\newcommand{\SYT}{\mathrm{SYT}}

\newcommand{\al}{\alpha}

\newcommand{\SF}{\mathcal{B}}

\newcommand{\pl}{\mathsf{p}}

\newcommand{\E}{\Sigma}
\newcommand{\cupsigma}{\cc(\gs)}
\newcommand{\cell}{F_\gs}
\newcommand{\isigma}{\ci_\gs}

\newcommand{\Fl}{\mathcal{F}\ell}

\newcommand{\m}[2]{m(#1,#2)}

\newcommand{\mi}[1]{m(#1)}

\newcommand{\parsum}{t}

\title{Ideals defining components of two-row Springer fibers}

\author{Cristina Sabando-Alvarez}
\address{Department of Mathematics\\ Washington University in St. Louis \\ One Brookings Drive\\ St. Louis, Missouri\\ 63130\\ USA }
\email{sabandoalvarez@wustl.edu}

\author{Martha Precup}
\address{Department of Mathematics\\ Washington University in St. Louis \\ One Brookings Drive\\ St. Louis, Missouri\\ 63130\\ USA  }
\email{martha.precup@wustl.edu}

\begin{document}

\begin{abstract}  Springer fibers are subvarieties of the flag variety parameterized by nilpotent matrices. They are central objects of study in geometry representation theory. This paper focuses on two-row Springer fibers, those corresponding to nilpotent matrices with two Jordan blocks. Irreducible components of two-row Springer fibers are in bijection with two-row standard Young tableaux and also with noncrossing matchings. 

Inspired by the combinatorial commutative algebra of matrix Schubert varieties, we define a polynomial ideal for each noncrossing matching and prove that these ideals define the corresponding components of the Springer fiber. Our proofs leverage geometric descriptions of Springer fibers established by Fung, Stroppel--Webster, Fresse, and Goldwasser--Nadeem--Sun--Tymoczko.  Using these ideals to compute examples, we give two conjectural formulas for the cohomology class of each component of a two-row Springer fiber. We apply commutative algebra techniques to prove these conjectures for a specific family of two-row tableaux.
\end{abstract}

\maketitle

\setcounter{tocdepth}{1}
\tableofcontents

\section{Introduction}
Springer fibers are important geometric objects in representation theory. Their cohomology rings carry a natural action of the symmetric group or, more generally, the Weyl group. In this paper, we study Springer fibers in the Type $A$ flag variety associated to nilpotent matrices with two Jordan blocks. Utilizing the combinatorics of noncrossing matchings, we identify specific polynomial ideals that define the irreducible components of these Springer fibers.

Let $n$ be a positive integer and $\Fl_n(\C)$ denote the flag variety of nested sequences of subspaces $V_\bullet = (\{0\}=V_0\subset V_1\subset V_2\subset \cdots \subset V_n=\mathbb{C}^n)$, where $\dim(V_i)=i$ for all $0\leq i\leq n$. Such a sequence is called a (full) flag. We identify the flag variety $\mathcal{F}\ell(n)$ with the quotient space $GL_n(\C)/B$ where $GL_n(\C)$ is the group of $n\times n$ invertible matrices and $B\subseteq GL_n(\C)$ is the subgroup of upper triangular matrices.   The Springer fiber $\SF_N$ is the collection of flags fixed under the action of an $n\times n$ nilpotent matrix $N$, that is, 
\[
\SF_N := \{V_\bullet \in \Fl_n(\C) \mid N(V_i)\subseteq V_{i-1}\}.
\]
The irreducible components of the Springer fiber $\mathcal{B}_N$ are in bijection with standard Young tableaux whose shape is determined by the Jordan type of $N$ \cite{Spaltenstein1976,Vargas79}. As such, the study of Springer fibers and their components naturally embody the interplay between tableaux combinatorics and geometry. Springer fibers can be defined in the flag variety associated to any complex reductive group $G$. This paper focuses on the case where $G=GL_n(\C)$.

Borel's presentation identifies the cohomology of the flag variety $\Fl_n(\C)$ as a quotient of the polynomial ring $\Z[x_1, \ldots, x_n]$. This identification naturally raises the problem of finding explicit polynomial representatives for the cohomology class $[Y]$ of any subvariety $Y \subseteq \Fl_n(\C)$. This problem provides the primary motivation for the present work and drives broader research into the subvarieties of the flag variety.

The problem from the previous paragraph is most extensively studied in the context of Schubert varieties. Let $S_n$ be the symmetric group of permutations on $[n]:=\{1,2,\ldots, n\}$. For each $w\in S_n$, let $X_w$ denote the associated Schubert variety, a subvariety of the flag variety defined by rank conditions imposed by $w$. The cohomology class $[X_w]$ is represented by the Schubert polynomial $\mathfrak{S}_w(\mathbf{x})$, introduced by Lascoux and Schützenberger \cite{LascouxSchutzenberger1982}. Schubert polynomials possess remarkable combinatorial properties, most notably a well-known expansion in the monomial basis indexed by pipe dreams \cite{BergeronBilley1993}. In 1992, Fulton \cite{Fulton1992} characterized the ideals of Schubert varieties using determinantal formulas, now called Schubert determinantal ideals. Knutson and Miller \cite{KnutsonMiller2005} demonstrated that Schubert polynomials arise as the multidegrees of matrix Schubert varieties. They utilized Gr\"{o}bner theory to provide a geometric justification for combinatorial pipe dreams.

The success of combinatorial commutative algebra in the Schubert case, particularly the study of Schubert determinantal ideals, motivates the use of commutative algebra to study other subvarieties of the flag variety. A recent survey on Schubert geometry and the related commutative algebra is~\cite{WooYong2023}. Similar techniques were recently used by the second author and Goldin to investigate flat degenerations of Hessenberg varieties~\cite{GoldinPrecup25}, which include Springer fibers as a special case.

In this paper, we set out to develop the combinatorial commutative algebra tools to study components of two-row Springer fibers and their cohomology classes.   
Each component of a two-row Springer fiber is indexed by a unique two-row standard tableau $\gs$, or equivalently, a unique noncrossing matching $\cupsigma$. We focus on two-row Springer fibers as they have a concrete geometric description in terms of noncrossing matchings~\cite{Fung2003, StroppelWebster2012}.

A noncrossing matching is a planar diagram with vertex set $[n]$ consisting of a collection of noncrossing edges (called cups) such that no unmatched vertex lies below a cup and such that no vertex is incident to more than one edge. We say that a noncrossing matching is standard if there are no unmatched vertices.  In this paper, we place rays at each unmatched vertex and draw all cups so that they pass above the vertex set (not below); see Figure~\ref{fig:introex}.

\begin{figure*}[ht]
\centering
    \begin{subfigure}{0.32\textwidth}
        \centering
        \begin{tikzpicture}[scale=0.5, pnt/.style={circle, fill, inner sep=1.5pt}]
            % Define nodes 1 through 6
            \foreach \i in {1,...,6} {
                \node[pnt] at (\i,0) (\i) {};
                \node at (\i,-1) {\small \i};
            }
            % Reference nodes using (number)
            \draw (1) to [bend left=80] (6);
            \draw (2) to [bend left=80] (3);
            \draw (4) to [bend left=80] (5);
        \end{tikzpicture}
    \end{subfigure}
        \hfill
    \begin{subfigure}{0.32\textwidth}
        \centering
        \begin{tikzpicture}[scale=0.5, pnt/.style={circle, fill, inner sep=1.5pt}]
            \foreach \i in {1,...,6} {
                \node[pnt] at (\i,0) (\i) {};
                \node at (\i,-1) {\small \i};
            }
            \draw (1) to [bend left=80] (6);
            \draw (3) to [bend left=80] (4);
            \draw (2) to [bend left=80] (5);
        \end{tikzpicture}
    \end{subfigure}
    \hfill
    \begin{subfigure}{0.32\textwidth}
        \centering
        \begin{tikzpicture}[scale=0.5, pnt/.style={circle, fill, inner sep=1.5pt}]
            \foreach \i in {1,...,6} {
                \node[pnt] at (\i,0) (\i) {};
                \node at (\i,-1) {\small \i};
            }
            \draw (1) to [bend left=80] (2);
            \draw (4) to [bend left=80] (5);
            \draw (3) -- (3,1.5); % Vertical ray
            \draw (6) -- (6,1.5); % Vertical ray
        \end{tikzpicture}

    \end{subfigure}
    \caption{Three noncrossing matchings with vertex set $[6]$; the left and middle matchings are standard, while the matching on the right is not.}
\label{fig:introex}
\end{figure*}
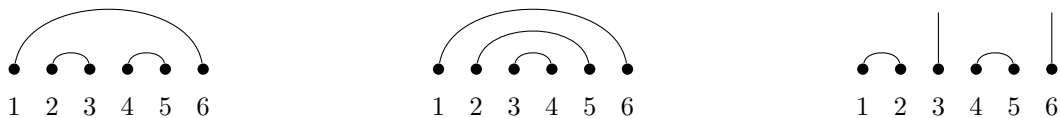

The matching $\cupsigma$ associated to a two-row standard tableau $\gs$ is standard if and only if $\gs$ is rectangular.  To every standard noncrossing matching $\cupsigma$ on $[n]$ we attach a nesting sequence $S_\gs$ that counts, from left to right, the number of cups that have begun in our matching so far. That is, for each $1\leq i \leq n-1$, $S_\gs(i)$ is the number of cups with left end point less than or equal to $i$. For example, the nesting sequence for the leftmost standard matching $\cupsigma$ in Figure~\ref{fig:introex} is $(1,2,2,3,3)$; the nesting sequence for the middle standard matching in Figure~\ref{fig:introex} is $(1,2,3,3,3)$. Using the combinatorial data of the nesting sequence of a standard noncrossing matching we define, in Section~\ref{sec:IdealIsig} below, a polynomial ideal $\ci_\gs$ in the ring $\C[M_n]$ of functions on the affine variety $M_n$ of $n\times n$ matrices. Applying the decomposition of Springer fibers induced by the prime decomposition of tableaux, we extend the definition to non-rectangular two-row standard tableau $\gs$ (that is, we extend the definition to noncrossing matchings containing rays).

Defining the ideal $\ci_\gs$ is our first step to establish a commutative algebra framework for two-row Springer fibers.   Our main result is that $\isigma$ serves as a Springer-theoretic analogue of a Schubert determinantal ideal:

\begin{customthm}{1} \label{thm.intro}
Let $\pi: GL_n(\C) \to GL_n(\C)/B\simeq \Fl_n(\C)$ denote the projection map and let $\ell$ be a positive integer with $1\leq \ell \leq \lfloor \frac{n}{2}\rfloor$. Let  $\gs$ be a standard tableau of shape $(n-\ell, \ell)$ and $\cb_\gs$ be the irreducible component of $\SF_N$ corresponding to $\gs$. The ideal $\ci_\gs\subseteq \C[\mathbf{z}]$ is right $B$-invariant with vanishing set $\cv(\ci_\gs)$ such that $\pi^{-1}(\cb_\sigma) = \cv(\ci_\gs)\cap GL_n(\C)$. 
\end{customthm}

The problem of computing the cohomology class of an irreducible component of a Springer fiber was originally posed by Springer \cite{Springer1983}.  G\"uemes gave a formula for components of hook-type Springer fibers while Graham and Zierau \cite{Graham2011} gave a formula for components defined using the geometry of $GL_p(\C)\times GL_q(\C)$-orbits on the flag variety. More recently, Karp and the second author~\cite{KarpPrecup25} characterized irreducible components of Springer fibers equal to Richardson varieties, which subsume both of these special cases. Spink and Tewari then gave a combinatorial formula for the cohomology class of these ``Richardson components''~\cite{SpinkTewari25}.

Leveraging the commutative algebra techniques described in~\cite{KnutsonMiller2005, ITW18, MillerSturmfels2005}, our main theorem provides a method to compute the cohomology class $[\cb_\gs]$ for any two-row standard tableau~$\gs$. We generate various examples which suggest two conjectures. Although we state these conjectures below for rectangular two-row tableaux, we discuss in Section~\ref{sec:cohomology} how they can be extended to the general setting.

In order to state our first conjecture we introduce a little more notation. Given a cup $(i<j)$ in $\cupsigma$ we define the size $\gd(i<j)$ to be the number of cups nested within $(i<j)$, including itself. For example, the leftmost matching in Figure~\ref{fig:introex} has three cups $(1<6),(2<3),$ and $(4<5)$ with sizes $\gd(1<6)=3$ and $\gd(2<3)=\gd(4<5)=1$. 

If $\gs$ is a two-row rectangular tableau, $\SF_\gs$ is a Richardson variety only when the matching $\cupsigma$ is a union of cups of size one (see~\cite[Remark 3.2(ii)]{KarpPrecup25}). In particular, the components we consider  are Richardson varieties only when the noncrossing matching $\cc(\gs)$ is a union of rays and size one cups, like the rightmost matching in appearing in Figure~\ref{fig:introex}.

We can now state our first conjecture, which computes the class of $\cb_\gs$ as a sum of monomials. The monomials that appear are determined combinatorially by the nesting sequence $S_\gs$ and cups with size at least $2$ in $\cupsigma$. 
\begin{customconj}{1}\label{conj1.intro}
Let $\sigma$ be a standard tableau of shape $(\ell, \ell)$. The cohomology class of $\SF_\gs$ in $H^*(\Fl_n(\C);\Z)$ is represented by the polynomial
\[
\prod_{\substack{j\in [n]\\S_\gs(j)<\ell}} x_j^{2(\ell-S_\gs(j))}  \cdot  \prod_{\substack{(i<j)\\ \gd(i<j)\geq 2}} \left(  \sum_{k=i}^{j-1} x_ix_{i+1}\cdots \hat{x}_k \cdots x_{j-1} \right)  
\]
where the product is taken over all cups $(i<j)$ with size at least $2$ in $\cupsigma$.
\end{customconj}

The Schubert classes $[X_w]$ for $w\in S_n$ form a basis for the cohomology of the flag variety. Consequently, our motivating problem can be rephrased as follows: Compute the expansion of the class $[Y]$ for a subvariety $Y \subseteq \Fl_n(\C)$ in the Schubert basis. Our second conjecture computes the Schubert expansion of $[\SF_\gs]$ using a sequence of divided difference operators; in the statement below $w_0:=[n,n-1, \ldots, 2, 1]$ is the longest permutation.

\begin{customconj}{2} \label{conj2.intro} Let $\gs$ be a standard tableau of shape $(\ell,\ell)$ with corresponding cup diagram $\cupsigma$. Let $(i_1<j_1), (i_2<j_2), \ldots, (i_\ell<j_\ell)$ be any total ordering of the cups in $\cupsigma$ such that when a cup appears, every cup nested inside of that one has already appeared. 

Then 
\[
[\SF_\gs] = \left( \partial_{i_1} + \cdots + \partial_{j_1-1} \right)\circ \left( \partial_{i_2} + \cdots + \partial_{j_2-1} \right)\circ \cdots \circ \left( \partial_{i_\ell} + \cdots + \partial_{j_\ell-1} \right)(\mathfrak{S}_{w_0}).
\]
\end{customconj}

Conjectures~\ref{conj1.intro} and~\ref{conj2.intro} imply that class $[\SF_\gs]$ is encoded by the combinatorics of the associated matching. Figure~\ref{fig:intro2} displays the matchings and computes the cohomology class for both irreducible components of the Springer fiber $\SF_N$ for $N$ of Jordan type $(2,2)$. The reader can find more examples in Section~\ref{sec:cohomology} below. 

\begin{figure*}[t]
\centering
\renewcommand{\arraystretch}{2.5}
\setlength{\tabcolsep}{12pt} 

\begin{tabularx}{\textwidth}{c c l X}
\toprule
$\sigma$ & $\cc(\gs)$ & Monomial expansion of $[\SF_\sigma]$ & Schubert expansion of $[\SF_\sigma]$ \\ 
\midrule

% Row 1: 1 3 / 2 4
\begin{ytableau} 1&3 \\ 2&4 \end{ytableau} & 
\begin{tikzpicture}[scale=0.5, pnt/.style={circle, fill, inner sep=1.5pt}, baseline=(current bounding box.center)]
    \foreach \i in {1,...,4} {
        \node[pnt] at (\i,0) (\i) {};
        \node at (\i,-0.7) {\small \i};
    }
    \draw (1) to [bend left=80] (2);
    \draw (3) to [bend left=80] (4);
\end{tikzpicture} & 
$x_1^2x_2^2$ & 
$\mathfrak{S}_{3412} = (\partial_1+\partial_3) \mathfrak{S}_{4321}$ \\ 

% Row 2: 1 2 / 3 4
\begin{ytableau} 1&2 \\ 3&4 \end{ytableau} & 
\begin{tikzpicture}[scale=0.5, pnt/.style={circle, fill, inner sep=1.5pt}, baseline=(current bounding box.center)]
    \foreach \i in {1,...,4} {
        \node[pnt] at (\i,0) (\i) {};
        \node at (\i,-0.7) {\small \i};
    }
    \draw (1) to [bend left=80] (4);
    \draw (2) to [bend left=80] (3);
\end{tikzpicture} & 
$x_1^2(x_1x_2+x_1x_3+x_2x_3)$ & 
$\mathfrak{S}_{4132} + \mathfrak{S}_{3241} = \partial_2\circ (\partial_1+\partial _2+\partial_3)\mathfrak{S}_{4321}$ \\ 
\bottomrule
\end{tabularx}
\caption{Cohomology classes for each irreducible component of the Springer fiber $\SF_N$, where $N$ is a nilpotent matrix of Jordan type $(2,2)$.}
\label{fig:intro2}
\end{figure*}

The leftmost noncrossing matching in Figure~\ref{fig:introex} is part of a collection that we call matchings with \emph{one big cup}. As indicated by the name, these are precisely the matchings with a cup $(1<n)$ and all other cups of size $1$, or unions thereof. Both matchings in Figure~\ref{fig:intro2} belong to this collection. The middle matching appearing in Figure~\ref{fig:introex} does not belong. We prove the validity of our conjectures for diagrams with one big cup:

\begin{customprop}{1} Conjectures~\ref{conj1.intro} and~\ref{conj2.intro} both hold when the matching $\cupsigma$ has one big cup. 
\end{customprop}

Although computational data supports the validity of both conjectures, proving them using combinatorial commutative algebra presents a significant obstacle: in general, the generating set of the ideal $\ci_\gs$ presented here is not a Gr\"obner basis.

While noncrossing matchings are purely combinatorial objects, they encode representations of the quantum group $U_q(\mathfrak{sl}_2)$~\cite{CautisKamnitzerMorrison2014}. Kuperberg formally defined these matchings as a diagrammatic basis for the $\mathfrak{sl}_2$-web space, providing a framework to describe morphisms between tensor products of representations~\cite{TemperleyLieb1971,Kuperberg1996}. More broadly, the theory of $\mathfrak{sl}_n$-webs sits at the intersection of knot theory, representation theory, and combinatorics.

For two-row Springer fibers, the combinatorics of matchings simplifies the task of relating algebraic structures to the geometry of the flag variety. Extending these results to $\mathfrak{sl}_n$-webs for $n \geq 3$ introduces substantial complexity; however, recent progress by Cummings~\cite{Cummings2026} addresses this problem for two-column tableaux. It remains a question for future research to determine the extent to which $\mathfrak{sl}_n$-webs capture the underlying geometry of three-row Springer fibers and more general cases.

This paper is structured as follows. In Section~\ref{sec:notback} we review noncrossing matchings and two-row standard tableaux, establishing the bijection between them. We then define the prime decomposition of a standard tableau and demonstrate how it corresponds to the decomposition of its associated matching. Finally, we summarize key results regarding the geometry of Springer fibers in the two-row case.  Section ~\ref{sec:PreLemms} proves lemmas we use later: we introduce statistics to study standard noncrossing matchings, study rank conditions imposed by certain determinantal equations, and define an explicit map to decompose components of two-row Springer fibers based on the prime decomposition of their associated tableaux. In Section \ref{sec:IdealIsig} we introduce the ideal $\ci_\gs$, provide numerous examples, and state our main result, Theorem~\ref{thm.intro}. 
The proof of Theorem~\ref{thm.intro} is divided into two parts: In Section \ref{sec.binv} we prove the ideal $\ci_\gs$ is right $\mathcal{B}$-invariant and in Section \ref{sec.vanishingset} we study the vanishing set $\cv(\ci_\gs)$. Section~\ref{sec:cohomology} establishes the two conjectures for expressing the classes of the components $\SF_\gs$ and proves both for the one big cup case.

\subsection*{Acknowledgments} The authors are grateful to Sean Griffin, Alexander Woo, and Jay Yang for helpful conversations and also to Zijing Zhuang for aiding Macaulay2 and SageMath computations. Both authors are supported in part by NSF CAREER grant DMS-2237057.

%%%%%%%%%%%%

\section{Notation and Background}\label{sec:notback}
Let $n$ be a positive integer and $[n]:=\{1,2,\ldots, n\}$. A \emph{composition} of $n$ is a sequence $\mu = (\mu_1, \mu_2, \ldots, \mu_k)$ of positive integers such that  $\sum_{i=1}^k \mu_i =n $.   A \emph{partition} of $n$ is a composition $\lambda=(\lambda_1, \lambda_2, \ldots, \lambda_m)$ of $n$ such that $\lambda_i \geq \lambda_{i+1}$ for all $1\leq i\leq m-1$.  We focus our attention on two-row partitions of the form $\lambda = (n-\ell, \ell)$ where $\ell$ is a positive integer such that $\ell\leq \lfloor\frac{n}{2}\rfloor$.

%-----------------

\subsection{Noncrossing matchings and standard two-row tableaux}  

Let $\lambda$ be a partition of $n$. The \emph{Young diagram} of shape $\lambda$ is an array of $n$ left-justified boxes with $\lambda_i$ in the $i$-th row.  A \emph{standard tableau of shape $\lambda$} is a filling $\gs$ of the Young diagram of $\lambda$ with the numbers $1,2,\ldots, n$ so that each appears exactly once and so that entries strictly increase across rows from left to right and down columns from top to bottom. Given a standard tableau $\gs$ of $\lambda$ we set $|\gs|:=\sum_{i}\lambda_i$.  When $\lambda= (n-\ell, \ell)$ we say that $\sigma$ is a \emph{standard two-row tableau} and let $\SYT(n-\ell, \ell)$ denote the set of all two-row tableau of shape $(n-\ell, \ell)$. 

A \emph{noncrossing matching} on $[n]$ is a planar diagram consisting of:
\begin{itemize}
\item $n$ labeled vertices,
\item $\ell$ edges, called \emph{cups}, drawn so that all cups pass above (not below) vertices, and 
\item $n-2\ell$ \emph{rays} directed upward
\end{itemize}
such that each vertex is attached to exactly one cup or ray and no cup or ray crosses any other. We say that a noncrossing matching is \emph{standard} when the diagram contains no rays (in which case $n$ is necessarily even).  The reader see noncrossing matchings in Examples~\ref{ex.33},~\ref{ex.64}, and~\ref{ex.noncon} below.

There is a bijection between standard two-row tableaux and noncrossing matchings; see, for example, \cite[Lemma 7.1]{Fung2003}. 

\begin{lemma}\label{lemma.bijection} Given a standard tableau $\gs$ of shape $(n-\ell, \ell)$ there exists a unique noncrossing matching $\cupsigma$ such that the bottom row of $\gs$ contains all numbers which are the right endpoints of a cup, and the top row of $\gs$ contains all numbers which are left endpoints of a cup or endpoints of a ray. The association $\gs \mapsto \cupsigma$ defines a bijection from the set of standard tableaux of shape $(n-\ell,\ell)$ to the set of noncrossing matchings on $[n]$ with exactly $\ell$ cups and $n-2\ell$ rays.
\end{lemma}

\begin{ex}\label{ex.33} Let $n=6$ and $\ell=3$.
The following tableau $\gs$ of shape $(3,3)$ has associated standard noncrossing matching $\cupsigma$ displayed to its right.
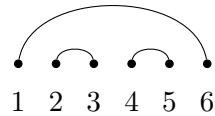
\begin{multicols}{2}
   \begin{center}
        \begin{ytableau}
        1&2&4\\
        3&5&6
        \end{ytableau}
   \end{center} 
\vspace*{-.3in}
\begin{tikzpicture}[scale=0.5]
\begin{scope}[shift={(8,0)}]
	\foreach \i in {0,...,5}{\node[pnt] at (\i,0)(\i){};}
        
	    \draw (0)  to [bend left=80](5);
		\draw(1)  to [bend left=80] (2);
        \draw(3)  to [bend left=80] (4);
		\draw (0,-1) node {$1$};
        \draw (1,-1) node {$2$};
        \draw (2,-1) node {$3$};
        \draw (3,-1) node {$4$};
        \draw (4,-1) node {$5$};
        \draw (5,-1) node {$6$};
	\end{scope}
\end{tikzpicture} \qedhere
 \end{multicols}
\vspace*{-.2in}
\end{ex}

\begin{ex}\label{ex.64} Let $n=10$ and $\ell=4$.
The following tableau $\gs$ of shape $(6,4)$ has associated standard noncrossing matching $\cupsigma$ displayed to its right.
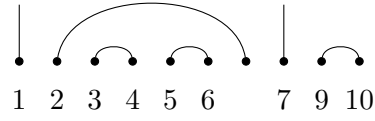
\begin{multicols}{2}
   \begin{center}
        \begin{ytableau}
        1&2&3&5&8&9\\
        4&6&7&10
        \end{ytableau}
   \end{center} 
\vspace*{-.3in}
\begin{tikzpicture}[scale=0.5]
\begin{scope}[shift={(8,0)}]
	\foreach \i in {0,...,9}{\node[pnt] at (\i,0)(\i){};}
           \draw (0,0) -- (0,1.5);
	    \draw (1)  to [bend left=80](6);
		\draw(2)  to [bend left=80] (3);
        \draw(4)  to [bend left=80] (5);
        \draw (7,0) -- (7,1.5);
        \draw (8)  to [bend left=80](9);
		\draw (0,-1) node {$1$};
        \draw (1,-1) node {$2$};
        \draw (2,-1) node {$3$};
        \draw (3,-1) node {$4$};
        \draw (4,-1) node {$5$};
        \draw (5,-1) node {$6$};
        \draw (7,-1) node {$7$};
        \draw (8,-1) node {$9$};
        \draw (9,-1) node {$10$};
	\end{scope}
\end{tikzpicture} \qedhere
 \end{multicols}
\vspace*{-.2in}
\end{ex}

\begin{ex}\label{ex.noncon} Let $n=14$ and $\ell=7$.
The following tableau $\gs$ of shape $(7,7)$ has associated standard noncrossing matching $\cupsigma$ displayed to its right.
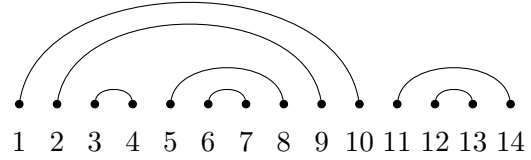
\begin{multicols}{2}
   \begin{center}
        \begin{ytableau}
        1&2&3&5&6&11&12\\
        4&7&8&9&10&13&14
        \end{ytableau}
   \end{center} 
\vspace*{-.3in}
\begin{tikzpicture}[scale=0.5]
\begin{scope}[shift={(8,0)}]
	\foreach \i in {0,...,13}{\node[pnt] at (\i,0)(\i){};}

	    \draw (0)  to [bend left=80](9);
		\draw(1)  to [bend left=80] (8);
        \draw(2)  to [bend left=80] (3);
         \draw(4)  to [bend left=80] (7);
         \draw(5)  to [bend left=80] (6);
            \draw(10)  to [bend left=80] (13);
             \draw(11)  to [bend left=80] (12);
		\draw (0,-1) node {$1$};
   
        \draw (1,-1) node {$2$};
     
        \draw (2,-1) node {$3$};
        \draw (3,-1) node {$4$};
   
        \draw (4,-1) node {$5$};
        \draw (5,-1) node {$6$};
         \draw (6,-1) node {$7$};
          \draw (7,-1) node {$8$};
         \draw (8,-1) node {$9$};
     \draw (9,-1) node {$10$};
     \draw (10,-1) node {$11$};
     \draw (11,-1) node {$12$};
     \draw (12,-1) node {$13$};
     \draw (13,-1) node {$14$};
	\end{scope}
\end{tikzpicture} \qedhere
 \end{multicols}
\vspace*{-.2in}
\end{ex}

We say that a noncrossing matching on $[n]$ is \emph{connected} if either $n=1$ and the diagram consists of a single ray or if $n>1$ and the pair $(1<n)$ is a cup. For example, the matching appearing in Example~\ref{ex.33} is connected, while those in Examples~\ref{ex.64} and~\ref{ex.noncon} are not. We call noncrossing matchings consisting of a single ray \emph{trivial matchings}. For example, the matching appearing in Example~\ref{ex.64} is the union of four connected components, two of which are trivial. It's clear that a connected, nontrivial noncrossing matching must be standard, but Example~\ref{ex.noncon} shows that not every standard noncrossing matching is connected. 

We conclude this section by defining some terminology for cups in a matching.  Given a (standard) noncrossing matching $\cc$ on $n$ vertices and a cup $(i<j)$ in $\cc$ we define 
\[
\delta(i<j):=\frac{1}{2}(j-i+1)
\]
to be the \emph{size of the cup $(i<j)$}, that is, $\delta(i<j)$ is the number of cups nested inside the cup connecting $i$ and $j$. We also set 
 \[
 \ga(i<j):=\#\{(s<t) \mid \text{ $(s<t) \ $  is a cup and $\ s<i<j<t$ }\}.
 \]
In other words, $\ga(i<j)$ is the number of cups that lie strictly above $(i<j)$ in $\cupsigma$.  For example, the cup $(5<8)$ appearing in the matching from Example~\ref{ex.noncon} has size $\delta(5<8)=2$ and there are $\ga(5<8)=2$ cups which lie strictly above it, namely $(2<9)$ and $(1<10)$.

%-----------------

\subsection{Connected noncrossing matchings and prime tableaux}
Let $\gs$ and $\omega$ be standard tableaux. We define the \emph{concatenation} of $\gs$ and $\omega$ to be the tableau obtained by concatenating the rows of $\gs$ and $\omega$ after adding $|\gs|$ to every entry of $\omega$. Concatenation is an associative operation. We say that the standard tableau $\gs$ is a \emph{prime} if it is nonempty and if it is not equal to the concatenation of two nonempty tableau. Every standard tableau has a unique prime decomposition \cite[Proposition 1.3.3.]{Tirrell2016}  

It is simple to determine if a two-row standard tableau $\gs$ is prime using the associated noncrossing matching $\cupsigma$.  

\begin{lemma}\label{lemma.primetoconn} Let $\gs$ be a standard tableau of shape $(n-\ell, \ell)$ and $\cupsigma$ be the associated noncrossing matching. The tableau $\gs$ is prime if and only if $\cupsigma$ is connected. 
\end{lemma}
\begin{proof} This follows immediately using the bijection $\gs \mapsto \cupsigma$ of Lemma~\ref{lemma.bijection}.
\end{proof}

Since nontrivial, connected noncrossing matchings are all standard, it follows that the only two-row prime tableaux are \emph{rectangular}, that is, of shape $(\ell, \ell)$.  Since every noncrossing matching is a disjoint union of connected noncrossing matchings  Lemma~\ref{lemma.primetoconn} immediately implies: 

\begin{corollary}\label{cupdecomposition} Every two-row tableau of shape $(n-\ell,\ell)$ is the concatenation of prime two-row rectangular tableau and one-row prime tableaux consisting of a single box.
\end{corollary}

\begin{ex} Consider the two-row tableau $\gs$ below.
\[
\begin{ytableau} 1&2&3&4 & 8 & 9 & 10 & 12 \\ 5&6&7&11&13& 14 \end{ytableau}
\]
We determine the prime decomposition of $\gs$ easily using the associated matching $\cupsigma$, displayed below.
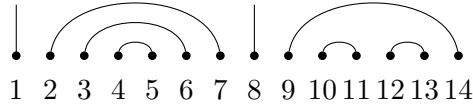
\begin{figure*}[h]
\begin{tikzpicture}[scale=0.45]\begin{scope}[shift={(8,0)}]
		\foreach \i in {0,...,13}
		{
			\node[pnt] at (\i,0)(\i){};
		}
		
        \draw (0,0) -- (0,1.5);
        \draw(1)  to [bend left=80] (6);
		\draw(2)  to [bend left=80] (5);
		\draw(3)  to [bend left=80] (4);
		\draw (8)  to [bend left=80](13);
		\draw(9)  to [bend left=80] (10);
        \draw(11)  to [bend left=80] (12);
        \draw (7,0) -- (7,1.5);
		\draw (0,-1) node {$1$};
        \draw (1,-1) node {$2$};
        \draw (2,-1) node {$3$};
        \draw (3,-1) node {$4$};
        \draw (4,-1) node {$5$};
        \draw (5,-1) node {$6$};
        \draw (6,-1) node {$7$};
        \draw (7,-1) node {$8$};
        \draw (8,-1) node {$9$};
        \draw (9,-1) node {$10$};
        \draw (10,-1) node {$11$};
        \draw (11,-1) node {$12$};
        \draw (12,-1) node {$13$};
        \draw (13,-1) node {$14$};
	\end{scope}
 \end{tikzpicture} 
\end{figure*}

\noindent Since $\cupsigma$ is the union of two rays and two connected standard noncrossing matchings on $6$ vertices, we obtain a corresponding decomposition of $\gs$ as follows. 
\[
\ytableausetup{centertableaux} \begin{ytableau} 1\end{ytableau} \circ 
\begin{ytableau} 1&2&3  \\ 4&5&6 \end{ytableau} \circ
\begin{ytableau} 1 \end{ytableau} \circ
\begin{ytableau}  1 & 2 & 4 \\ 3&5& 6 \end{ytableau} \qedhere
\]
\end{ex}

%-----------------

\subsection{Springer fibers and two-row Springer fibers} \label{Springer} 
In this section we give a brief introduction to Springer fibers, highlighting results needed for our work. For a more detailed review, we refer the reader to \cite{TymoczkoSF,Jantzen2004}.

A \emph{complete flag in $\C^n$} is a nested sequences of subspaces of $\C^n$, 
\[
V_\bullet = (V_1\subset V_2\subset \cdots \subset V_n=\mathbb{C}^n),
\]

such that $\dim V_i=i$ for $1\leq i \leq n$. The collection of all complete flags is the \emph{flag variety}, denoted here by $\Fl_n(\C)$. It is a smooth projective variety of dimension ${n\choose 2}$. We identify $\mathcal{F}\ell_n(\C)$ with the homogeneous space $GL_n(\C)/B$ where $B$ is the subgroup of $GL_n(\C)$ of upper triangular matrices. Specifically, given a matrix $g\in GL_n(\C)$ we obtain a flag $V_\bullet$ by setting $V_i$ equal to the span of the first $i$ columns of $g$.

Let $N$ be a  $n\times n$ nilpotent complex matrix. The \emph{Springer fiber} $\cb_N$ is defined to be
\[
\cb_N:=\{V_\bullet\in \Fl_n(\C)\mid N\left(V_i\right)\subseteq V_{i-1}~\text{for all}~ 1\leq i\leq n\}.
\]
The irreducible components of the Springer fiber $\mathcal{B}_N$ are in bijection with standard Young tableaux corresponding to the partition type of $N$ \cite{Spaltenstein1976,Vargas79}. Given a standard tableau $\gs$, we let $\SF_\gs$ denote the corresponding irreducible component of $\SF_N$ under this bijection.

Fresse showed each component $\SF_\gs$ of the Springer fiber respects concatenation of tableau $\gs$ in the following sense~\cite[Proposition 7.1]{Fresse2011}

\begin{prop}[Fresse] \label{prop.Fresse} Given standard tableaux $\sigma$ and $\omega$, recall that $\sigma\circ \omega$ denotes the standard tableau obtained by concatenation. Then $\SF_{\sigma\circ \omega} \simeq \SF_\sigma \times \SF_\omega$. 
\end{prop}

This result reduces the study of the geometry of the component $\SF_{\sigma \circ \omega}$ to that of $\SF_\sigma$ and $\SF_\omega$. We use it later to reduce our arguments to the case where $\gs$ is a prime two-row tableau, that is, $\cupsigma$ is a standard matching.

When $N$ has two Jordan blocks we say that $\SF_N$ is a \emph{two-row Springer fiber}. Given a tableau $\gs\in \SYT(n-\ell, \ell)$, Fung~\cite[Theorem 5.2]{Fung2003} gave a simple description of each component $\SF_\gs$ using the noncrossing matching $\cupsigma$. To state the result, we introduce some terminology. If $i$ is the endpoint of a ray, we let 
\[
c(i) = \#\left\{ \text{cups } (i'<j') \text{ in }\cupsigma\mid j'<i \right\}
\]
denote the number of cups that appear before the ray. Similarly, we let $r(i)$ be the number of rays that have appeared up through index $i$, including the ray with endpoint $i$. Notice that $i=r(i)+c(i)$.  The following is Fung's description of the irreducible component $\SF_\gs$~\cite{Fung2003} (statement adapted from~\cite[Proposition 7]{StroppelWebster2012}).

\begin{prop}[Fung]\label{prop.Fung} Let $N$ be a nilpotent matrix with Jordan type $(n-\ell, \ell)$ and $\gs\in \SYT(n-\ell, \ell)$.  Let $V_\bullet$ be a flag in the Springer fiber $\SF_N$. Then $V_\bullet\in \SF_\gs$ if and only if    
\begin{enumerate}
\item[\textup{(1)}] for each cup $(i<j)$ in $\cc(\sigma)$,  $N^{\delta(i<j)}\left(V_{j}\right)=V_{i-1}$ and
\item[\textup{(2)}] for each ray with end point $i$, $V_i=N^{-c(i)}\left(\im N^{n-\ell-r(i)}\right)$. 
\end{enumerate}
\end{prop}

\begin{remark}\label{rem.ray-condition} Suppose $N$ is a nilpotent matrix of Jordan type $\lambda = (n-\ell, \ell)$ and suppose $N$ is in Jordan canonical form with block sizes decreasing. That is, on the standard basis $\{e_1, e_2, \ldots, e_n\}$ of $\C^n$ the matrix $N$ acts by
\[
e_{n-\ell} \mapsto e_{n-\ell-1}\mapsto \cdots \mapsto e_1\mapsto 0 \quad \text{ and } \quad e_n \mapsto e_{n-1}\mapsto \cdots \mapsto e_{n-\ell+1} \mapsto 0.
\]
Consider the condition of Proposition~\ref{prop.Fung}(2) with respect to this choice of basis. Note that there are precisely $n-2\ell$ rays in $\cupsigma$ and thus $r(i)\leq n-2\ell$ so $n-\ell-r(i) \geq \ell$. In particular, we get that 
\[
\im N^{n-\ell-r(i)} = \C\{e_1, e_2,\ldots, e_{r(i)}\}
\]
and 
\[
N^{-c(i)}\left(\im N^{n-\ell-r(i)}\right) = \C\{e_1, e_2, \ldots, e_{r(i)+c(i)}\} \oplus\C\{ e_{n-\ell+1}, \ldots, e_{n-\ell +c(i)}\}.
\]
This is a subspace of dimension $i = r(i)+2c(i)$. Thus, Proposition~\ref{prop.Fung}(2) is precisely the condition that $V_i$ is the coordinate subspace spanned by the standard basis vectors $e_1, e_2, \ldots, e_{r(i)+c(i)}$ and $e_{n-\ell +1}, \ldots, e_{n-\ell + c(i)}$. 
\end{remark}

\begin{ex}
Consider the standard Young tableau $\gs$ and associated standard noncrossing matching $\cupsigma$ from Example~\ref{ex.33}. Let $N$ be a $6\times 6$ nilpotent matrix of Jordan type $(3,3)$. 

The matching $\cupsigma$ has three cups $(1<6)$, $ (2<3)$ and $(4<5)$ with sizes $\gd(1<6)=3$ and $\gd(2<3)=\gd(4<5)=1$. By Proposition~\ref{prop.Fung}, the complete flag $V_\bullet\in \SF_N$ lies in $\cb_\gs$ if and only if
\[
N^3\left(V_6\right)=\{0\}, \;  N\left(V_3\right)=V_1, \;  \text{ and }\; N\left(V_5\right)=V_3.
\]
Note that the first of these conditions holds vacuously since $N^3 = 0$. 
\end{ex}

\begin{ex}\label{ex.component.disconn} Let $N$ be a $6\times 6$ nilpotent matrix of Jordan type $(3,3)$ and consider the tableaux
\[
\begin{ytableau} 1&3&4&6\\ 2&5
\end{ytableau} \quad\quad\quad\quad \;
\begin{ytableau} 1 & 2 & 3& 6\\ 4 & 5
\end{ytableau}
\]
which correspond, respectively, the following matchings.\\

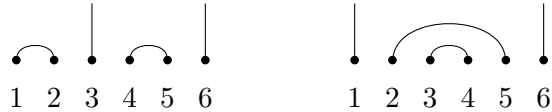
\begin{center}
\begin{tikzpicture}[scale=0.5]
	\begin{scope}[shift={(8,0)}]
		\foreach \i in {0,...,5}
		{
			\node[pnt] at (\i,0)(\i){};
		}

		\draw (0)  to [bend left=80](1);
     
		\draw(3)  to [bend left=80] (4);
          \draw (2,0) -- (2,1.5);
            \draw (5,0) -- (5,1.5);
		\draw (0,-1) node {$1$};
       
        \draw (1,-1) node {$2$};

        \draw (2,-1) node {$3$};
        \draw (3,-1) node {$4$};
        \draw (4,-1) node {$5$};
        \draw (5,-1) node {$6$};
	\end{scope}
\end{tikzpicture} \hspace*{.5in}
\begin{tikzpicture}[scale=0.5]
	\begin{scope}[shift={(8,0)}]
		\foreach \i in {0,...,5}
		{
			\node[pnt] at (\i,0)(\i){};
		}

		\draw (1)  to [bend left=80](4);
		\draw(2)  to [bend left=80] (3);
          \draw (0,0) -- (0,1.5);
            \draw (5,0) -- (5,1.5);
		\draw (0,-1) node {$1$};
       
        \draw (1,-1) node {$2$};
  
        \draw (2,-1) node {$3$};
      
         \draw (3,-1) node {$4$};
       
        \draw (4,-1) node {$5$};
        \draw (5,-1) node {$6$};
	\end{scope}
    \end{tikzpicture}
\end{center}

Consider the tableau/matching on the left first. There are two cups in $\cupsigma$, $(1<2)$ and $(4<5)$, and rays at $3$ and $6$. We have $\gd(1<2)=\gd(4<5)=1$, $c(3)=1$ and $c(6)=2$, and $r(3)=1$ and $r(6)=2$. Thus, $V_\bullet\in \SF_N$ is in the component $\SF_\gs$ if and only if 
\[
N(V_2) = \{0\},\; V_3 = N^{-1}\left(\im N^{3} \right),\; N(V_5) = V_3,\; \text{ and } \; V_6 = N^{-2}\left(\im N^{2} \right),
\]
where the last condition is the vacuous condition that $V_6 = \C^6$. 

Now consider the tableau/matching on the right. There are two cups in $\cupsigma$, $(2<5)$ and $(3<4)$, and rays at $1$ and $6$. We have $\gd(2<5)=2$ and $\gd(3,4)=1$, $c(1)=0$, and $c(6)=2$, and $r(1)=1$ and $r(6)=2$. Thus, $V_\bullet\in \SF_N$ is in the component $\SF_\gs$ if and only if 
\[
V_1 = \im N^{3},\; N(V_4)=V_2,\; N^2(V_5) = V_1,\; \text{ and } \; V_6 = N^{-2}\left(\im N^{2} \right),
\]
and as before, the last condition is the vacuous condition that $V_6 = \C^6$. 
\end{ex}

From here on, we fix a basis $\{e_1, e_2, \ldots, e_n\}$ of $\C^n$ so that $N$ is in Jordan canonical form. That is, $N$ acts on the standard basis $\{e_1, e_2, \ldots, e_n\}$ by
\begin{eqnarray}\label{eqn.Jbasis}
e_{n-\ell} \mapsto e_{n-\ell-1}\mapsto \cdots \mapsto e_1\mapsto 0 \quad \text{ and } \quad e_n \mapsto e_{n-1}\mapsto \cdots \mapsto e_{n-\ell+1} \mapsto 0.
\end{eqnarray}
Given a standard tableaux of shape $(n-\ell, \ell)$ and the associated noncrossing matching $\cupsigma$ we now identify an open set of flags, denoted below by $\cell$, such that $\cb_\sigma = \overline{\cell}$.

\begin{definition}\label{def:cell}
Let $\cupsigma$ be a standard noncrossing matching in which each cup $(i<j)$ is labeled by $a_i$. For each $i\in [n]$ we write
\[
c(i):= \#\{ \text{ cups }(i'<j')\text{ in $\cupsigma$} \mid j'\leq i\}
\]
to denote the number of complete cups that appear at or before vertex $i$ and 
\[
r(i) := \#\{ \text{ rays in $\cupsigma$ with endpoint $i'$} \mid i' \leq i \}
\]
for the number of rays that appear at or before vertex $i$.
We construct an affine cell $\cell$ in $\mathcal{F}\ell(n)$ by defining matrix representatives of the flags in $\cell$ column-by-column as follows.
\begin{enumerate}
\item  If $j$ is the right endpoint of a cup or endpoint of a ray then the $j$-th column of a matrix representative for a flag in $\cell$ is the standard basis vector $e_{r(j)+c(j)}$.

\item If $i$ be the left end point of a cup $(i<j)$, let $\al = \al(i<j)$ and label the cups above $(i<j)$ from bottom to top by $\left(a_{i_1}, a_{i_2}, \ldots, a_{i_{\ga}}\right)$. Then the $i$-th column of a matrix representative for a flag in $\cell$ is
\[
\left(0^{r(i)+c(i)},a_i,a_{i_1}, \ldots, a_{i_{\ga}}, 0^{n-\ell-r(i)-1}, 1,0^{\ell-c(i)-\ga-1}\right)^T
\]
where $a_{i}, a_{i_1},\ldots, a_{i_{\ga}}\in \C$. Note that $1$ appears in position $n-\ell+1 + c(i)+\ga$.
\end{enumerate}
\end{definition}

From the definition we see that the right end points of cups and the rays in $\cupsigma$ dictate the position of $1$'s in the first $n-\ell$ rows of the matrices in~$\cell$. Left end points of cups dictate positions of the $1$'s in the lower $\ell$ rows of matrices in~$\cell$.

\begin{ex}\label{ex.disconnected1}
Consider the standard Young tableau $\gs$ of shape $(6, 4)$ and its associated noncrossing matching pictured below. Here we have labeled each cup $(i<j)$ by $a_i$. 
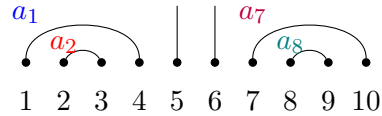
\begin{multicols}{2}
\begin{center} \begin{ytableau}1&2&5&6&7 & 8\\3&4&9 & 10 \end{ytableau}
\end{center} 
\vspace*{-.2in}
\begin{tikzpicture}[scale=0.5] \begin{scope}[shift={(8,0)}]
\foreach \i in {0,...,9}{\node[pnt] at (\i,0)(\i){};}
\draw (0)  to [bend left=90](3);
\draw(1)  to [bend left=60] (2);
\draw (4,0) -- (4,1.5);
\draw (5,0) -- (5,1.5);
\draw (6)  to [bend left=90](9);
\draw(7)  to [bend left=60] (8);
     
\draw (0,-1) node {$1$};
\draw (0,0.5) node[label=\textcolor{blue}{$a_1$}] {};
\draw (1,-1) node {$2$};
\draw (1,-0.3) node[label=\textcolor{red}{$a_2$}] {};
\draw (2,-1) node {$3$};
\draw (3,-1) node {$4$};
\draw (4,-1) node {$5$};
\draw (5,-1) node {$6$};
\draw (6,-1) node {$7$};
\draw (6,0.5) node[label=\textcolor{purple}{$a_7$}] {};
\draw (7,-1) node {$8$};
\draw (7,-0.3) node[label=\textcolor{teal}{$a_8$}] {};
\draw (8,-1) node {$9$};
 \draw (9,-1) node {$10$};
\end{scope}
\end{tikzpicture} 
\end{multicols}
\noindent A generic matrix in $\cell$ is of the form:
\[
\begin{bmatrix} \textcolor{blue}{a_1} & \textcolor{red}{a_2} & 1 & 0 & 0 & 0 & 0 & 0 & 0 & 0 \\ 
0 & \textcolor{blue}{a_1} & 0 & 1 & 0 & 0 & 0 & 0 & 0 & 0\\
0 & 0 & 0 & 0 & 1 & 0 & 0 & 0 & 0 & 0\\
0 & 0 & 0 & 0 & 0 & 1 & 0 & 0 & 0 & 0\\
0 & 0 & 0 & 0 & 0 & 0 & \textcolor{purple}{a_7} & \textcolor{teal}{a_8} & 1 & 0\\
0 & 0 & 0 & 0 & 0 & 0 & 0 & \textcolor{purple}{a_7} & 0 & 1\\
1 & 0 & 0 & 0 & 0 & 0 & 0 & 0 & 0 & 0\\
0 & 1 & 0 & 0 & 0 & 0 & 0 & 0 & 0 & 0\\
0 & 0 & 0 & 0 & 0 & 0 & 1 & 0 & 0 & 0\\
0 & 0 & 0 & 0 & 0 & 0 & 0 & 1 & 0 & 0\\
\end{bmatrix}.
\]
We have, for example, that the $8$-th column is defined using (2) in Definition~\ref{def:cell} since $8$ is the left endpoint of a cup. In this case, $r(8)=2$ and $c(8) = 2$ since there are two rays and two completed cups which appear before the vertex $8$. There is one cup containing $(8<9)$ so $\alpha(8<9)=1$, and $n-\ell - r(8)-1 = 10-4-2-1 = 3$ and $\ell-c(8)-\alpha(8<9)-1 = 4-2-1-1 = 0$. Thus, the $8$-th column is:
\[
(0,0,0,0,a_8, a_7, 0,0,0, 1)^T
\]
as in the matrix above. 
\end{ex}

\begin{ex} \label{ex.n=4nested}
The equations display two standard noncrossing matching $\cupsigma$ and the corresponding generic matrices in $\cell$.
\vspace*{-.2in}
\begin{center}
\begin{tikzpicture}[scale=0.5]
\begin{scope}[shift={(8,0)}]
\foreach \i in {0,...,3}{\node[pnt] at (\i,0)(\i){};}

\draw (0)  to [bend left=80](3) ;
\draw(1)  to [bend left=90] (2);
\draw (0,-1) node {$1$};
\draw (0,0.5) node[label=\textcolor{blue}{$a_1$}] {};
\draw (1,-1) node {$2$};
\draw (1.1,-0.15) node[label=\textcolor{red}{$a_2$}] {};
\draw (2,-1) node {$3$};
\draw (3,-1) node {$4$};
\end{scope}
\end{tikzpicture}
\hspace*{1in}
\begin{tikzpicture}[scale=0.5]
	\begin{scope}[shift={(8,0)}]
		\foreach \i in {0,...,7}
		{
			\node[pnt] at (\i,0)(\i){};
		}

		\draw (0)  to [bend left=80](7) ;
      
		\draw(1)  to [bend left=80] (6);
        \draw(2)  to [bend left=80] (3);
        \draw(4)  to [bend left=80] (5);
		\draw (0,-1) node {$1$};
        \draw (0,0.5) node[label=\textcolor{blue}{$a_1$}] {};
        \draw (1,-1) node {$2$};
        \draw (1.2,0.25) node[label=\textcolor{red}{$a_2$}] {};
        \draw (2,-1) node {$3$};
        \draw (3,-1) node {$4$};
         \draw (2.4,0.05) node[label=\textcolor{violet}{$a_3$}] {};
         \draw (4.35,0.05) node[label=\textcolor{teal}{$a_5$}] {};
        \draw (4,-1) node {$5$};
        \draw (5,-1) node {$6$};
        \draw (6,-1) node {$7$};
        \draw (7,-1) node {$8$};
	\end{scope}
\end{tikzpicture}
\end{center}

\[ 
\begin{bmatrix}
\textcolor{blue}{a_1}& \textcolor{red}{a_2}& 1& 0\\
0& \textcolor{blue}{a_1}& 0& 1 \\
1& 0& 0& 0\\ 0&1&0& 0
\end{bmatrix} 
\quad\quad\quad
\quad\quad\quad
\begin{bmatrix}
\textcolor{blue}{a_1}& \textcolor{red}{a_2}& \textcolor{violet}{a_3}& 1& 0 & 0& 0& 0\\
 0& \textcolor{blue}{a_1}& \textcolor{red}{a_2}& 0&\textcolor{teal}{a_5}& 1& 0& 0 \\
0& 0& \textcolor{blue}{a_1}& 0&\textcolor{red}{a_2}&0 & 1&0\\
0& 0& 0&0& \textcolor{blue}{a_1}& 0&0&1\\
1& 0& 0& 0& 0& 0& 0& 0\\
0& 1&0& 0& 0& 0& 0& 0\\
0& 0& 1& 0& 0& 0& 0&0\\
0& 0& 0& 0&1&0& 0&0
\end{bmatrix}. \qedhere\]  
\end{ex}

\begin{ex} The following are, respectively, generic matrices in each of the cells $\cell$ for the tableaux and cup diagrams from Example~\ref{ex.component.disconn}.   
\[\begin{bmatrix}
a_1&1&0&0&0&0\\
0&0&1&0&0&0\\  
0&0&0&a_4&1&0\\ 
0&0&0&0&0&1\\ 
1&0&0&0&0&0\\ 
0&0&0&1&0&0
\end{bmatrix} \quad\quad\quad \begin{bmatrix}
1&0&0&0&0&0\\
0&a_2&a_3&1&0&0\\  
0&0&a_2&0&1&0\\ 
0&0&0&0&0&1\\ 
0&1&0&0&0&0\\ 
0&0&1&0&0&0
\end{bmatrix} \qedhere
\]
\end{ex}

Goldwasser--Nadeem--Sun--Tymoczko showed in \cite[Theorem 4.10]{GNGT2025} that the set of flags represented by matrices in $\cell$ are in the Springer fiber $\cb_N$ and their closure is precisely the irreducible component indexed by $\gs$.

\begin{theorem}\label{thm.GNST}
Let $N$ be a nilpotent matrix of Jordan type $(n-\ell,\ell)$ and assume $N$ is in Jordan canonical form. For each $\gs\in \SYT(n-\ell, \ell)$ the cell $\cell$ is in the Springer fiber $\cb_N$. Furthermore, $\overline{F_\sigma}$ is the irreducible component $\SF_\gs$ of $\cb_N$ corresponding to the tableau $\sigma$.  
\end{theorem}

\begin{remark}
Our notation for $\cell$ from Definition~\ref{def:cell} differs from that used in~\cite{GNGT2025}. To translate between the two, we note the matrix representatives of $\cell$ defined above comprise the image of the map from \cite[Definition 4.3]{GNGT2025}, denoted there by  $f_\mathcal{M}: \C^\ell \to M_\ell,\; v \mapsto w_\mathcal{M}+\Im({v})$, where $\mathcal{M}$ is a noncrossing matching.  In particular, when $\mathcal{M}= \cupsigma$ we have that $v=(a_i)_{(i<j)}$ with coordinates indexed by cups $(i<j)$ in $\cupsigma$, the permutation matrix $w_{\mathcal{M}}$ is precisely the matrix obtained from our description of the elements of $\cell$ by setting $a_i=0$ for all $i$, and $\Im(v)$ is obtained by subtracting $w_\mathcal{M}$ from the matrix in $\cell$ determined by $v$.
\end{remark}

%%%%%%%%%%%%%%%%%%%%%%%

\section{Preliminary Lemmas}\label{sec:PreLemms}

%-----------------

\subsection{Lemmas on standard noncrossing matchings}\label{sec:SYTmatchings}

In this section we introduce statistics to study standard noncrossing matchings and establish identities we will use to prove our results. Throughout this section we set $n=2\ell$ and let $\gs$ be a standard tableau of shape $(\ell, \ell)$ with associated standard noncrossing matching $\cupsigma$.

\begin{definition} For each $i\in [n-1]$ we define $S_\gs(i)\in [\ell]$ by the rule that 
\[ 
S_\gs(i) := \# \{(s<t ) \mid \text{$(s<t)$ is a cup in $\cupsigma$} \}\ .
\]
We call  $S_\gs:= (S_\gs(1), \ldots, S_\gs(n-1))$ the \emph{total cup sequence} of $\cupsigma$.
\end{definition}

\begin{ex}\label{ex.totalseq}
Let $\cupsigma$ be the following standard noncrossing matching.
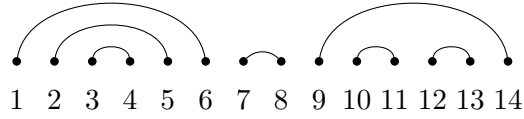
\begin{center}
\begin{tikzpicture}[scale=0.5]\begin{scope}[shift={(8,0)}]
		\foreach \i in {1,...,14}
		{
			\node[pnt] at (\i,0)(\i){};
		}

		\draw(7)  to [bend left=45] (8);
        \draw(1)  to [bend left=80] (6);
		\draw(2)  to [bend left=80] (5);
		\draw(3)  to [bend left=80] (4);
		\draw (9)  to [bend left=80](14);
		\draw(10)  to [bend left=80] (11);
        \draw(12)  to [bend left=80] (13);
	
        \draw (1,-1) node {$1$};
       
        \draw (2,-1) node {$2$};
        \draw (3,-1) node {$3$};
       
        \draw (4,-1) node {$4$};
        \draw (5,-1) node {$5$};
        \draw (6,-1) node {$6$};
        \draw (7,-1) node {$7$};
        \draw (8,-1) node {$8$};
        \draw (9,-1) node {$9$};
        \draw (10,-1) node {$10$};
        \draw (11,-1) node {$11$};
        \draw (12,-1) node {$12$};
        \draw (13,-1) node {$13$};
        \draw (14,-1) node {$14$};
       
	\end{scope}
 \end{tikzpicture}
\end{center} 
The total cup sequence associated to $\cupsigma$ is $S_\gs=\left(1,2,3,3,3,3,4,4,5, 6,6,7,7 \right)$.
\end{ex}

Given a cup $(i<j)$, recall that $\delta(i<j)$ is the size of $(i<j)$ and $\al(i<j)$ denotes the number of cups strictly above $(i<j)$. 

\begin{lemma}\label{lem.length}  Let $(i<j)$ be a cup in the standard noncrossing matching $\cupsigma$. Then 
\[
S_\gs(j)-S_\gs(i)+1 = \delta(i<j)  .
\]
\end{lemma}
\begin{proof} Let $(i<j)$ be a cup in $\cupsigma$. Every vertex $i'$ such that $i<i'<j$ must be matched with a $j'$ such that $i'<j'<j$ since the cups in $\cupsigma$ cannot cross. There are exactly $\frac{1}{2}(j-i-1)$ many cups nested in $(i<j)$. But this is exactly the value of $S_\gs(j)-S_\gs(i)$ and therefore,
\begin{equation*}
    S_\gs(j)-S_\gs(i)+1 =\frac{1}{2}(j-i-1)+1 = \frac{1}{2}(j-i+1)=\gd(i<j)
\end{equation*}
as desired.
\end{proof}

\begin{lemma}\label{lem.Sdiff}
Let $(i^*<j^*)$ be the unique cup directly above the cup $(i<j)$ in $\cupsigma$. Then
\[
S_\gs(i)-S_\gs(i^*)=\frac{1}{2} \left(i-i^*+1\right).
\]
\end{lemma}
\begin{proof}
Let $(i<j)$ and $(i^*<j^*)$ be cups in $\cupsigma$ with $(i^*<j^*)$ the cup directly above $(i<j)$. In other words, we have $i^*<i$, $j<j^*$, and $\ga(i^*<j^*)=\ga(i<j)-1$. By definition, $S_\gs(i)$ is the number of cups with left end point less than or equal to $i$. Thus, 

\begin{equation}\label{eqn.cupdiff}
S_\gs(i)-S_\gs(i^*)   = \#\{\text{cups } (i'<j')\mid i^*< i'\leq i\}.
\end{equation}

Since $(i^*<j^*)$ is the unique cup directly above $(i<j)$, every other cup $(i',j')$ with left end point $i'$ such that $i^*<i'<i$ cannot contain $(i<j)$ and therefore $j'<i$. In particular, $i-i^*-1$ is even and there are precisely $\frac{1}{2}(i-i^*-1)$ many cups $(i'<j')$ such that $i^*<i'<i$. Now~\eqref{eqn.cupdiff} tells us that $S_\gs(i)-S_\gs(i^*) = \frac{1}{2}(i-i^*-1) + 1 = \frac{1}{2}(i-i^*+1)$.
\end{proof}

Recall that $c(i)$ denotes the number of cups $(i',j')$ such that $j'\leq i$, that is, the number of cups that are completed before vertex $i$. 

\begin{lemma}\label{lem.rightend}
For any cup $(i<j)$ in $\cupsigma$. 
\begin{equation}
c(j)=c(i)+\gd(i<j).
\end{equation}
\end{lemma}
\begin{proof}
Let $(i<j)$ be a cup in $\cupsigma$. By definition, $c(i)$ is the number of completed cups in the diagram appearing before vertex $i$. The size $\gd(i<j)$  of the cup $(i<j)$, by definition, counts the number of cups completed between $i$ and $j$. Therefore the number of cups completed before $i$, which is $c(i)$, plus the number of cups completed between $i$ and $j$, which is $\gd(i<j)$, is $c(j)$ the number of cups completed by the vertex $j$.
\end{proof}

\begin{lemma} \label{lemma.completedcups} Let $(i<j)$ be a cup in $\cupsigma$. Then $S_\gs(j)-\delta(i<j)-\alpha(i<j)=c(i)$ .  
\end{lemma}
\begin{proof}
Let $(i<j)$ be a cup in $\cupsigma$.  
By Lemma~\ref{lem.length} we can write
    \begin{equation*}
        \begin{split}
          S_\gs(j)-\gd(i<j)-\alpha(i<j) &=  S_\gs(j)- \left(S_\gs(j)-S_\gs(i)+1\right)-\alpha(i<j)\\
           & = (S_\gs(i)-1)-\ga(i<j)  \\
        & = \#\{(i'<j')\mid i'<i \}-\#\{(i'<j')|~i'<i<j<j'\}\\
        &=\#\{(i'<j')\mid j'<i\}=c(i),
    \end{split}
\end{equation*}
where $(i'<j')$ denotes a cup in $\cupsigma$. 

\end{proof}

\begin{lemma} \label{lem.leftend} For any left end point $i$ of a cup $(i<j)$,
\begin{equation}\label{rem:i}
    i=2S_\gs(i)-\ga(i<j)-1.
\end{equation}
\end{lemma}

\begin{proof} 
By definition of the total cup sequence, $S_\gs(i) = c(i) + \al(i<j)+1$ where $c(i)$ is the total number of cups completed before the cup $(i<j)$. We also know that $i=2c(i)+\al(i<j)+1$. Thus
\begin{equation*}\label{eqn.cup.above}
\begin{split}
2S_\gs(i)-i &= 2\left( c(i)+ \al(i<j)+1\right) - \left(2c(i)+\al(i<j)+1\right)\\
&= \al(i<j)+1 
\end{split}
\end{equation*}
The formula now follows.
\end{proof}

\begin{corollary}
Let $(i^*<j^*)$ be the unique cup directly above a cup  $(i<j)$ in $\cupsigma$. Then
\[
i^*-2\left(S_\gs(i^*)-\ga(i<j)\right)=\ga(i<j).
\]
\end{corollary}
\begin{proof} Since $(i^*<j^*)$ lies directly above $(i<j)$, we know that $\ga(i^*<j^*)=\ga(i<j)-1$. The statement now follows directly by applying~\eqref{rem:i} to the cup $(i^*<j^*)$.
\end{proof}

%----------------------------
\subsection{A linear algebra lemma on rank conditions}
In this section we prove a lemma regarding the rank of a matrix satisfying certain vanishing conditions. This result will give a technical foundation for our choice of generators of our ideals $\ci_\gs$ in later sections. 

\begin{lemma} \label{lem.inv_matrix}
Let $m\geq 3$ be an odd integer and $b = \frac{1}{2}(m+1)$. Let $\binom{[m]}{b}$ and $\binom{[m]}{b-1}$ denote the set of all subsets of $[m]$ of cardinality $b$ and $b-1$, respectively. Let $M$ denote the incidence matrix with rows indexed by $\binom{[m]}{b-1}$ and column by $\binom{[m]}{b}$ defined by
\[
M_{Q,R} = \left\{ \begin{array}{ll} 1 & \text{ if $Q\subseteq R$} \\  0 & \text{ otherwise} \end{array} \right.
\]
for each $Q\in \binom{[m]}{b-1}$ and $R\in \binom{[m]}{b}$. 
Then the matrix $M$ is invertible. 
\end{lemma}

\begin{proof} Since $2b = m+1$ we have in particular that $b=m-b+1$ and $m-b=b-1$ so
\[
\binom{m}{b} = \frac{m!}{b!(m-b)!} =  \frac{m!}{(m-b+1)! (b-1)!} = \binom{m}{b-1}.
\]
In particular, the matrix $M$ is square. For $k = 0, 1, \ldots, b-1$ we define rational numbers $q_k\in \Q$ by 
\begin{align}\label{eqn.formula_inv}
q_{b-1} = \frac{1}{b} \; \text{ and } \; q_{k} = -\frac{(b-k-1)q_{k+1}}{k+1} \; \text{ for all $k=0, \ldots, b-2$}.
\end{align}
For example, if $m=9$ then $b=5$ and the reader can confirm that $q_4 = \frac{1}{5}$, $q_3 =-\frac{1}{20}$, $q_2 = \frac{1}{30}$, $q_1= -\frac{1}{20}$, and $q_0 = \frac{1}{5}$. 

Define an $\binom{m}{b}\times \binom{m}{b}$ matrix $N$ with rows indexed by elements of $\binom{[m]}{b}$ and columns by elements of $\binom{[m]}{b-1}$ by the rule 
\[
N_{R,Q} = q_{\#(Q\cap R)} 
\]
for each $Q\in \binom{[m]}{b-1}$ and $R\in \binom{[m]}{b}$. In other words, the value of each entry of $N$ is uniquely determined by the cardinality of the intersection of $Q$ with $R$.  We'll prove $MN$ is the identity matrix. 

Let $Q,T\in \binom{[m]}{b-1}$. 
By definition of the matrix $M$ we have
\begin{align}\label{eqn.productMN}
(MN)_{Q,T} = \sum_{R\in \binom{[m]}{b}} M_{Q,R}N_{R,T} = \sum_{\substack{R\in \binom{[m]}{b}\\ Q \subset R}} N_{R,T} .
\end{align}
There are precisely $b$ many subsets $R\in \binom{[m]}{b}$ containing $Q$, so the sum above has $b$ many terms. 

Suppose $T=Q$. When $Q\subset R$ we have $N_{R,Q}=q_{b-1}$ and the sum~\eqref{eqn.productMN} becomes
\[
(MN)_{Q,Q} = b \cdot q_{b-1} = b \cdot \frac{1}{b} = 1.
\]

Next suppose that $T\neq Q$. Let $k = \#(Q\cap T)$.  Given any subset $R\in \binom{[m]}{b}$ with $Q\subset R$ we have either $\#(T\cap R)=k$ or $\#(T\cap R) = k+1$. The first case occurs when $R=Q\sqcup\{\ell\}$ for some $\ell\notin T\cup Q$. There are $m-2(b-1)+k = k+1$ many possible such $\ell$, so there are $k+1$ many subsets $R\in \binom{[m]}{b}$ such that $Q\subset R$ and $\#(T\cap R) = k$. The second case of $\#(T\cap R)=k+1$ occurs when $R=Q\sqcup \{\ell\}$ for some $\ell\in T\setminus Q$. There are $b-1-k$ many possible values of $\ell$, so there are $b-1-k$ many subsets $R\in \binom{[m]}{b}$ such that $Q\subset R$ and $\#(T\cap R) = k$. Thus in this case, the sum~\eqref{eqn.productMN} becomes
\[
(MN)_{Q,T} = (k+1)q_k + (b-1-k)q_{k+1}.
\]
Our assumption that $T\neq Q$ implies $k\leq b-2$. Using the definition of $q_k$ from~\eqref{eqn.formula_inv} gives us
\[
(k+1)q_k + (b-1-k)q_{k+1} = -(k+1)\frac{(b-k-1)q_{k+1}}{k+1} + (b-1-k)q_{k+1}=0
\]
as desired. Thus, $M$ is invertible. 
\end{proof}

\begin{ex} Let $m = 5$. We have $\binom{5}{3} = \binom{5}{2} = 10$. Relative to the order the elements are listed in the sets 
\begin{align*}
\binom{[5]}{2} & = \{ 12,13,14,15,23,24,25,34,35,45 \} \\
\binom{[5]}{3} & = \{ 123,124,125,134,135,145,234,235, 245, 345\} , 
\end{align*}
the incidence matrix of Lemma~\ref{lem.inv_matrix} is 
\[
M= \begin{bmatrix} 1&1&1&0&0&0&0&0&0&0 \\ 1&0&0&1&1&0&0&0&0&0 \\ 0&1&0&1&0&1&0&0&0&0\\ 0&0&1&0&1&1&0&0&0&0 \\ 1&0&0&0&0&0&1&1&0&0\\ 0&1&0&0&0&0&1&0&1&0 \\ 0&0&1&0&0&0&0&1&1&0\\ 0&0&0&1&0&0&1&0&0&1\\ 0&0&0&0&1&0&0&1&0&1\\ 0&0&0&0&0&1&0&0&1&1  \end{bmatrix}.
\]
In this case, the rational numbers $q_0, q_1, q_{2}$ defined as in the proof of Lemma~\ref{lem.inv_matrix} are 
\[
q_2 = \frac{1}{3},\; q_1 = -\frac{1}{6} = -\frac{q_2}{2},\; q_0 = - 2q_1 = \frac{1}{3} .
\]
These numbers are the entries of the inverse matrix:

\[
M^{-1} = \frac{1}{6} \begin{bmatrix}
2 & 2 & -1 & -1 & 2 & -1 & -1 & -1 & -1 & 2 \\
2 & -1 & 2 & -1 & -1 & 2 & -1 & -1 & 2 & -1 \\
2 & -1 & -1 & 2 & -1 & -1 & 2 & 2 & -1 & -1 \\
-1 & 2 & 2 & -1 & -1 & -1 & 2 & 2 & -1 & -1 \\
-1 & 2 & -1 & 2 & -1 & 2 & -1 & -1 & 2 & -1 \\
-1 & -1 & 2 & 2 & 2 & -1 & -1 & -1 & -1 & 2 \\
-1 & -1 & -1 & 2 & 2 & 2 & -1 & 2 & -1 & -1 \\
-1 & -1 & 2 & -1 & 2 & -1 & 2 & -1 & 2 & -1 \\
-1 & 2 & -1 & -1 & -1 & 2 & 2 & -1 & -1 & 2 \\
2 & -1 & -1 & -1 & -1 & -1 & -1 & 2 & 2 & 2
\end{bmatrix} . \qedhere
\]
\end{ex}

Let $B$ be a matrix, $R$ be a subset of the index set for the rows of $B$, and $C$ be a subset of the index set for the columns of $B$. Assume $\#R = \#C$. We let $\pl_{R,C}(B)$ be the determinant  of the submatrix of $B$ defined using the rows indexed by elements of $R$ and columns by elements of $C$. The following is a statement of Laplace's theorem for determinants in a special case; see~\cite[Theorem 10.33]{Shafarevich}.

\begin{theorem}[Generalized Laplace expansion with $2\times 2$ minors] \label{thm.gen.laplace}
 Let $B$ be an $n\times n$ matrix and $\binom{[n]}{2}$ be the set of subsets of $[n]$ of cardinality $2$. Let $I\in\binom{[n]}{2}$. Then
 \begin{equation}\label{prop:GLE}
\det\left(B\right)=\sum_{S\in\binom{[n]}{2}} (-1)^{m(I,S)}\,\pl_{I, S}\left(B\right)\pl_{[n]\setminus I,\,[n]\setminus S}\left(B\right)
 \end{equation}
 where $m(I,S)=\sum_{i\in I} i+\sum_{s\in S}s$.
\end{theorem}

\begin{ex} Consider the $4\times 4$ matrix $A=(a_{ij})_{i,j\in [4]}$. Let $I = \{1,2\}$. Then
\begin{align*}
\det(A) = \left| \begin{matrix} a_{1,1} & a_{1,2} & a_{1,3} & a_{1,4} \\ a_{2,1} & a_{2,2} & a_{2,3} & a_{2,4} \\ a_{3,1} & a_{3,2} & a_{3,3} & a_{3,4}\\ a_{4,1} & a_{4,2} & a_{4,3} & a_{4,4}\end{matrix} \right| &= \left| \begin{matrix} a_{1,1} & a_{1,2}\\ a_{2,1} & a_{2,2} \end{matrix} \right| \cdot \left| \begin{matrix} a_{3,3} & a_{3,4}\\ a_{4,3} & a_{4,4} \end{matrix} \right| -  \left| \begin{matrix} a_{1,1} & a_{1,3}\\ a_{2,1} & a_{2,3} \end{matrix} \right| \cdot \left| \begin{matrix} a_{3,2} & a_{3,4}\\ a_{4,2} & a_{4,4} \end{matrix} \right| \\
&  \quad +   \left| \begin{matrix} a_{1,1} & a_{1,4}\\ a_{2,1} & a_{2,4} \end{matrix} \right| \cdot \left| \begin{matrix} a_{3,2} & a_{3,3}\\ a_{4,2} & a_{4,3} \end{matrix} \right| +  \left| \begin{matrix} a_{1,2} & a_{1,3}\\ a_{2,2} & a_{2,3} \end{matrix} \right| \cdot \left| \begin{matrix} a_{3,1} & a_{3,4}\\ a_{4,1} & a_{4,4} \end{matrix} \right| \\ \\
& \quad \quad  -   \left| \begin{matrix} a_{1,2} & a_{1,4}\\ a_{2,2} & a_{2,4} \end{matrix} \right| \cdot \left| \begin{matrix} a_{3,1} & a_{3,3}\\ a_{4,1} & a_{4,3} \end{matrix} \right| +  \left| \begin{matrix} a_{1,3} & a_{1,4}\\ a_{2,3} & a_{2,4} \end{matrix} \right| \cdot \left| \begin{matrix} a_{3,1} & a_{3,2}\\ a_{4,1} & a_{4,2} \end{matrix} \right| \\
\end{align*}
is the expansion of $\det(A)$ from~\eqref{prop:GLE}.
\end{ex}

\begin{lemma}\label{lem.rankcond} Let $a,k\geq 2$ be a positive integers with $k\geq a+1$ and let $A$ be $2a\times k$ matrix such that 
\begin{equation}\label{vanmins}
\sum_{k=1}^a \pl_{\{k,2a-k+1\},\{c_1,c_2\}}(A)=0 \quad \textup{ for all } \quad 1\leq c_1<c_2\leq k.
\end{equation}
Then $\rk(A)\leq a$, that is, all $(a+1)\times(a+1)$-minors of $A$ are zero.
\end{lemma}

\begin{proof}
Up to permuting the rows of $A$, the assumptions from~\eqref{vanmins} are equivalent to: 
\begin{equation}\label{vanmins1}
\sum_{k=1}^a \pl_{\{2k-1,2k\},\{c_1,c_2\}}(A)=0 \textup{ for all } 1\leq c_1<c_2\leq k.
\end{equation}
Since exchanging rows does not change the rank, it suffices to prove $\rk (A) \leq a$ under the assumptions of~\eqref{vanmins1}.

We will use the generalized Laplace expansion of $A$ to show that if \eqref{vanmins} holds, then all $(a+1)\times(a+1)$ minors of $A$ vanish. 
Without loss of generality, we may assume that $k=a+1$. Otherwise, we may choose a subset of $a+1$ columns of $A$ and use the argument below.

We group the rows of $A$ into sets of pairs of the form $\{2k-1, 2k\}$ with $k=1, \ldots, a$. Any subset $R\subset[2a]$ of cardinality $a+1$ must contain at least one such pair. We classify the subsets of cardinality $a+1$ in $[2a]$ according to the ``unpaired" indices they contain. 

Let $U$ be a subset of $[2a]$ such that
\begin{align}\label{eqn.nopairs}
\{2k-1,2k\}\nsubseteq U\; \text{ for all } \; k=1, \ldots, a.
\end{align}
We use the notation `$U$' because this will be the set of ``unpaired'' indices within a subset of $[2a]$ of cardinality $a+1$. Note that $U$ could be the empty set. If $U$ is indeed the subset of unpaired indices in the subset $R\subset [2a]$ of cardinality $\#R = a+1$, then $\#R-\#U \geq 2$ is even and the condition $\#U + (\#R-\#U) = a+1$ implies that $\#U$ has the same parity as $a+1$ and $\#U \leq a-1$. For example, if $a=3$ with pairs $\{1, 2\}, \,\{3,4\},\, \{5,6\}$ the subsets $U$ that arise in the way are either the empty set or any of the two element subsets of $[2a]=[6]$ that do not consist of a pair. No one element subset of $[6]$ can occur as the set of unpaired elements of a subset of $[6]$ of cardinality $a+1=4$ because the cardinality has the wrong parity.

Given a subset $U$ satisfying~\eqref{eqn.nopairs} such that $\#U$ has the same parity as $a+1$ and $\#U\leq a-1$ we consider the collection $\cs_U$ of all subsets $R\subseteq [2a]$ of cardinality $a+1$ such that $U\subseteq R$ and
\[
k\in R\backslash U\Rightarrow\begin{cases}
    k-1\in R, & \text{if $k$ is even } \\
    k+1\in R,  & \text{if $k$ is odd. }
\end{cases}
\]
In other words, $\cs_U$ is the set of subsets of $[2a]$ of cardinality $a+1$ such that $R\setminus U$ is a union of pairs. Each $R\in \cs_U$ contains exactly $b:=\frac{1}{2}(a+1-\#U)$ pairs and there are precisely $\binom{a-\#U}{b}$ elements in~$\cs_U$. To complete the proof, it suffices to show that $\pl_{R,[a+1]}(A)=0$ for all $R\in \cs_U$ and $U\subset [2a]$ such that~\eqref{eqn.nopairs} holds, $\#U$ has the same parity as $a+1$, and $\#U\leq a-1$.

Fix a $U$ satisfying the relevant hypotheses and set
\[
P=\{k\in[a] \mid  \{2k-1, 2k\}\cap U = \emptyset\}.
\]
In other words $P$ is the indexing set of all pairs $\{2k-1, 2k\}$ that do not contain any element of $U$, so $P$ has cardinality $a-\#U$. Recall that $b$ is the number of pairs in any $R\in \cs_U$, that is $2b = a+1-\#U$, and thus $\# P =a-\#U = 2b-1$ is odd. 
The set $P$ keeps track of the possible pairs, each of the form $\{2k-1,2k\}$ with $k\in P$, that we can add to the set $U$ in order to obtain an element of $\cs_U$.

Given a subset $Q\subset P$ of cardinality $b-1$, let $Q'$ be the disjoint union 
\[
Q':=\{2k-1,2k \mid  k\in Q\} \sqcup U
\]
of the $b-1$ pairs indexed by elements of $Q$ and $U$. In order to obtain an element of $\cs_U$ we must add one more pair to the set. Note that we may add any pair $\{2r-1, 2r\}$ for $r\in P\setminus Q$. Consider such a subset $R=Q'\sqcup \{2r-1,2r\}\in \cs_U$. Applying the generalized Laplace expansion with $I= \{2r-1,2r\}$ to the determinant $\pl_{R, [a+1]}$ we have 
\begin{align}\label{eqn.genlaplace}
\pl_{R, [a+1]} (A) 
= \pl_{\left( Q'\sqcup\{2r-1, 2r\}\right), [a+1]}(A) 
= \sum_{S\in \binom{[a+1]}{2}} (-1)^{m(I,S)} \pl_{\{2r-1,2r\}, S}(A) \, \pl_{Q',[a+1]\setminus S}(A),
\end{align}
where we may take $m(I,S) = 1+\sum_{s\in S}s$ since the sum of elements in $I$ is odd. 
Summing over all $R$ we obtain by adding a pair to $Q'$ we have,
\begin{align*}
\sum_{\substack{R\in \cs_U\\ Q'\subseteq R}} \pl_{R,[a+1]}(A) &=  
\sum_{r\in P\setminus Q} \pl_{\{2r-1,2r\}\sqcup Q' , [a+1]}(A) \\ 
&= \sum_{r=1}^a \pl_{\{2r-1,2r\}\sqcup Q' , [a+1]}(A) \quad \text { since $\pl_{\{2r-1, 2r\} \sqcup Q'}(A)=0$ for all $r\notin P\setminus Q$ }\\
&= \sum_{r=1}^a \left( \sum_{S\in \binom{[a+1]}{2}} (-1)^{m(I,S)} \pl_{\{2r-1, 2r\},S}(A) \pl_{Q', [a+1]\setminus S}(A) \right) \quad \text{ by~\eqref{eqn.genlaplace}} \\
&= \sum_{S\in \binom{[a+1]}{2}} \left(  \left( \sum_{r=1}^a \pl_{\{2r-1, 2r\},S}(A) \right) (-1)^{m(I,S)}  \pl_{Q', [a+1]\setminus S}(A)\right)\\
&= 0 \quad\quad \text{by \eqref{vanmins1}.}
\end{align*}

In particular, for each subset $Q$ of $P$ with cardinality $b-1$ we have
\begin{equation}\label{eqn:GLP2}
\sum_{\substack{R\in\cs_U\\ Q'\subseteq R}}  \pl_{R,[a+1]}\left(A\right) = 0.
\end{equation}
Thus, the column vector $\vec{p}_U:=\left(\pl_{R, [a+1]}\left(A\right)\right)_{R\in\cs_U}$ with entries indexed by elements of $\cs_U$ satisfies the matrix equation
\[
M\vec{p}_U=\vec{0}
\]
where $M$ is the invertible incidence matrix as in Lemma~\ref{lem.inv_matrix} with $m= \#P$ (which is odd) defined by 
\begin{equation*}
M_{Q, R}:=\begin{cases}
        1& \text{ if } Q'\subseteq R\\
        0& \text{otherwise}
    \end{cases} 
\end{equation*}
for all $Q\in \binom{P}{b-1}$ and $R\in \cs_U$ (we may identify $R$ with a subset in $\binom{P}{b}$ associating $R$ to the index set of the pairs in $R\setminus U$).
Since $M$ is invertible by Lemma~\ref{lem.inv_matrix}, we conclude $\vec{p}_U=\vec{0}$ and thus $\pl_{R, [a+1]} = 0$ for all $R\in \cs_U$, as desired. 
\end{proof}

%%%%%%%%%%%%%%%%%%%%%%

\subsection{Decomposing components of two-row Springer fibers}
In this section, we make the isomorphism from Proposition~\ref{prop.Fresse} explicit in the two-row case using the cell $\cell$ corresponding to $\gs\in \SYT(n-\ell, \ell)$. 

Suppose $\gs$ has prime decomposition $\gs=\gs_1\circ \gs_2 \circ \cdots \circ \gs_k$. For each $i$ with $1\leq i \leq k$, let $\mu_i:= |\gs_i|$ and set $\mu_0=0$. The sequence $\mu:=(\mu_1,\mu_2, \ldots, \mu_k)$ is a composition of $n$ and we let $\parsum_i = \mu_1+\cdots + \mu_{i-1}+\mu_i$ for $i=1, \ldots, k$ be the $i$-th partial sum of $\mu$. Since $\gs$ has two-rows, we have by Corollary~\ref{cupdecomposition} that either $\mu_i = 1$ (which occurs when $\parsum_i$ is the endpoint of a ray in $\cupsigma$) or $\mu_i$ is even (which occurs when the subdiagram supported on the vertices $\parsum_{i-1}+1, \parsum_{i-1}+2, \ldots, \parsum_i$ is a connected standard noncrossing matching in $\cupsigma$). When $\mu_i\geq 2$ we call the vertices $\{\parsum_{i-1}+1, \parsum_{i-1}+2, \ldots, \parsum_i\}$ a \emph{vertex block}. We say that the set $\{\parsum_{i-1}+1, \parsum_{i-1}+2, \ldots, \parsum_{i-1}+\frac{\mu_i}{2}\}$ is the \emph{first half} of the block and that $\{\parsum_{i-1}+\frac{\mu_i}{2}+1 , \parsum_{i-1}+\frac{\mu_i}{2}+2, \ldots, \parsum_i \}$ is the \emph{second half}. For example, in the cup diagram from Example~\ref{ex.disconnected1}, the vertices ${7,8,9,10}$ form a block with first half $\{7,8\}$ and second half $\{9,10\}$.

\begin{definition}\label{def.tau_gs} Let $\tau_\mu\in S_n$ be the permutation with one-line notation that lists (in increasing order) all endpoints of rays and all indices in the first half of every block determined by the prime decomposition of $\gs$, followed by all indices appearing in the second half of any block (in increasing order). 
\end{definition}

 If $\gs$ and $\cupsigma$ are as in Example~\ref{ex.disconnected1}, then $\mu_1 = \mu_4= 4$ and $\mu_2 = \mu_3 = 1$ so $\tau_\mu = [1,2,5,6,7,8,3,4,9,10]$. 
Note that the permutation $\tau_\mu$ is uniquely determined by the composition $\mu$ defined by the size of each prime factor of $\gs$, and not on the prime factors themselves. For example, any prime tableau $\gs$ of shape $(\ell, \ell)$ will have $\tau_\mu$ equal to the identity.

Recall that for all $i\in [n]$, $c(i)$ is the number of completed cups that appear at or before vertex $i$ in $\cupsigma$ and $r(i)$ is the number of rays that appear at or before vertex $i$ in $\cupsigma$. 

\begin{lemma}\label{lemma.tau_gs.formula} Suppose $\gs$ is a standard tableau of shape $(n-\ell, \ell)$ with prime decomposition $\gs=\gs_1\circ\gs_2\circ \cdots \circ \gs_k$. Let $\mu_i = |\gs_i|$ and $\parsum_i = \mu_1+\mu_2+\cdots +\mu_i$ for all $1\leq i \leq k$. The permutation $\tau_\mu$ of Definition~\ref{def.tau_gs} is given explicitly by the following formulas.
\begin{enumerate}
\item   If $t=\parsum_i$ is the endpoint of a ray in $\cupsigma$, then
\[
\tau_\mu^{-1}(t) = r(t)+c(t).
\]
\item  If $\mu_i\geq 2$ is even and $t=\parsum_{i-1}+j$ for some $1\leq j \leq \frac{\mu_i}{2}$ occurs in the first half of the block $\{\parsum_{i-1}+1, \parsum_{i-1}+2, \ldots, \parsum_i\}$ in $\cupsigma$, then 
\[
\tau_\mu^{-1}(t) = r(\parsum_{i-1})+c(\parsum_{i-1})+j.
\]
\item  If $\mu_i\geq 2$ is even and $t=\parsum_{i-1}+\frac{\mu_i}{2}+j$ for some $1\leq j \leq \frac{\mu_i}{2}$ occurs in the second half of the block $\{\parsum_{i-1}+1, \parsum_{i-1}+2, \ldots, \parsum_i\}$ in $\cupsigma$, then
\[
\tau_\mu^{-1}(t) = n-\ell + c(\parsum_{i-1}) +j. 
\]
\end{enumerate}
\end{lemma}
\begin{proof} Suppose $\mu_i\geq 2$. The number of entries in the first half of the block $\{\parsum_{i-1}+1, \parsum_{i-1}+2, \ldots, \parsum_i\}$ (which is, of course, also the number entries in the second half of the $i$-th block) is $\frac{\mu_i}{2}$. The key fact used in the proof below is that, since the subdiagram of $\cupsigma$ on the $\mu_i$ vertices $\{\parsum_{i-1}+1, \parsum_{i-1}+2, \ldots, \parsum_i\}$ is the standard matching $\cc(\sigma_i)$, there are $\frac{\mu_i}{2}$ cups in $\cc(\sigma_i)$. 

By definition, the first $n-\ell$ entries of the one-line notation for the permutation $\tau_\mu$ list, in order, all endpoints of rays and all indices in the first half of every block. 

Suppose $\mu_i =1$ and $t=\parsum_i$ is the endpoint of a ray. The number $r(\parsum_{i-1})+c(\parsum_{i-1})$ is equal to the number of rays and complete cups that have already appeared, which is precisely the number of earlier rays plus the number of indices in the first half of the earlier blocks. The fact that $t$ is the endpoint of a ray implies $r(t)+c(t) = r(\parsum_{i-1})+c(\parsum_{i-1}) +1$, and the formula in (1) follows.  

Next, suppose $\mu_i\geq 2$ and $t= \parsum_{i-1}+j$ is the $j$-th entry in the first half of the corresponding block. The position of $t$ in the one-line notation of $\tau_\mu$ is exactly the total number of completed cups and rays appearing before the start of the block plus $j$. Since the total number of completed cups and rays appearing before block $i$ is $r(\parsum_{i-1})+c(\parsum_{i-1})$, the formula from (2) follows. 

By definition, the last $\ell$ entries of the one-line notation for the permutation $\tau_\mu$ list, in order, all the indices in the second half of every block.
Suppose $\mu_i\geq 2$ and $t=\parsum_{i-1}+\frac{\mu_i}{2}+j$ is the $j$-th entry in the second half of the corresponding block. The sum total number of entries in the second half of every earlier block is given by the number of completed cups appearing before block $i$, which is $c(\parsum_{i-1})$. Thus, the position of $t$ in the one-line notation of $\tau_\mu$ is equal to $n-\ell + c(\parsum_{i-1})+j$, which proves (3).
\end{proof}

Let $\dot \tau_\mu$ denote the permutation matrix whose $i$-th column is the standard basis vector $e_{\tau_{\mu}(i)}$. 
We study the image of a generic matrix in $F_\gs$ after left multiplication by $\dot\tau_\mu$.

\begin{ex} \label{ex.disconnected2} Let $\gs$ and $\cupsigma$ be as in Example~\ref{ex.disconnected1}; we have $\tau_\mu = [1,2,5,6,7,8,3,4,9,10]$.  A generic matrix in $\cell$ was also computed in Example~\ref{ex.disconnected1}. The image after left multiplication by $\dot\tau_\mu$ is:
\[
\begin{bmatrix} 
\textcolor{blue}{a_1} & \textcolor{red}{a_2} & 1 & 0 & 0 & 0 & 0 & 0 & 0 & 0 \\ 
0 & \textcolor{blue}{a_1} & 0 & 1 & 0 & 0 & 0 & 0 & 0 & 0\\
1 & 0 & 0 & 0 & 0 & 0 & 0 & 0 & 0 & 0\\
0 & 1 & 0 & 0 & 0 & 0 & 0 & 0 & 0 & 0\\
0 & 0 & 0 & 0 & 1 & 0 & 0 & 0 & 0 & 0\\
0 & 0 & 0 & 0 & 0 & 1 & 0 & 0 & 0 & 0\\
0 & 0 & 0 & 0 & 0 & 0 & \textcolor{purple}{a_7} & \textcolor{teal}{a_8} & 1 & 0\\
0 & 0 & 0 & 0 & 0 & 0 & 0 & \textcolor{purple}{a_7} & 0 & 1\\
0 & 0 & 0 & 0 & 0 & 0 & 1 & 0 & 0 & 0\\
0 & 0 & 0 & 0 & 0 & 0 & 0 & 1 & 0 & 0\\
\end{bmatrix}. 
\]
Note that this matrix is block diagonal with blocks of size $\mu_1=4,\, \mu_2=1,\, \mu_3=1$, and $\mu_4=4$. The projection to each block is either the $1\times 1$ matrix $[1]$ or a $4\times 4$ matrix in $F_{\sigma_1} = F_{\sigma_4}$ (see Example~\ref{ex.n=4nested}).
\end{ex}

For a composition $\mu= (\mu_1, \mu_2, \ldots, \mu_k)$ we consider the inclusion map
\begin{align}\label{eqn.inclusion}
\iota: GL_{\mu_1}(\C) \times GL_{\mu_2}(\C) \times \cdots \times GL_{\mu_k}(\C) \hookrightarrow GL_n(\C)
\end{align}
defined by embedding each $GL_{\mu_i}(\C)$ as a diagonal block of $GL_n(\C)$. The inclusion map induces a closed embedding of flag varieties, $\bar\iota:\Fl_{\mu_1}(\C) \times \Fl_{\mu_2}(\C)\times \cdots \times \Fl_{\mu_k}(\C) \hookrightarrow \Fl_n(\C)$.

\begin{lemma}\label{lem.taushift} Let $\gs=\gs_1\circ \cdots \circ \gs_k$ be the prime decomposition of $\gs$ and $\iota$ denote the inclusion map from~\eqref{eqn.inclusion} above. Then
\[
\iota(F_{\gs_1}, F_{\gs_2}, \ldots,F_{\gs_k} ) = \dot\tau_\mu \cell.
\]  
In particular, $\iota$ induces an isomorphism $\SF_{\gs_1}\times \SF_{\gs_2}\times \cdots \times \SF_{\gs_k}\simeq \SF_\gs $.
\end{lemma}
\begin{proof} Denote the $t$-th column of a matrix $A\in \cell$ by $v_t$; recall that the vector $v_t$ is defined explicitly using $\cupsigma$ in Definition~\ref{def:cell}. The product $\dot\tau_\mu A$ has columns $\dot\tau_\mu(v_1), \dot\tau_\mu(v_2), \ldots, \dot\tau_\mu(v_n)$. We compute $\dot\tau_\mu(v_t)$ for each $t$ by considering cases: either $t$ is the endpoint of a ray, the right endpoint of a cup, or the left endpoint of a cup in $\cupsigma$. 

When $t$ is the endpoint of a ray, $v_t = e_{r(t)+c(t)}$ and Lemma~\ref{lemma.tau_gs.formula}(1) yields $\dot\tau_\mu(v_t) = e_{\tau_\mu(r(t)+c(t))} = e_t$. 

When $t$ is the right endpoint of a cup, we also have $v_t = e_{r(t)+c(t)}$. Since $t$ is incident to a cup, it must appear in some block of $\gs$. Say $t$ is in the block $\{\parsum_{i-1}+1, \parsum_{i-1}+2, \ldots, \parsum_i\}$ for some $i$ with $\mu_i\geq 2$. This means the positive integer $t-\parsum_{i-1}$ is an entry of the prime factor $\gs_i$ of $\gs$. Since $r(t)+c(t)$ counts the total number of complete cups and rays that appear up through vertex $t$, we have that $r(t)+c(t) = r(\parsum_{i-1})+c(\parsum_{i-1})+c_i(t-s_{i-1})$ where $c_i$ denotes the cup count for the diagram $\cc(\gs_i)$ of the prime factor $\gs_i$. We know the connected component of the matching $\cupsigma$ on the block $\{\parsum_{i-1}+1, \parsum_{i-1}+2, \ldots, \parsum_i\}$ is the standard matching associated to $\gs_i$ and has exactly $\frac{\mu_i}{2}$ many cups, so $c_i(t-\parsum_{i-1}) \leq \frac{\mu_i}{2}$. Thus, $\parsum_{i-1}+c_i(t-\parsum_{i-1})$ is in the first half of the block and applying Lemma~\ref{lemma.tau_gs.formula}(2) with $j=c_i(t-\parsum_{i-1})$ yields 
\[
\tau_\mu(r(t)+c(t)) = \tau_\mu(r(\parsum_{i-1})+c(\parsum_{i-1})+c_i(t-\parsum_{i-1})) = \parsum_{i-1}+c_i(t-\parsum_{i-1}).
\]
We conclude $\dot\tau_\mu(v_t) = e_{\parsum_{i-1}+c_i(t-\parsum_{i-1})}$. Note, in particular, that the only nonzero entries of the column $\dot\tau_\mu (v_t)$ for $t$ in the block $\{\parsum_{i-1}+1, \parsum_{i-1}+2, \ldots, \parsum_i\}$ are in rows indexed by the same set.

When $t$ is the left endpoint of a cup, we denote the right endpoint by $\E(t)$ and let $\ga=\ga(t<\E(t))$. By Definition~\ref{def:cell} we have 
\[
v_t = a_t e_{r(t)+c(t)+1}+a_{t_1}e_{r(t)+c(t)+2}+\cdots + a_{t_{\ga}}e_{r(t)+c(t)+\ga +1} + e_{n-\ell+c(t)+\ga+1}.
\]
First, consider the indices of the form $r(t)+c(t)+b$ for $1\leq b\leq \ga+1$. As above, since $t$ is adjacent to a cup it appears in a block of $\gs$. Say $t$ is in the block $\{\parsum_{i-1}+1, \parsum_{i-1}+2, \ldots, \parsum_i\}$ for some $i$ with $\mu_i\geq 2$. Letting $c_i$ denote the cup count on $\cc(\gs_i)$,  we have 
\[
r(t)+c(t)+b = r(\parsum_{i-1})+c(\parsum_{i-1})+c_i(t-s_{i-1}) +b.
\]
Because the numbers $c_i(t-\parsum_{i-1})$ and $\ga=\ga(t<\E(t))$ count distinct sets of cups in $\cc(\gs_i)$ and since neither of these numbers counts the cup $(t<\E(t))$ we have that $c_i(t-\parsum_{i-1}) +b \leq c_i(t-\parsum_{i-1}) + \ga+1 \leq  \frac{\mu_i}{2}$. In particular, the number $\parsum_{i-1}+c_i(t-\parsum_{i-1})+b$ must be in the first half of the block. Applying  Lemma~\ref{lemma.tau_gs.formula}(2) with $j = c_i(t-\parsum_{i-1})+b$ yields 
\[
\tau_\mu(r(t)+c(t)+b) = \tau_\mu(r(\parsum_{i-1})+c(\parsum_{i-1})+c_i(t-\parsum_{i-1})+b) = \parsum_{i-1}+c_i(t-\parsum_{i-1})+b.
\]
Next, consider the index $n-\ell+c(t)+\ga+1 = n-\ell + c(\parsum_{i-1})+c_i(t-\parsum_{i-1}) +\ga +1$. We showed above that $c_i(t-\parsum_{i-1}) +\ga +1 \leq \frac{\mu_i}{2}$. Thus $\parsum_{i-1}+\frac{\mu_i}{2}+ c_i(t-\parsum_{i-1}) +\ga +1$ appears in the second half of the block and applying  Lemma~\ref{lemma.tau_gs.formula}(3) with $j=c_i(t-\parsum_{i-1}+\ga+1)$,
\begin{align*}
\tau_\mu(n-\ell + c(t)+\ga +1) &= \tau_\mu(n-\ell + c(\parsum_{i-1})+c_i(t-\parsum_{i-1})+\ga +1)\\
&= \parsum_{i-1}+\frac{\mu_i}{2}+  c_i(t-\parsum_{i-1}) +\ga +1.
\end{align*}
In summary, we have proved that 
\begin{align*}
\dot\tau_\mu (v_t) &=  a_t e_{\parsum_{i-1}+c_i(t-\parsum_i)+1}+a_{t_1}e_{\parsum_{i-1}+c_i(t-\parsum_i)+2}+\cdots \\
& \quad \quad \quad \quad \quad + a_{t_{\ga}}e_{\parsum_{i-1}+c_i(t-\parsum_i)+\ga +1} + e_{\parsum_{i-1}+\frac{\mu_i}{2} + c_{i}(t-\parsum_{i-1}) + \ga+1}.
\end{align*}
Once again, we observe that the only nonzero entries of the column $\dot\tau_\mu (v_t)$ for $t$ in the block $\{\parsum_{i-1}+1, \parsum_{i-1}+2, \ldots, \parsum_i\}$ are only in rows indexed by the same set.

Consider the matrix $\dot\tau_\mu A$. Our computations above show the matrix is block diagonal. If $t$ is the endpoint of a cup, then it corresponds to a $1\times 1$ block equal to $[1]$. When $\mu_i\geq 2$ is even, then the diagonal block of $A$ corresponding to the rows and columns indexed by $\{\parsum_{i-1}+1, \parsum_{i-1}+2, \ldots, \parsum_i\}$ is the $\mu_i\times \mu_i$ matrix with columns equal to $v_j = e_{c_i(j)}$ when $j$ is the right endpoint of a cup, and if $j$ is the left endpoint of a cup then, 
\[
v_j = a_t e_{c_i(j)+1} + a_{t_1}e_{c_i(j)+2}+\cdots + a_{t_\ga}e_{c_i(j)+\ga+1}+ e_{\frac{\mu_i}{2} + c_i(j)+\ga+1}
\]
where $j=t-\parsum_{i-1}$. We have $\ga = \ga(t,\E(t)) = \ga(j, \E(j))$ since the number of cups nested above $(t,\E(t))$ in $\cupsigma$ and the number of cups nested above $(j,\E(j))$ in $\cc(\gs_i)$ are equal. Note that we can write
\[
v_j = (0^{c_i(j)}, a_t, a_{t_1}, \ldots, a_{\ga}, 0^{\frac{\mu_i}{2}}, 1, 0^{\mu_i - c_i(j) - \ga -1}).
\]
Comparing with Definition~\ref{def:cell} for the prime tableau $\gs_i$, we conclude the diagonal submatrix corresponding to the rows and columns indexed by $\{\parsum_{i-1}+1, \parsum_{i-1}+2, \ldots, \parsum_i\}$  is a matrix representative for the flags in $F_{\gs_i}$ and the proof of the first statement of the lemma is complete.

The second statement follows immediately from Theorem~\ref{thm.GNST}, since
\[
\SF_\gs \simeq \dot \tau_\mu \SF_\gs = \overline{\dot\tau_\mu F_\gs} \simeq \overline{F_{\gs_1}} \times \overline{F_{\gs_2}} \times \cdots \times \overline{F_{\gs_k}}  = \SF_{\gs_1}\times \SF_{\gs_2}\times \cdots \times \SF_{\gs_k}   
\]
This completes the proof. 
\end{proof}

%%%%%%%%%%%%%%%%%%%%%%%%%%
\section{The ideal $\texorpdfstring{\isigma}{isigma}$}\label{sec:IdealIsig}

We now define a polynomial ideal $\isigma$ associated to each two-row tableau $\gs$ of shape $(n-\ell, \ell)$, or equivalently, to the noncrossing matching $\cupsigma$. 

Let $\mathbf{z}$ denote the set of $n^2$ variables  $\{z_{i,j}\mid 1\leq i,j\leq n\}$ and $\C[\mathbf{z}] = \C[z_{1,1},z_{1,2}, z_{1,3} \ldots]$. Throughout this paper, we identify $z_{i,j}$ with the coordinate function on the affine variety $M_n$ of $n\times n$ matrices, returning the $(i,j)$ entry.  Let $GL_n(\C)$ denote the algebraic group of $n\times n$ complex invertible matrices.  As is typical, we identify $GL_n(\C)$ as a principal open subset in $M_n$.

Given an ideal $\ci\subset \C[\mathbf{z}]$, its \emph{vanishing set} $\cv(\ci)$ is the affine subvariety of $n\times n$ matrices $A$ such that $f(A)=0$ for all $f\in \ci$. 

\subsection{Definition for standard noncrossing matchings} Our first step is to identify a the ideal $\ci_\gs$ when $\gs$ has shape $(\ell, \ell)$, that is, when $\cupsigma$ is a standard noncrossing matching. The ideal $\isigma$ is defined using the total cup sequence of $\cupsigma$ from Section~\ref{sec:SYTmatchings}. We start with the following monomial ideal.

\begin{definition}\label{def:JSigma} Let $\gs$ be a standard tableau of shape $(\ell, \ell)$. We set
\begin{eqnarray*}
\cj_\gs &:=& \sum_{\substack{j\in [n]\\ S_\sigma(j)<\ell}}\left< z_{i,j}, \, z_{i+\ell, j} \mid S_\gs(j)+1\leq i \leq \ell \right> \subset \C[\mathbf{z}].
\end{eqnarray*}
\end{definition}

\begin{ex}\label{ex.Jsigmaideal1}
 Consider the following standard Young tableau and its associated standard noncrossing matching. 
 \begin{multicols}{2}
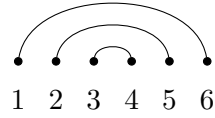
\begin{center}
\ytableausetup{centertableaux}
  \begin{ytableau}
          1&2&3\\
          4&5&6
 \end{ytableau} 
 \end{center}
\vspace*{-.3in}
\begin{tikzpicture}[scale=0.5]	
	\begin{scope}[shift={(8,0)}]
		\foreach \i in {0,...,5}
		{
			\node[pnt] at (\i,0)(\i){};
		}
		\draw(0)  to [bend left=80] (5);
		\draw(1)  to [bend left=80] (4);
		\draw(2)  to [bend left=80] (3);
		\draw (0,-1) node {$1$};
        \draw (1,-1) node {$2$};
        \draw (2,-1) node {$3$};
        \draw (3,-1) node {$4$};
        \draw (4,-1) node {$5$};
        \draw (5,-1) node {$6$};
	\end{scope}
\end{tikzpicture}
\end{multicols}
\vspace*{-.2in}
\noindent We have $n=6$, $\ell=3$ and total cup sequence is $S_\gs=(1,2,3,3,3)$, so
\begin{equation*}
  \cj_{\gs}=\langle z_{2,1}, z_{3,1}, z_{5,1}, z_{6,1},z_{3,2}, z_{6,2}\rangle.\qedhere
\end{equation*}
The ideal $\cj_\gs$ corresponds to certain entries of matrices in the cell $F_\gs$ that must be equal to zero. In this case $\cell$ consists of matrices of the form
\[
\begin{bmatrix} a_1 & a_2 & a_3 & 1 & 0 & 0\\  \mathbf{0} & a_1 & a_2 & 0 & 1 & 0 \\  \mathbf{0} & \mathbf{0} & a_1 & 0 & 0 & 1 \\ 1 & 0 & 0 & 0 & 0 & 0\\  \mathbf{0} & 1 & 0 & 0 & 0 & 0\\  \mathbf{0} & \mathbf{0} & 1 & 0 & 0 & 0  \end{bmatrix} ,\,\text{ with } \, a_1,a_2,a_3\in \C
\]
where the bold entries indicate the zero corresponding to the variables in $\cj_\gs$.  
\end{ex}

We define a particular matrix for each cup $(i<j)$ in $\cupsigma$ with size $\gd(i<j)\geq 2$.  Given such a cup we let $M_{(i<j)}$ be the $(2\al(i<j)+2)\times (j-i)$ matrix,
\begin{equation}\label{eqn.matrix}
M_{(i<j)}:=  \begin{bmatrix}
    z_{S_\gs(i)-\al(i<j),i} & z_{S_\gs(i+1)-\al(i<j),i+1}& \cdots & z_{S_\gs(j-1)-\al(i<j),j-1}\\
    z_{S_\gs(i)-\al(i<j)+1,i}& z_{S_\gs(i+1)-\al(i<j)+1,i+1} & \cdots &z_{S_\gs(j-1)-\al(i<j)+1,j-1}\\
    \vdots & \vdots &\ddots& \vdots\\
    z_{S_\gs(i),i} & z_{S_\gs(i+1),i+1}& \cdots & z_{S_\gs(j-1),j-1}\\
z_{S_\gs(i)-\al(i<j)+\ell,i} & z_{S_\gs(i+1)-\al(i<j)+\ell,i+1}& \cdots & z_{S_\gs(j-1)-\al(i<j)+\ell,j-1}\\
    z_{S_\gs(i)-\al(i<j)+\ell+1,i}& z_{S_\gs(i+1)-\al(i<j)+\ell+1,i+1} & \cdots &z_{S_\gs(j-1)-\al(i<j)+\ell+1,j-1}\\
    \vdots & \vdots &\ddots& \vdots\\
    z_{S_\gs(i)+\ell,i} & z_{S_\gs(i+1)+\ell,i+1}& \cdots & z_{S_\gs(j-1)+\ell,j-1}\\
\end{bmatrix}.
\end{equation}

\begin{remark}\label{rem.evaluation} Recall that $N$ is a nilpotent matrix of Jordan type $\lambda=(\ell, \ell)$ in Jordan canonical form; see~\eqref{eqn.Jbasis}. Let 
\[
v_k = \left[z_{1,k}, z_{2,k}, \ldots, z_{S_\gs(k),k}, 0^{\ell-S_\gs(k)}, z_{\ell+1,k}, z_{\ell+2,k}, \ldots, z_{\ell+S_\gs(k),k}, 0^{\ell-S_\gs(k)}\right]^T
\] 
denote the $k$-th column of the matrix $Z=(z_{i,j})_{1\leq i,j\leq n}$ where all variables in $\cj_\gs$ are set equal to $0$.   Then 
\[
N^{S_{\gs}(k)-\al(i<j)-1}v_k = \begin{bmatrix} z_{S_\gs(k)-\al(i<j), k}\\ z_{S_\gs(k)-\al(i<j)+1, k}\\ \vdots\\ z_{S_\gs(k), k}\\ 0^{\ell - \al(i<j)-1} \\ z_{S_\gs(i)-\al(i<j)+\ell, k}\\ z_{S_\gs(i)-\al(i<j)+\ell+1, k}\\ \vdots\\ z_{S_\gs(i)+\ell, k}\\ 0^{\ell - \al(i<j)-1} \end{bmatrix} .
\]
In particular, we note that the matrix $M_{(i<j)}$ defined in~\eqref{eqn.matrix} is equal to 
\begin{eqnarray*}
\begin{bmatrix}  | & | &  & |  \\   N^{S_\gs(i)-\al(i<j)-1}v_i &  N^{S_\gs(i+1)-\al(i<j)-1}v_{i+1} & \cdots &  N^{S_\gs(j-1)-\al(i<j)-1}v_{j-1} \\ | & | &  & |   \end{bmatrix}
\end{eqnarray*}
with the rows of zeros removed.    
\end{remark}

Given a $k\times m$ matrix $Z$ with entries in $\C[\mathbf{z}]$ and subsets $R\subseteq [k]$ and $C\subseteq [m]$, we let $\pl_{R,C}(Z)\in \C[\mathbf{z}]$ be the determinant  of the submatrix of $Z$ determined by the rows in $R$ and in $C$. 

\begin{definition}\label{def:MSigma} Let $\gs$ be a standard tableau of shape $(\ell, \ell)$. For each cup $(i<j)$ in $\cupsigma$ with size $\gd(i<j)\geq 2$ and for each pair of indices $i\leq c_1<c_2 \leq j-1$, we set 
\begin{align*}
\pl_{c_1, c_2} ^{(i<j)} &:= \sum_{k=0}^{\ga(i<j)} \pl_{\{k+1, 2\ga(i<j)+2-k\}, \{c_1-i+1, c_2-i+1\}}(M_{(i<j)}) \\
& = \sum_{k=0}^{\al(i<j)} \begin{vmatrix}  z_{S_\gs(c_1) - \al(i<j) +k, c_1} &  z_{S_\gs(c_2) - \al(i<j) +k, c_2} \\ z_{\ell + S_\gs(c_1) -k , c_1} & z_{\ell + S_\gs(c_2)-k, c_2} \end{vmatrix}  
\end{align*}
and define the ideal
\[
\cm_{\gs,(i<j)}: = \left< \pl_{c_1,c_2}^{(i<j)} \mid i \leq c_1, c_2 \leq j-1  \right> .
\]  
Now we set
\[
\cm_\sigma:= \sum \cm_{\gs,(i<j)} \subseteq \C[\mathbf{z}]
\]
where the sum is taken over all cups in $\cupsigma$ with size at least $2$. 
We define the ideal $\ci_\gs$ in $\C[\mathbf{z}]$ to be the sum
\[
\isigma := \cj_{\gs} + \cm_\gs.
\]
\end{definition}

\begin{ex}\label{1component.a} 
We consider the following standard Young tableau and its associated standard noncrossing matching as in Example~\ref{ex.Jsigmaideal1}, where we computed $\cj_\gs$. 
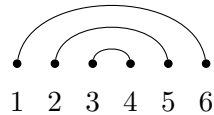
\begin{multicols}{2}
\begin{center}
\ytableausetup{centertableaux}
  \begin{ytableau}
          1&2&3\\
          4&5&6
 \end{ytableau} 
 \end{center}

\vspace*{-.3in}
\begin{tikzpicture}[scale=0.5]	
	\begin{scope}[shift={(8,0)}]
		\foreach \i in {0,...,5}
		{
			\node[pnt] at (\i,0)(\i){};
		}
		\draw(0)  to [bend left=80] (5);
		\draw(1)  to [bend left=80] (4);
		\draw(2)  to [bend left=80] (3);
		\draw (0,-1) node {$1$};
 
        \draw (1,-1) node {$2$};
 
        \draw (2,-1) node {$3$};
        \draw (3,-1) node {$4$};
    
        \draw (4,-1) node {$5$};
        \draw (5,-1) node {$6$};
	\end{scope}
\end{tikzpicture}
\end{multicols}
\vspace*{-.2in}
There are two cups of size two or more, namely $(1<6)$ and $(2<5)$. These gives us matrices
\begin{equation*}
M_{(1<6)} = \begin{bmatrix}
    z_{1,1}& z_{2,2}& z_{3,3} & z_{3,4} &z_{3,5}\\
    z_{4,1}& z_{5,2}& z_{6,3} & z_{64} &z_{6,5}
    \end{bmatrix}
\quad \text{ and } \quad
M_{(2<5)} = \begin{bmatrix}
    z_{1,2}& z_{2,3}& z_{2,4} \\
    z_{2,2}& z_{3,3}& z_{3,4}\\
    z_{4,2} & z_{5,3} & z_{5,4}\\
    z_{5,2}& z_{6,3}& z_{6,4}
    \end{bmatrix}.
\end{equation*}
Since $\ga(1<6)=0$ the ideal $\cm_{\gs,(1<6)}$ is generated by all $2\times 2$ minors of the matrix $M_{(1<6)}$.  Since $\ga(2<5) = 1$, $\cm_{\gs,(2<5)}$ is generated by sums of $2\times 2$ minors of $M_{(2<5)}$:
\begin{align*}
\pl_{2,3}^{(2<5)} &= \pl_{\{1,4\},\{1,2\}}(M_{(2<5)})+\pl_{\{2,3\},\{1,2\}}(M_{(2<5)}),\\
\pl_{2,4}^{(2<5)}&=\pl_{\{1,4\},\{1,3\}}(M_{(2<5)})+\pl_{\{2,3\},\{1,3\}}(M_{(2<5)}),\\
\pl_{3,4}^{(2<5)}&= \pl_{\{1,4\},\{2,3\}}(M_{(2<5)})+\pl_{\{2,3\},\{2,3\}}(M_{(2<5)}).
\end{align*}
In particular, $\cm_{\gs,(2<5)}$ is generated by the polynomials:
\begin{eqnarray}\label{eq.M25}
\pl_{2,3}^{(2<5)} &=  z_{1,2}z_{6,3}-z_{2,3}z_{5,2}+z_{2,2}z_{5,3}-z_{3,3}z_{4,2}, \\ 
\nonumber \pl_{2,4}^{(2<5)}&= z_{1,2}z_{6,4}-z_{2,4}z_{5,2}+z_{2,2}z_{5,4}-z_{3,4}z_{4,2},\\ 
\pl_{3,4}^{(2<5)} &= z_{2,3}z_{6,4}-z_{2,4}z_{6,3}+z_{3,3}z_{5,4}-z_{3,4}z_{5,3}. \nonumber
\end{eqnarray}
Now $\ci_\gs = \cj_\gs+ \cm_{\gs, (1<6)}+ \cm_{\gs,(2<5)}$. 
\end{ex}

\begin{ex}\label{1component.b}
 Consider the following standard Young tableau and its associated standard noncrossing matching.
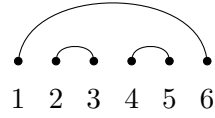
 \begin{multicols}{2}
   \begin{center}
        \begin{ytableau}
        1&2&4\\
        3&5&6
        \end{ytableau}
   \end{center} 
\vspace*{-.3in}
\begin{tikzpicture}[scale=0.5]
	\begin{scope}[shift={(8,0)}]
		\foreach \i in {0,...,5}
		{
			\node[pnt] at (\i,0)(\i){};
		}

		\draw (0)  to [bend left=80](5);
		\draw(1)  to [bend left=80] (2);
        \draw(3)  to [bend left=80] (4);
		\draw (0,-1) node {$1$};
     
        \draw (1,-1) node {$2$};
       
        \draw (2,-1) node {$3$};
        \draw (3,-1) node {$4$};
      
        \draw (4,-1) node {$5$};
        \draw (5,-1) node {$6$};
	\end{scope}
\end{tikzpicture} 
 \end{multicols}
\vspace*{-.2in}
\noindent The nesting sequence is $S_\gs = (1,2,2,3,3)$. Thus,
\begin{equation*}
  \cj_{\gs}=\langle z_{2,1}, z_{3,1}, z_{5,1},z_{6,1}, z_{3,2}, z_{6,2},z_{3,3} , z_{6,3}\rangle. 
\end{equation*}
There is only one cup with size greater than one, namely $(1<6)$ and
\begin{equation*}
M_{(1<6)}=\begin{bmatrix}
    z_{1,1}& z_{2,2}& z_{2,3} & z_{3,4} &z_{3,5}\\
    z_{4,1}& z_{5,2}& z_{5,3} & z_{6,4} &z_{6,5}
    \end{bmatrix}.
\end{equation*}
The ideal $\cm_{\gs,(1<6)}$ is generated by all $2\times2$ minors of $M_{(1<6)}$.  
We have  $\ci_\gs= \cj_\gs+ \cm_{\gs,(1<6)}$.
\end{ex}

\begin{ex}\label{2component} Let $n=12$, $\ell=6$ and consider the following standard tableau $\gs$ and corresponding standard noncrossing matching.
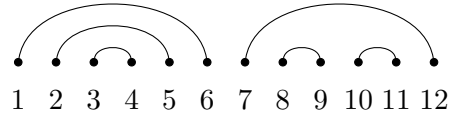
\begin{multicols}{2}
 \begin{center}
     \begin{ytableau}
        1&2&3&7&8&10\\
        4&5&6&9&11&12
    \end{ytableau}
 \end{center}  
\vspace*{-.3in}
  \begin{tikzpicture}[scale=0.5]\begin{scope}[shift={(8,0)}]
		\foreach \i in {0,...,11}
		{
			\node[pnt] at (\i,0)(\i){};
		}
		
        \draw(0)  to [bend left=80] (5);
		\draw(1)  to [bend left=80] (4);
		\draw(2)  to [bend left=80] (3);
		\draw (6)  to [bend left=80](11);
		\draw(7)  to [bend left=80] (8);
        \draw(9)  to [bend left=80] (10);
		\draw (0,-1) node {$1$};
       
        \draw (1,-1) node {$2$};
       
        \draw (2,-1) node {$3$};
        \draw (3,-1) node {$4$};
       
        \draw (4,-1) node {$5$};
        \draw (5,-1) node {$6$};
        \draw (6,-1) node {$7$};
        \draw (7,-1) node {$8$};
        \draw (8,-1) node {$9$};
        \draw (9,-1) node {$10$};
        \draw (10,-1) node {$11$};
        \draw (11,-1) node {$12$};
	\end{scope}
 \end{tikzpicture}
\end{multicols}
\vspace*{-.1in}
\noindent The cup sequence is $S(\gs) = (1,2,3,3,3,3,4,5,5,6,6)$. Note that $\cupsigma$ has two connected components. The cups $(1<6)$, $(2<5)$, and $(7<12)$ give us the matrices
\begin{equation*}
M_{(1<6)} = \begin{bmatrix} z_{1,1}& z_{2,2}& z_{3,3} & z_{3,4} &z_{3,5}\\ z_{7,1}& z_{8,2}& z_{9,3} & z_{9,4} &z_{9,5} \end{bmatrix}
\quad  \quad
M_{(2<5)} = \begin{bmatrix} z_{1,2}& z_{2,3}& z_{2,4} \\
z_{2,2}& z_{3,3}& z_{3,4}\\ z_{7,2} & z_{8,3} & z_{8,4}\\
z_{8,2}& z_{9,3}& z_{9,4} \end{bmatrix}
\end{equation*} 
    \quad \text{and}\quad
\begin{equation*}
 M_{(7<12)}=\begin{bmatrix} z_{4,7}& z_{5,8}& z_{5,9}& z_{6,10}& z_{6,11} \\ z_{10,7}& z_{11,8}& z_{11,9}& z_{12,10}& z_{12,11} 
\end{bmatrix}.
\end{equation*}
We also have
\begin{align*}
\cj_\gs&=\left\langle z_{2,1}, z_{3,1}, z_{3,2}, z_{8,1}, z_{9,1}, z_{9, 2} \right \rangle + \left\langle z_{5,7},z_{6,7}, z_{6,8}, z_{6,9}, z_{11,7},z_{12,7}, z_{12,8}, z_{12,9}\right\rangle \\
& \quad \quad +\left\langle z_{i,j} \mid i \in \{4,5,6,10,11,12\}, 1\leq j \leq 6 \right\rangle.
\end{align*}
Since $\ga(1<6)=\ga(7<12)=1$ the ideals $\cm_{\gs,(1<6)}$ and $\cm_{\gs,(7<12)}$ are generated by all $2\times 2$ minors of the matrices $M_{(1<6)}$ and $M_{(7<12)}$, respectively. Since $\ga(2<5) = 2$, $\cm_{\gs,(2<5)}$ is generated by a sum of the $2\times 2$ minors of $M_{(2<5)}$ as in~\eqref{eq.M25} of Example~\ref{1component.a}. 

\end{ex}

\subsection{Definition for arbitrary noncrossing matchings} It remains to define the ideal $\ci_\gs$ when $\gs$ is an arbitrary two-row tableau of shape $(n-\ell, \ell)$. Let $\gs=\gs_1\circ\gs_2\circ \cdots \circ \gs_k$ be the prime decomposition of $\gs$. Recall from Corollary~\ref{cupdecomposition} that each $\gs_i$ is either a prime two-row rectangular tableau or the unique prime tableau with a single box. If $\gs_i$ is a prime two-row rectangular tableau, then let $\ci_{\gs_i}$ be the ideal from Definition~\ref{def:MSigma} above. If $\gs_i$ is a prime tableau with a single box, then set $\ci_{\gs_i} = \{0\}$. 

For each $i$ with $1\leq i \leq k$, set $\mu_i = |\gs_i|$ and consider the composition $\mu=(\mu_1, \mu_2, \ldots, \mu_k)$ with partial sums $\parsum_i = \mu_1+\mu_2+\cdots +\mu_{i}$, $1\leq i \leq k$. By convention, we set $\mu_{0}=\parsum_0=0$. 

\begin{definition}\label{def.general} For each $i=1, 2, \ldots, k$ let $\ci_{\gs_i}'$ be the ideal obtained from $\ci_{\gs_i}$ by adding $\parsum_{i-1}$ to variable index, i.e.~$z_{a,b}\mapsto z_{a+\parsum_{i-1}, b+\parsum_{i-1}}$.  Set
\[
\ci_\gs' := \left( \sum_{i=1}^k \ci_{\gs_i}' \right)+ \mathcal{K}_\mu.
\]
where $\mathcal{K}_\mu:=\left< z_{i,j}\mid \parsum_{p-1}+1\leq j \leq \parsum_p<i \, \text{ for all }\, p= 1, \ldots, k-1 \right>$. We define
\begin{eqnarray}\label{eqn.gen.definition}
\ci_\gs := \tau_\mu^{-1} \ci_\gs'
\end{eqnarray}
to be the image of the ideal $\ci_\gs'$ under the action of the permutation $\tau_\mu^{-1}$ on $\C[\mathbf{z}]$ given by $z_{a,b}\mapsto z_{\tau_\mu^{-1}(a), b}$. 
\end{definition}

\begin{ex} \label{ex.2component.cont} Consider the tableau $\gs$ and corresponding cup diagram from Example~\ref{2component} above. We show that the ideal $\ci_\gs$ computed in that example aligns with the definition given above. Here $\gs$ has two connected components where $\gs_1$ is the tableau of size $6$ from Example~\ref{1component.a} and $\gs_2$ is the tableau of size $6$ from Example~\ref{1component.b}. We have $\tau_\mu =[1,2,3,7,8,9,4,5,6,10,11,12]$.  In this case, $|\gs_1|=|\gs_2|=6$ and 
\[
\ci_{\gs_1} = \cj_{\gs_1}+ \cm_{\gs_1,(1<6)}+\cm_{\gs_1, (2<5)},
\]
where each of $\cj_{\gs_1}$, $\cm_{\gs_1,(1<6)}$, and $\cm_{\gs_1, (2<5)}$ were computed in Example~\ref{1component.a}, and we have 
\[
\ci_{\gs_2} = \cj_{\gs_2}+ \cm_{\gs_2,(1<6)},
\]
where $\cj_{\gs_2}$ and $\cm_{\gs_2,(1<6)}$ were computed in Example~\ref{1component.b}. To compute $\ci_\gs'$, we begin by setting $\ci_{\gs_1}'=\ci_{\gs_1}$ and let $\ci_{\gs_2}'$ be the ideal $\ci_{\gs_2}$ with $\parsum_1=6$ added to all variable indices. For example, $\ci_{\gs_2}'$ contains all $2\times 2$ minors of the matrix
\[
M_{(1<6)}':=\begin{bmatrix} z_{7,7}& z_{8,8}& z_{8,9} & z_{9,10} &z_{9,11}\\ z_{10,7}& z_{11,8}& z_{11,9} & z_{12,10} &z_{12,11}
\end{bmatrix},
\]
which is the image of the matrix $M_{(1<6)}$ from Example~\ref{1component.b} under the map $z_{a,b}\mapsto z_{a+6,b+6}$. Now,
\[
\ci_\gs' = \ci_{\gs_1}'+\ci_{\gs_2}' + \left< z_{i,j} \mid 7\leq i \leq 12, 1 \leq j \leq 6\right>.
\]
Translating by $\tau_\mu^{-1} = [1,2,3,7,8,9,4,5,6,10,11,12]$ gives us 
\begin{align*}
\ci_\gs &= \tau_\mu^{-1}\ci_{\gs_1}'+ \tau_\mu^{-1}\ci_{\gs_2}' + \tau_\mu^{-1}\left< z_{i,j} \mid 7\leq i \leq 12, 1 \leq j \leq 6\right>\\
&= \tau_\mu^{-1}\ci_{\gs_1}'+ \tau_\mu^{-1}\ci_{\gs_2}' +\left< z_{i,j} \mid i\in \{4,5,6,10,11,12\}, 1 \leq j \leq 6\right>.
\end{align*}
It is easy to confirm using Example~\ref{2component} that
\[
\tau_\mu^{-1}\ci_{\gs_1}' = \cm_{\gs,(1<6)}+\cm_{\gs, (2<5)}+ \langle z_{2,1}, z_{3,1}, z_{3,2}, z_{8,1}, z_{9,1}, z_{9,2} \rangle
\]
and 
\[
\tau_\mu^{-1}\ci_{\gs_2}' = \cm_{\gs,(7<12)}+ \langle z_{5,7}, z_{6,7}, z_{6,8}, z_{6,9}, z_{11,7}, z_{12,7}, z_{12, 8}, z_{12,9} \rangle.
\]
Thus, the ideal $\ci_\gs$ is indeed same one computed in Example~\ref{2component}. For example, $\tau_\mu^{-1}\cdot\ci_{\gs_2}'$ contains the $2\times 2$ minors of the matrix 
\[
\tau_\mu^{-1}  M_{(1<6)}' = \begin{bmatrix} z_{4,7}& z_{5,8}& z_{5,9} & z_{6,10} &z_{6,11}\\ z_{10,7}& z_{11,8}& z_{11,9} & z_{12,10} &z_{12,11}
\end{bmatrix}
\]
which is the matrix $M_{(7<12)}$ used to compute the ideal $\cm_{\gs,(7<12)}$ in Example~\ref{2component}.
\end{ex}

\begin{remark} When $\gs$ is a rectangular tableau which is not prime (as in Examples~\ref{2component} and~\ref{ex.2component.cont}) it is straightforward to check that the ideal $\ci_{\gs}$ from Definition~\ref{def:MSigma} and that from Definition~\ref{def.general} coincide. In fact, both define the ideal $\ci_\gs$ using the same generating set.
\end{remark}

\begin{ex}\label{ex.disconnected3} Let $\gs$ be the tableau from Example~\ref{ex.disconnected1}, pictured below.
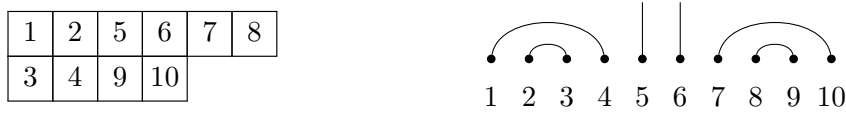
\begin{multicols}{2}
\begin{center} \begin{ytableau}1&2&5&6&7 & 8\\3&4&9 & 10 \end{ytableau}
\end{center} 
\vspace*{-.2in}
\begin{tikzpicture}[scale=0.5] \begin{scope}[shift={(8,0)}]
\foreach \i in {0,...,9}{\node[pnt] at (\i,0)(\i){};}
\draw (0)  to [bend left=80](3);
\draw(1)  to [bend left=80] (2);
\draw (4,0) -- (4,1.5);
\draw (5,0) -- (5,1.5);
\draw (6)  to [bend left=80](9);
\draw(7)  to [bend left=80] (8);
     
\draw (0,-1) node {$1$};
\draw (1,-1) node {$2$};
\draw (2,-1) node {$3$};
\draw (3,-1) node {$4$};
\draw (4,-1) node {$5$};
\draw (5,-1) node {$6$};
\draw (6,-1) node {$7$};
\draw (7,-1) node {$8$};
\draw (8,-1) node {$9$};
 \draw (9,-1) node {$10$};
\end{scope}
\end{tikzpicture} 
\end{multicols}
\noindent The tableau and matching decomposes into four components, $\gs_1$, $\gs_2$, $\gs_3$, $\gs_4$. The components $\gs_2$ and $\gs_3$ are each equal to the prime tableau with a single box. We have that $\gs_1=\gs_4$ are both equal to the following.
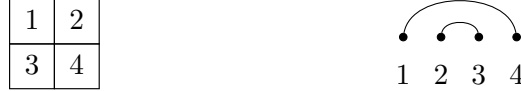
\begin{multicols}{2}
\begin{center} \begin{ytableau}1&2\\3&4 \end{ytableau}
\end{center} 
\vspace*{-.2in}
\begin{tikzpicture}[scale=0.5] \begin{scope}[shift={(8,0)}]
\foreach \i in {0,...,3}{\node[pnt] at (\i,0)(\i){};}
\draw (0)  to [bend left=80](3);
\draw(1)  to [bend left=80] (2);
     
\draw (0,-1) node {$1$};
\draw (1,-1) node {$2$};
\draw (2,-1) node {$3$};
\draw (3,-1) node {$4$};
\end{scope}
\end{tikzpicture} 
\end{multicols}
\vspace*{-.1in}
\noindent We have that 
\[
\ci_{\gs_1}=\ci_{\gs_4} = \langle z_{2,1}, z_{4,1} \rangle + \mathrm{minors}\left(2, \begin{bmatrix} z_{1,1} & z_{2,2} & z_{2,3}\\ z_{3,1} & z_{4,2} & z_{4,3} \end{bmatrix} \right).
\]
Now, $\ci_{\gs_1}'=\ci_{\gs_1}$, $\ci_{\gs_2}=\ci_{\gs_3} = \{0\}$, and 
\[
\ci_{\gs_4}' = \langle z_{8,7}, z_{10,7} \rangle + \mathrm{minors}\left(2, \begin{bmatrix} z_{7,7} & z_{8,8} & z_{8,9}\\ z_{9,7} & z_{10,8} & z_{10,9} \end{bmatrix} \right).
\]
We have 
\begin{align*}
\ci_\gs' &= \ci_{\gs_1}' + \ci_{\gs_4}' + \langle z_{i,j} \mid 5\leq i \leq 10, 1\leq j \leq 4 \rangle \\
& \quad\quad\quad\quad \qquad\qquad + \langle z_{i,5} \mid 6\leq i \leq 10  \rangle+\langle z_{i,6} \mid 7\leq i \leq 10 \rangle.
\end{align*}
Finally, we have $\tau_\mu^{-1} = [1,2,7,8,3,4,5,6,9,10]$ and
\begin{align*}
\ci_\gs = \tau_\mu^{-1} \ci_{\gs}' &= \tau_\mu^{-1}\ci_{\gs_1}' + \tau_\mu^{-1}\ci_{\gs_4}' + \langle z_{i,j} \mid i\in \{3,4,5,6,9,10\}, 1\leq j \leq 4 \rangle \\
& \quad\quad\quad\quad \qquad\qquad + \langle z_{i,5} \mid i\in \{4,5,6,9,10\}  \rangle+\langle z_{i,6} \mid i\in \{5,6,9,10\} \rangle.
\end{align*}
where 
\[
 \tau_\mu^{-1}\ci_{\gs_1}' = \langle z_{2,1}, z_{8,1} \rangle + \mathrm{minors}\left(2, \begin{bmatrix} z_{1,1} & z_{2,2} & z_{2,3}\\ z_{7,1} & z_{8,2} & z_{8,3} \end{bmatrix} \right)
\]
and
\[
 \tau_\mu^{-1}\ci_{\gs_4}' = \langle z_{6,7}, z_{10,7} \rangle + \mathrm{minors}\left(2, \begin{bmatrix} z_{5,7} & z_{6,8} & z_{6,9}\\ z_{9,7} & z_{10,8} & z_{10,9} \end{bmatrix} \right).
\]
The reader may wish to compare the ideal
\begin{align*}
\mathcal{K}_\mu &= \langle z_{i,j} \mid 5\leq i \leq 10, 1\leq j \leq 4 \rangle  + \langle z_{i,5} \mid 6\leq i \leq 10  \rangle \\
 &\qquad \qquad \qquad +\langle z_{i,6} \mid 7\leq i \leq 10 \rangle
\end{align*}
with the matrices $\dot\tau_\mu \cell$ computed in Example~\ref{ex.disconnected2}. The ideal $\mathcal{K}_\mu$ is the prime monomial ideal that forces any matrix in its vanishing set into block upper triangular form, with blocks of sizes $\mu = (4,1,1,4)$. 
\end{ex}

Our computations in Macaulay2 point to the fact that the ideal $\isigma$ is prime for all two-row tableaux. 
\begin{conj}
Let $\ell$ be a positive integer with $1\leq \ell \leq \lfloor \frac{n}{2}\rfloor$ and  $\gs\in \SYT(n-\ell, \ell)$. The ideal $\ci_\gs$ is prime.
\end{conj}
We prove the conjecture for a special case in Lemma~\ref{lemma.Grobner} below.

%%%%%%%%%%%%%%%%%%%%%%%%%%%%%%%%%%

\subsection{Statement of the main theorem} 

With the definition of the ideal $\ci_\gs$ in hand, we are ready to state the main theorem of this paper.
Let $N$ be a nilpotent matrix of Jordan type $\lambda = (n-\ell, \ell)$ and suppose $N$ is in Jordan canonical form; see~\eqref{eqn.Jbasis}. 
We consider the action of the subgroup $B$ of upper triangular matrices in $GL_n(\C)$ on $\C[\mathbf{z}]$ induced by right multiplication. Let $\pi:GL_n(\C)\to GL_n(\C)/B$ denote the projection map. The main theorem of this paper is: 

\begin{theorem}\label{thm.ideal} Let $\ell$ be a positive integer with $1\leq \ell \leq \lfloor \frac{n}{2}\rfloor$ and $\gs\in \SYT(n-\ell, \ell)$. The ideal $\ci_\gs$ is right $B$-invariant and $\pi^{-1}(\cb_\sigma) = \cv(\ci_\gs)\cap GL_n(\C)$. 
\end{theorem}

This is a restatement of Theorem~\ref{thm.intro} from the Introduction. The proof consists of two main parts. We prove $B$-invariance in Section~\ref{sec.binv}. We then study the vanishing set of $\ci_\gs$ in Section~\ref{sec.vanishingset}.

%%%%%%%%%%%%%%%%%%%%%%%%%%%%%%%%%%%%

\section{$B$-invariance of $\ci_\gs$}\label{sec.binv}
In this section, we take our first step in the proof of Theorem~\ref{thm.ideal}. We show that $\ci_\gs$ is invariant under the action of $B$ on $\C[\mathbf{z}]$ induced by right multiplication of matrices.

Recall that $B=TU\subset GL_n(\C)$ where $T$ is the subgroup of diagonal matrices and $U$ is generated by the matrices
\begin{eqnarray}\label{gen}
b_{q,r} = I_n+cE_{q,r} ,\; 1\leq q<r \leq n,\, \text{ and }\, c\in \C
\end{eqnarray}
where $E_{q,r}$ is the $n\times n$ elementary matrix with unique nonzero entry equal to $1$ in row $q$ and column $r$. 
The group $B$ acts on $M_n(\C)$ by right multiplication and there is an induced action on $\C[\mathbf{z}]$. This action multiplies columns by any nonzero scalar and adds any scalar multiple of a column to `later' ones.  For example, the action of the matrix from~\eqref{gen} on the $r$-th column~is 
\[
\begin{bmatrix} z_{1,r}\\z_{2,r} \\ \vdots \\ z_{n,r} \end{bmatrix} \cdot b_{q,r} = \begin{bmatrix}  z_{1,r} + cz_{1,q}\\ z_{2,r} + cz_{2,q} \\ \vdots \\ z_{n,r} + cz_{n,q} \end{bmatrix}.
\]
We begin with a lemma which simplifies our computations later.

\begin{lemma}\label{pairingdeterminants} For any positive integer $a\geq 2$ and set of indeterminates $\{v_{i,j}^k \mid i,j\in \{1,2\} \text{ and } 0\leq k \leq a \}$, we have
\begin{eqnarray}\label{lem:ident1}
v_{1,1}^0v_{2,2}^0+\displaystyle\sum_{k=1}^{a-1}\begin{vmatrix} v_{1,1}^k&v_{1,2}^k\\v_{2,1}^k&v_{2,2}^k\end{vmatrix}-v_{2,1}^av_{1,2}^a
=\sum_{k=0}^{a-1}\begin{vmatrix}
v_{1,1}^k&v_{1,2}^{k+1}\\
v_{2,1}^{k+1}&v_{2,2}^k
\end{vmatrix}.
\end{eqnarray}
\end{lemma}
\begin{proof} We use induction on $a$ to show the identity~\eqref{lem:ident1}. When $a=2$,
\begin{align*}
        v_{1,1}^0v_{2,2}^0+\begin{vmatrix}
        v_{1,1}^1&v_{1,2}^1\\
        v_{2,1}^1&v_{2,2}^1
    \end{vmatrix}-v_{1,2}^2v_{2,2}^2&=v_{1,1}^0v_{2,2}^0+v_{1,1}^1v_{2,2}^1-v_{1,2}^1v_{2,1}^1-v_{1,2}^2v_{2,1}^2\\
    &=\begin{vmatrix}
         v_{1,1}^0&v_{1,2}^1\\
        v_{2,1}^1&v_{2,2}^0
    \end{vmatrix}+\begin{vmatrix}
         v_{1,1}^1&v_{1,2}^2\\
        v_{2,1}^2&v_{2,2}^1
    \end{vmatrix}.
\end{align*}
Now suppose the identity~\eqref{lem:ident1} holds for some $a$ with $a\geq 2$. Then,

\begin{align*}
v_{1,1}^0v_{2,2}^0+&\sum_{k=1}^a\begin{vmatrix}
        v_{1,1}^k&v_{1,2}^k\\
        v_{2,1}^k&v_{2,2}^k
    \end{vmatrix}-v_{1,2}^{a+1}v_{2,1}^{a+1}\\
&= v_{1,1}^0v_{2,2}^0+\sum_{k=1}^{a-1}\begin{vmatrix} v_{1,1}^k&v_{1,2}^k\\ v_{2,1}^k&v_{2,2}^k \end{vmatrix} + v_{1,1}^av_{2,2}^a-v_{1,2}^av_{2,1}^a - v_{2,1}^{a+1}v_{1,2}^{a+1} \\
&=\sum_{k=0}^{a-1}\begin{vmatrix}
        v_{1,1}^k&v_{1,2}^{k+1}\\
        v_{2,1}^{k+1}&v_{2,2}^k
    \end{vmatrix}+v_{1,1}^av_{2,2}^a-v_{1,2}^{a+1}v_{2,1}^{a+1} \quad \text{ (by the induction hypothesis)}\\
    &=\displaystyle\sum_{k=0}^{a}\begin{vmatrix}
        v_{1,1}^k&v_{1,2}^{k+1}\\
        v_{2,1}^{k+1}&v_{2,2}^k
    \end{vmatrix}.\end{align*}
Thus, the identity \eqref{lem:ident1} holds by induction. 
\end{proof}

The next lemma reduces the proof of $B$-invariance to the rectangular case.

\begin{lemma} Suppose $\ci_\gs$ is $B$-invariant for all $\gs\in \SYT(\ell, \ell)$. Then $\ci_\gs$ is $B$-invariant for all $\gs\in \SYT(n-\ell, \ell)$. 
\end{lemma}
\begin{proof} Let $\gs = \gs_1\circ\gs_2\circ \cdots \circ \gs_k$ be the prime decomposition of $\gs$. Recall that $\mu_i = |\gs_i|$ and $\parsum_{i} = \mu_1+\cdots +\mu_i$, with the convention that $\parsum_{0} = 0$.   As the action of $\tau_\mu$ on $M_n$ by left multiplication by $\dot \tau _\mu$ permutes rows, but not columns, it suffices to prove that the ideal $\ci_{\gs}' = \mathcal{K}_\mu+ \sum_{i=1}^k \ci_{\gs_i}'$ from Definition~\ref{def.general} is right $B$-invariant, where

\[
\mathcal{K}_\mu:= \left< z_{i,j}\mid \parsum_{p-1}+1\leq j \leq \parsum_p<i \, \text{ for all }\, p= 1, \ldots ,k-1 \right>.
\]
A matrix belongs to $\cv(\mathcal{K}_\mu)$ if and only if it is block upper-triangular with blocks of size $\mu_1, \mu_2, \ldots, \mu_k$. Right multiplication by any upper-triangular matrix preserves this block upper-triangular form, so $\mathcal{K}_\mu$ is $B$-invariant.

Consider the decomposition $B = B'U'$ where $B'$ is the subgroup of $B$ of block diagonal matrices with blocks of sizes $\mu_1$, $\mu_2, \ldots, \mu_k$ and $U'$ is the subgroup of strictly block upper triangular matrices relative to the same block sizes. 
In other words, $B'$ is the Borel subgroup of the Levi subgroup $\iota(GL_{\mu_1}(\C)\times \cdots \times GL_{\mu_k}(\C))$ and $U'$ is the unipotent radical of the parabolic subgroup of $GL_n(\C)$ corresponding to this Levi subgroup.

For each $i$ with $\mu_i\geq 2$, the ideal $\ci_{\gs_i}'$ is the image of the ideal $\ci_{\gs_i}$ corresponding to the rectangular tableau $\gs_i$ under the map which shifts every index by $\parsum_{i-1}$. Thus each $\ci_{\gs_i}'$ is an ideal in the polynomial ring with indeterminates $z_{a,r}$ where $a$ and $r$ belong to the block $\{\parsum_{i-1}+1,\parsum_{i-1}+2, \ldots, \parsum_{i} \}$. 
Each  $A\in \cv(\ci_\gs')$ must be block upper triangular, since $A\in \cv(\mathcal{K}_\mu)$. Right multiplication by $U'$ only affects entries of $A$ in the strictly upper triangular blocks.  Thus, $U'$ acts trivially on $\sum_{i=1}^k \ci_{\gs_i}'$.  The assumption that each $\ci_{\gs_i}$ is $B$-invariant (where $B \subseteq GL_{\mu_i}(\C)$) implies $\sum_{i=1}^k \ci_{\gs_i}'$ is $B'$-invariant. We conclude $\ci_\gs'$ is right $B$-invariant, as desired.
\end{proof}

We are now ready to state and prove the main result of this section:

\begin{prop}\label{binv} Let $\gs\in \SYT(\ell, \ell)$. The ideal $\ci_\gs$ is right $B$-invariant. 
\end{prop}

\begin{proof} To prove $\ci_\gs$ is stable under the action of $B=TU$, we argue that 
\begin{enumerate}
\item each generator of $\ci_\gs$ is fixed under the rescaling action of right multiplication by $\mathrm{diag}(t_1,t_2, \ldots, t_n)\in T$ by $z_{a,r} \mapsto t_r z_{a,r}$ for each $r\in [n]$, and 
\item the image of each generator of $\isigma$ under the action of $b_{q,r}$ for $1\leq q<r\leq n$ as in~\eqref{gen}, given by $z_{a,r}\mapsto z_{a,r}+cz_{a,q}$ for all $1\leq a \leq n$ and $c\in \C$, is an element of the ideal $\ci_\gs$.
\end{enumerate}

Recall from Definition~\ref{def:MSigma} that $\ci_\gs = \cj_\gs+\cm_\gs$. First, we show that each generator of $\ci_\gs$ is invariant under the re-scaling action, as in (1). Since $\cj_\gs$ is a monomial ideal, invariance under re-scaling is obvious.

Recall that 
\[
\cm_\gs = \sum_{\substack{(i<j)\\ \gd(i<j)\geq 2}} \cm_{\gs, (i<j)}
\]
where $\cm_{\gs, (i<j)}$ is as in Definition~\ref{def:MSigma}.
Given a cup $(i<j)$ in $\cupsigma$ of size at least $2$ and $c_1,\,c_2$ with $i\leq c_1<c_2\leq j-1$ recall that
\[
\pl_{c_1,c_2}^{(i<j)}= \sum_{k=0}^{\al(i<j)}\pl_{\{k+1, 2\al(i<j)+2-k\}, \{c_1-i+1,c_2-i+1\}} (M_{(i<j)})
\]
denotes a generator of $\cm_{\gs,(i<j)}$. For notational simplicity, write $\al=\al(i<j)$ for the remainder of the proof. Recall that the columns of the matrix $M_{(i<j)}$ are indexed by $i,i+1,\dots, j-1$ and each generator selects two columns and takes a particular sum of $2\times 2$ minors in these columns. We write $\left\{c_1, c_2\right\}\subseteq \left\{i, i+1, \dots, j-1\right\}$ for the indices of these columns. By definition,
\begin{align}\label{eqn.sum}
\pl_{c_1,c_2}^{(i<j)}= \sum_{k=0}^{\al} \begin{vmatrix}
    z_{S_\sigma(c_1)-\al+k, c_1}&z_{S_\sigma(c_2)-\al+k, c_2}\\
    z_{S_\sigma(c_1)+\ell-k, c_1} &  z_{S_\sigma(c_2)+\ell-k, c_2}
\end{vmatrix} \ .
\end{align}
We have
\begin{align*}
\pl_{c_1,c_2}^{(i<j)} \cdot \mathrm{diag}(t_1, t_2, \ldots, t_n) &=\displaystyle\sum_{k=0}^{\al}\begin{vmatrix}
          t_{c_1}z_{S_\sigma(c_1)-\al+k, c_1}& t_{c_2}z_{S_\sigma(c_2)-\al+k, c_2}\\
    t_{c_1} z_{S_\gs(c_1)+\ell-k, c_1} &  t_{c_2} z_{S_\gs(c_2)+\ell-k, c_2}   
        \end{vmatrix}\\
        &=t_{c_1}t_{c_2}\, \displaystyle\sum_{k=0}^{\al} \begin{vmatrix}
          z_{S_\sigma(c_1)-\al+k, c_1}& z_{S_\sigma(c_2)-\al+k, c_2}\\ z_{S_\gs(c_1)+\ell-k, c_1} &   z_{S_\gs(c_2)+\ell-k, c_2}   
        \end{vmatrix} \\
        &= t_{c_1}t_{c_2}\, \pl_{c_1,c_2}^{(i<j)}\in\cm_{\gs,(i<j)}
\end{align*}
by multi-linearity.
Thus $\ci_\gs$ is invariant under rescaling.

Next we argue invariance under the action from (2). We treat the cases of generators in the ideals $\cj_\gs$ and $\cm_{(i<j)}$ separately. 

Let $z_{a,r}$ be a generator of $\cj_\gs$. Let $b_{q,r}=I_n+cE_{q,r}$ as in~\eqref{gen}. We want to show that for all $1\leq q<r \leq n$,  
\[
z_{a,r}\cdot b_{q,r}= z_{a,r}+cz_{a,q}\in \cj_\gs.
\]
By definition of the ideal $\cj_\gs$ we have that $S_\gs(r)<\ell$ and either  $S_\gs(r)+1\leq a \leq \ell$  or  $S_\gs(r)+\ell+1\leq a \leq n$. Because the nesting sequence of $\gs$ is weakly increasing it must be that $S_\gs(q)\leq S_\gs(r)<\ell$ and, therefore, $S_\gs(q)+1\leq a\leq \ell$ or $S_\gs(q)+\ell+1\leq a\leq n$.  Thus, $z_{a,q} \in\cj_\gs$ and $z_{ar}+cz_{aq}\in \cj_\gs$.

To prove that $\pl_{c_1,c_2}^{(i<j)}\cdot b_{q,r}\in \ci_\gs$, we consider two possibilities: either $r=c_1$ or $r=c_2$. (If $r\neq c_1$ and $r\neq c_2$ then the action of $b_{q,r}$ fixes all variables appearing in the sum~\eqref{eqn.sum} and there is nothing to show.) Note that exchanging the columns $c_1$ and $c_2$ in the definition of the generator~\eqref{eqn.sum} has the effect of exchanging the columns in each $2\times 2$ minor $\pl_{\{k+1, 2\al + 2-k\},\{c_1,c_2\}}(M_{(i<j)})$, which results in $-\pl_{c_1,c_2}^{(i<j)}$ so, up to symmetry, it suffices to consider $r=c_1$.

Suppose from now on that $r=c_1$. This implies $i\leq r \leq j-1$ since $r$ must index a column of $M_{(i<j)}$. By multi-linearity,
\begin{align*}
  \pl_{r,c_2}^{(i<j)}\cdot b_{q,r} &= \sum_{k=0}^{\al}\begin{vmatrix}
    z_{S_\sigma(r)-\al+k, r}+ cz_{S_\sigma(r)-\al+k, q}&z_{S_\sigma(c_2)-\al+k, c_2}\\
 z_{S_\sigma(r)+\ell-k, r}+ cz_{S_\sigma(r)+\ell-k, q} &  z_{S_\sigma(c_2)+\ell-k, c_2}
\end{vmatrix} \\
   &=  \pl_{ r,c_2}^{(i<j)}+c\,\underbrace{\displaystyle\sum_{k=0}^{\al}\begin{vmatrix}
    z_{S_\sigma(r)-\al+k, q}&z_{S_\sigma(c_2)-\al+k, c_2}\\
    z_{S_\gs(r)+\ell-k, q} &  z_{S_\gs(c_2)+\ell-k, c_2}
\end{vmatrix}}_{(*)}
\end{align*}
Since $\pl_{r,c_2}^{(i<j)} = \pl_{c_1,c_2}^{(i<j)}\in\cm_{\gs,(i<j)}$, it remains to show that the expression $(*)$ is an element of $\ci_\sigma$. We consider two cases: either $S_\gs(q)=S_\gs(r)$ or $S_\gs(q)<S_\gs(r)$.

If $S_\gs(q)=S_{\gs}(r)$ then it follows immediately that $i\leq q$. Indeed, since $i$ is the left endpoint of a cup, if $q<i$ then $S_\gs (r) = S_\gs(q)<S_\gs(i)$, which contradicts the fact that $i\leq r$ and the nesting sequence is weakly increasing.  As $i\leq q$, the set $\{q ,c_2\}$ is a valid choice of columns of $M_{(i<j)}$. In this case we get
\begin{align*}
(*)&=\sum_{k=0}^{\al}\begin{vmatrix}
          z_{S_\sigma(q)-\al+k, q}&z_{S_\sigma(c_2)-\al+k, c_2}\\
    z_{S_\gs(q)+\ell-k, q} &  z_{S_\gs(c_2)+\ell-k, c_2}   
        \end{vmatrix}\\
&=\sum_{k=0}^{\al} \pl_{\{k+1, 2 \al +2-k\},\{q,c_2\}}(M_{(i<j)})=\pl_{q,c_2}^{(i<j)}\in\cm_{\gs,(i<j)}.
\end{align*}
Thus, $(*)\in \ci_\gs$ and we are done in this case.

Now suppose $S_\gs(q)<S_{\gs}(r)$. We prove $(*)\in \ci_\gs$ by induction on $\al = \al(i<j)\geq 0$, the number of cups strictly above the cup containing $(i<j)$ in $\cupsigma$. Since $S_\gs(q)<S_\gs(r) \leq \ell$, it follows that $z_{S_\gs(r), q}$ and $z_{S_{\gs}(r)+\ell, q}$ are elements of $\cj_\gs$. These variables appear in the first and last minor (for $k=0$ and $k=\al$) of the sum $(*)$.  

If $\alpha=0$, that is, if $(i<j)$ has no cup appearing above it in $\cupsigma$, then we get 
\begin{align*}
(*) = \begin{vmatrix} z_{S_\gs(r),q} & z_{S_\gs(c_2),c_2} \\ z_{S_\gs(r)+\ell, q} & z_{S_\gs(c_2) + \ell , c_2} \end{vmatrix}  \in \cj_\sigma,
\end{align*}
proving the base case.

Assume from now on that $\al\geq 1$ so there is some cup $(i^*< j^*)$ directly above $(i<j)$ in $\cupsigma$ and let $\al^*= \al(i^*<j^*)=\al-1$. Modulo the ideal $\cj_\gs$, we get
\begin{align*}
(*)\; \mathrm{mod}\; \cj_\gs =
z_{S_\sigma(r)-\al, q}\, z_{S_\gs(c_2)+\ell, c_2}+\displaystyle\sum_{k=1}^{\al-1}&\begin{vmatrix}
z_{S_\sigma(r)-\al+k, q}&z_{S_\sigma(c_2)-\al+k, c_2}\\
z_{S_\gs(r)+\ell-k, q} &  z_{S_\gs(c_2)+\ell-k, c_2}
\end{vmatrix}\\
&\quad\quad\quad\quad -z_{S_\gs(r)+\ell-\al, q}\, z_{S_\sigma(c_2), c_2}.
\end{align*}
Using Lemma \ref{pairingdeterminants} and re-writing the result in terms of $\al^*$ yields
\begin{align}\label{eqn.nextcup}
(*)\; \mathrm{mod}\; \cj_\gs &= \sum_{k=0}^{\al-1}\begin{vmatrix}
z_{S_\sigma(r)-\al+k, q}&z_{S_\sigma(c_2)-\al+k+1, c_2}\\
z_{S_\gs(r)+\ell-k-1, q} &  z_{S_\gs(c_2)+\ell-k, c_2}
\end{vmatrix}\\ 
&= \sum_{k=0}^{\alpha^*} \begin{vmatrix}
z_{(S_\sigma(r)-1)-\al^*+k, q}&z_{S_\sigma(c_2)-\al^*+k, c_2}\\ \nonumber
z_{(S_\gs(r)-1)+\ell-k, q} &  z_{S_\gs(c_2)+\ell-k, c_2}
\end{vmatrix}.
\end{align}

If $S_\gs(q)=S_\gs(r)-1$, then $S_\gs(q) = S_{\gs}(r)-1 \geq S_{\gs}(i)-1 \geq  S_{\gs}(i^*)$. This implies $i^*\leq q$ and thus $\{q,c_2\}$ is a valid choice of columns in the matrix $M_{(i^*<j^*)}$ since $i^*\leq q < c_2\leq j^*$ and~\eqref{eqn.nextcup} becomes
\begin{align*}
(*)\; \mathrm{mod}\; \cj_\gs &= \sum_{k=0}^{\al^*} \begin{vmatrix}
z_{S_\gs(q)-\al^*+k, q}&z_{S_\sigma(c_2)-\al^*+k, c_2}\\
z_{S_\gs(q)+\ell-k, q} &  z_{S_\gs(c_2)+\ell-k, c_2}
\end{vmatrix}\\ &= \sum_{k=0}^{\al^*} \pl_{\{k+1, 2\al^* +2-k\},\{q,c_2\}} (M_{(i^*<j^*)}) = \pl_{q,c_2}^{(i^*<j^*)} \in \cm_{\gs,(i^*<j^*)}.
\end{align*}
If $S_\gs(q)\neq S_\gs(r)-1$, let $r'>q$ denote the minimum value such that $S_\gs(r') = S_\gs(r)-1$. Then $S_\gs(q)<S_\gs(r')$ and~\eqref{eqn.nextcup} becomes
\begin{align*}
(*)\; \mathrm{mod}\; \cj_\gs = (*)' \; \text{ where } \;  (*)' := \sum_{k=0}^{\alpha^*} \begin{vmatrix}
z_{S_\gs(r')-\al^*+k, q}&z_{S_\sigma(c_2)-\al^*+k, c_2}\\
z_{S_\gs(r')+\ell-k, q} &  z_{S_\gs(c_2)+\ell-k, c_2}
\end{vmatrix} .
\end{align*}
Since $\al^* < \al$ and $c_2$ indexes a column of $M_{(i^*<j^*)}$ we have by induction that $(*)'\in \ci_\gs$.   We conclude $(*)\in \ci_\gs$, as desired. This shows that $\ci_\gs$ is right $B$-invariant.
\end{proof}

%%%%%%%%%%%%%%%%%%%%%%%%%%%

\section{The vanishing set of \texorpdfstring{$\ci_\gs$}{ci	extunderscore gs}}\label{sec.vanishingset}

We now work to prove the second part of Theorem~\ref{thm.ideal}.

\subsection{The containment $\cv(\ci_\gs) \cap GL_n(\C) \subseteq \pi^{-1}(\cb_\gs)$}

Throughout this section, $A$ denotes an $n\times n$ matrix,
\[
A = \begin{bmatrix} \; | & | & & | \; \\ a_1 & a_2 & \cdots & a_n \\
| & | & & | \end{bmatrix},
\]
where $a_1, a_2, \ldots, a_n\in M_{n,1}(\C)$ are the columns of $A$, with $a_k^T:= [a_{1,k}\;a_{2,k}\; \cdots\; a_{n,k}]$ for each $1\leq k\leq n$. We let
\[
F_k := \C\{a_1, a_2, \ldots, a_k\}
\]
denote the span of the first $k$ columns of $A$.

Our first lemma reduces our study of the vanishing set of $\ci_\gs$ to the case where $\gs$ is prime. 

\begin{lemma}\label{lem.reduction} Let $\gs\in \SYT(n-\ell,\ell)$ with prime decomposition $\gs = \gs_1\circ \gs_2 \circ \cdots \circ \gs_k$  and associated composition $\mu=(\mu_1, \mu_2, \ldots, \mu_k)$ with $\mu_i := |\gs_i|$. Set $\parsum_i = \mu_1+\mu_2+\cdots +\mu_i$.
Let $A\in \cv(\ci_\gs)\cap GL_n(\C)$. 
\begin{enumerate}
\item For all $1\leq i \leq k$, 
\[
F_{\parsum_i} = \C\{e_1, e_2, \ldots, e_{r(\parsum_i)+c(s_i)}\} \oplus \C\{e_{n-\ell+1}, e_{n-\ell+2},, \ldots, e_{n-\ell+c(\parsum_i)}\}.
\]
\item If $\mu_i\geq 2$, that is, if $\gs_i$ is a rectangular prime tableau, then 
\[
N^{\mu_i/2} (F_{\parsum_i}) = F_{\parsum_{i-1}} ,
\]
and if $N_i$ denotes the nilpotent matrix obtained by restricting $N$ to the quotient space $F_{\parsum_i}/F_{\parsum_{i-1}}$, then $N_i$ has Jordan form $(\frac{\mu_i}{2},\frac{\mu_i}{2})$ and is in Jordan canonical form with respect to the standard basis. 
\end{enumerate}
In particular, if $\cv(\ci_{\gs_i})\cap GL_{m_i}(\C) \subseteq \pi^{-1}(\SF_{\gs_i})$ for all $i$, then $\cv(\ci_\gs)\cap GL_n(\C)\subseteq  \pi^{-1}(\SF_\gs)$. 
\end{lemma}

\begin{proof} Suppose $A\in \mathcal{V}(\isigma)$, or equivalently, $\dot \tau_\mu A\in \mathcal{\isigma'}$. 
The monomial ideal 
\[
\mathcal{K}_\mu = \left< z_{i,j} \mid \parsum_{p-1}\leq j \leq \parsum_p <i , \text{ for all } p=1, \ldots, k-1 \right>
\]
in $\ci_{\gs}'$ forces the matrix $\dot\tau_\mu A$ to be block upper-triangular with blocks of size $(\mu_1, \mu_2, \ldots, \mu_k)$. Since $\dot\tau_\mu A$ is invertible, this implies
\[
\dot \tau_\mu (F_{\parsum_i}) = \C\{e_1, e_2, \ldots, e_{\parsum_i }\}
\]
for each $1\leq i \leq k$. By Lemma~\ref{lemma.tau_gs.formula} 
\[
F_{\parsum_i} = \C\{e_1, e_2, \ldots, e_{r(\parsum_i)+c(\parsum_i)}\}\oplus \C\{e_{n-\ell + 1}, e_{n-\ell+2}, \ldots, e_{n-\ell+c(\parsum_i)}\}
\]
for each $1\leq i \leq k$. This proves (1). 

When $\mu_i\geq 2$, we get that $r(\parsum_i)+c(\parsum_i) = r(\parsum_{i-1}) +c(\parsum_{i-1})+\frac{\mu_i}{2}$ since there are exactly $\frac{\mu_i}{2}$ cups in $\gs_i$. In particular,
\begin{align*}
F_{\parsum_i}/F_{\parsum_{i-1}} \simeq &\; \C\{e_{r(\parsum_{i-1})+c(\parsum_{i-1}) +1}, e_{r(\parsum_{i-1})+c(\parsum_{i-1}) +2},\ldots,  e_{r(\parsum_{i-1})+c(\parsum_{i-1}) +\frac{\mu_i}{2}}  \}\\ 
&\quad \qquad \qquad  \oplus \C\{ e_{n-\ell +c(\parsum_{i-1})+1}, e_{n-\ell +c(\parsum_{i-1})+2}, \ldots, e_{n-\ell +c(\parsum_{i-1})+\frac{\mu_i}{2}} \}.
\end{align*}
Since $N$ has fixed Jordan form as~\eqref{eqn.Jbasis} above, $N(F_{\parsum_i}) = F_{\parsum_{i-1}}$ and $N_i$ is in standard Jordan form with blocks determined by the partition $(\frac{\mu_i}{2}, \frac{\mu_i}{2})$. This proves (2). 

For each $\gs_i$ with $\mu_i\geq 2$, the pair $(\parsum_{i-1}+1, \parsum_{i})$ is a cup of $\gs$. Furthermore, every cup of $\gs$ is nested in one of these.  Statements (1) and (2) of our lemma tell us the flag $F_\bullet$ represented by the matrix $A$ always satisfies condition (2) of Proposition~\ref{prop.Fung} (see Remark~\ref{rem.ray-condition}) and $F_\bullet$ satisfies condition (1) for each of the cups $(\parsum_{i-1}+1, \parsum_{i})$ where $\mu_i\geq 2$.

Let $A_i$ denote the $\mu_i\times \mu_i$ matrix obtained by projecting $\dot\tau_\mu A$ to the $\mu_i\times \mu_i$ diagonal block. Each $A_i\in GL_{\mu_i}(\C)$ is a matrix representative for the flag in $F_{\parsum_i}/F_{\parsum_{i-1}}$ given by $(0\subset F_{\parsum_{i-1}+1}/F_{\parsum_{i-1}} \subset F_{\parsum_{i-1}+2}/F_{\parsum_{i-1}} \subset \cdots \subset F_{\parsum_i}/F_{\parsum_{i-1}})$. On the other hand, by definition of the ideal $\isigma'$, each $A_i$ in an element of $\cv(\ci_{\gs_i})$. If $\cv(\ci_{\gs_i})\cap GL_{\mu_i}(\C) \subseteq \pi^{-1}(\ci_{\gs_i})$, then the flag $(0\subset F_{\parsum_{i-1}+1}/F_{\parsum_{i-1}} \subset F_{\parsum_{i-1}+2}/F_{\parsum_{i-1}} \subset \cdots \subset F_{\parsum_i}/F_{\parsum_{i-1}})$ satisfies condition (2) of Proposition~\ref{prop.Fung} for every cup nested within $(\parsum_{i-1}+1, \parsum_i)$ with respect to $N_i$. Thus, by Proposition~\ref{prop.Fung} and Lemma~\ref{lem.taushift},  $A\in \pi^{-1}(\cb_\gs)$. 
\end{proof}

We can now focus our attention on the case where $\gs$ is a prime standard tableau of shape $(\ell,\ell)$ and $\cupsigma$ is a connected, standard noncrossing matching. Our first lemma characterizes the vanishing set of $\cj_\gs$ using the nesting sequence of $\gs$.

\begin{lemma}\label{lem.kernel} Let $\gs$ be a standard tableau of shape $(\ell,\ell)$ and $\cj_\gs$ the monomial ideal from Definition~\ref{def:JSigma}. We have $A\in \cv(\cj_\gs)$ if and only if $N^{S_\gs(k)} a_k =0$ for all $k$.
\end{lemma}

\begin{proof} By definition,
\[
(N^{S_\gs(k)}a_k)^{\mathrm{T}} = \left[ a_{S_\gs(k)+1 ,k} \,\cdots \, a_{\ell, k} \; 0 \, \cdots \, 0 \; a_{S_\gs(k)+\ell + 1,k} \, \cdots \, a_{n,k}\; 0 \, \cdots \, 0 \right]
\]
where each string of zeros in the vector above is $S_\gs(k)$ long. Thus $N^{S_\gs(k)}a_k$ is the zero vector if and only if 
\[
a_{i, k}=0 \; \text{ for all  }\; S_\gs(k)+1\leq i \leq \ell \;\text{ and }\; S_\gs(k)+\ell +1 \leq i \leq n.
\]
These are precisely the coordinate functions in the ideal $\cj_\gs$, which proves the claim.
\end{proof}

We introduce some notation for rest of this section to streamline our arguments. For each cup $(i<j)$ in $\cupsigma$ and each $k$ with $1\leq k \leq j-1$, define
\[
\m{i}{k} := S_\gs(k)-\al(i<j)-1.
\]
For simplicity, when $i=k$ we write $\mi{i} := \m{i}{i}$. 

Recall the definition of the matrix $M_{(i<j)}$ from~\eqref{eqn.matrix} with entries equal to variables in the ring $\C[\mathbf{z}]$. Let $A$ be an $n\times n$ matrix. Consider the matrix 
\[
A_{(i<j)}:=\begin{bmatrix}
a_{S_\gs(i)-\al(i<j),i} & a_{S_\gs(i+1)-\al(i<j),i+1}& \cdots & a_{S_\gs(j-1)-\al(i<j),j-1}\\
    a_{S_\gs(i)-\al(i<j)+1,i}& a_{S_\gs(i+1)-\al(i<j)+1,i+1} & \cdots &a_{S_\gs(j-1)-\al(i<j)+1,j-1}\\
    \vdots & \vdots &\ddots& \vdots\\
    a_{S_\gs(i),i} & a_{S_\gs(i+1),i+1}& \cdots & a_{S_\gs(j-1),j-1}\\
a_{S_\gs(i)-\al(i<j)+\ell,i} & a_{S_\gs(i+1)-\al(i<j)+\ell,i+1}& \cdots & a_{S_\gs(j-1)-\al(i<j)+\ell,j-1}\\
    a_{S_\gs(i)-\al(i<j)+\ell+1,i}& a_{S_\gs(i+1)-\al(i<j)+\ell+1,i+1} & \cdots &a_{S_\gs(j-1)-\al(i<j)+\ell+1,j-1}\\
    \vdots & \vdots &\ddots& \vdots\\
    a_{S_\gs(i)+\ell,i} & a_{S_\gs(i+1)+\ell,i+1}& \cdots & a_{S_\gs(j-1)+\ell,j-1}\\
\end{bmatrix}
\]
obtained by evaluating $M_{(i<j)}$ at $A$. By Remark~\ref{rem.evaluation}, if $a_1, a_2, \ldots, a_n$ are the columns of $A$, then (up to removing rows of zeros) we have
\begin{align*}
A_{(i<j)} &= \begin{bmatrix}  | & | &  & |  \\   N^{S_\gs(i)-\al(i<j)-1}a_i &  N^{S_\gs(i+1)-\al(i<j)-1}a_{i+1} & \cdots &  N^{S_\gs(j-1)-\al(i<j)-1}a_{j-1} \\ | & | &  & |   \end{bmatrix} \\ 
&= \begin{bmatrix}  | & | &  & |  \\   N^{\mi{i}}a_i &  N^{\m{i}{i+1}}a_{i+1} & \cdots &  N^{\m{i}{j-1}}a_{j-1} \\ | & | &  & |   \end{bmatrix}.
\end{align*}
The next result is a linchpin for latter arguments. It tells us that the minors in $\ci_\gs$ force certain rank conditions on $A\in \cv(\ci_\gs)$ using Lemma~\ref{lem.rankcond}. 
We call the matrix $\widetilde{A}_{(i<j)}$ defined in~\eqref{eqn.extmatrix} below the \emph{extended cup matrix} associated to~$A$.

\begin{prop}\label{prop.aug.rank} Let $\gs$ be a standard tableau of shape $(\ell, \ell)$ and suppose $(i<j)$ is a cup in $\cc(\gs)$ of size at least $2$. If $A\in GL_n(\C)\cap\cv(\ci_\gs)$ then the matrix
\begin{eqnarray}\label{eqn.extmatrix}
\;\;\;\;\;\widetilde{A}_{(i<j)} :=  \begin{bmatrix} | &  & | & | & | &  & |  \\ N^{\mi{i}}a_1 & \cdots & N^{\mi{i}}a_{i-1} & N^{\mi{i}}a_i & N^{\m{i}{i+1}}a_{i+1} & \cdots & N^{\m{i}{j-1}}a_{j-1} \\| &   & | & | & | & &|  \end{bmatrix}
\end{eqnarray}
has $\rk \widetilde{A}_{(i<j)} =\al(i<j)+1$. 
\end{prop}
\begin{proof} Let $A\in GL_n(\C)\cap\cv(\ci_\gs)$ and set $\al=\al(i<j)$. 

It suffices to prove $\rk \widetilde A_{(i<j)} \leq \al+1$. Given this fact, the assertion $\rk \widetilde A_{(i<j)} = \al+1$ follows. Indeed, $N^{\mi{i}}\left( F_i \right) = \C\{ N^{\mi{i}}a_1, \ldots,N^{\mi{i}}a_i \}$ is a subspace of the column space of $\widetilde A_{(i<j)}$ so $\rk \widetilde A_{(i<j)}$ is at least $\dim N^{\mi{i}}\left( F_i \right)$. Since $\dim F_i = i$ and $\dim \ker N^{\mi{i}} = 2\mi{i}$ we have 
\[
\rk \widetilde A_{(i<j)}  \geq \dim N^{\mi{i}}\left( F_i \right) \geq i - 2\mi{i} = \al +1,
\]
where the last equality follows by Lemma~\ref{lem.leftend} using the fact that $\mi{i}=S_\gs(i)-\al-1$. 

We proceed with our proof that $\rk  \widetilde A_{(i<j)} \leq \al+1$. Consider the projection from $\C^n$ to $\C^{2\al+2}$ defined by
\[
\left[ v_1, v_2, \ldots, v_n\right]^T \mapsto \left[v_1,v_2, \ldots, v_{\al+1}, v_{\ell+1}, v_{\ell+2}, \ldots, v_{\ell+\al+1}\right]^T.
\]
By Lemma~\ref{lem.kernel}, we know $a_k\in \ker N^{S_\gs(k)}$ for all $k$. For all $i\leq k \leq j-1$ we have $S_\gs(k) - \m{i}{k} = \al+1$ and thus
\[
N^{\m{i}{k}} a_k \in \ker N^{\al +1} = \C\{e_1, e_2, \ldots, e_{\al+1}\} \oplus \C\{ e_{\ell+1}, e_{\ell+2}, \ldots, e_{\ell+\al+1}\}.
\]
Now consider $k$ with $1\leq k \leq i$. If $S_\gs(k)\leq \mi{i}$ then $N^{\mi{i}}a_k = 0$ and otherwise,
\begin{align}\label{eqn.Nm(i)}
N^{\mi{i}} a_k \in \ker N^{S_\gs(k)-\mi{i}} &= \C\{e_1, e_2, \ldots, e_{S_\gs(k)-\mi{i}}\} \\ \nonumber
& \qquad \qquad \oplus \C\{e_{\ell+1}, e_{\ell+2}, \ldots, e_{\ell + S_\gs(k)-\mi{i}}\}.
\end{align}
In particular, the matrix obtained from $\widetilde A_{(i<j)}$ by replacing each column with its projection to $\C^{2\al+2}$ deletes rows $\al+2, \al+3, \ldots, \ell$ and $\ell+\al+2, \ell+\al+3, \ldots, 2\ell =n$ of $\widetilde A_{(i<j)}$, all of which are zero.  The rank of the resulting $(2\al+2)\times (j-1)$ matrix is equal to the rank of the original. Abusing notation, we denote this matrix by $\widetilde A_{(i<j)}$ as well.

For any choice of columns $c_1$ and $c_2$ of $\widetilde A_{(i<j)}$ we consider
\begin{eqnarray}\label{eqn.minorsextended}
\pl_{c_1,c_2}\left(\widetilde A_{(i<j)}\right) := \sum_{k=0}^{\al(i<j)}\pl_{\{k+1, 2\al(i<j)+2-k\}, \{c_1,c_2\}} \left(\widetilde A_{(i<j)}\right) .
\end{eqnarray}
The assertion that $\rk \widetilde A_{(i<j)} \leq \al(i<j)+1$ follows directly from Lemma~\ref{lem.rankcond} once we prove~\eqref{eqn.minorsextended} is equal to $0$ for all $1\leq c_1<c_2\leq j-1$. 

If $i\leq c_1<c_2\leq j-1$, then $\pl_{c_1, c_2} = \pl_{c_1, c_2}^{(i<j)}=\sum_{k=0}^{\al}\pl_{\{k+1, 2\al+2-k\}, \{c_1,c_2\}}$ is a generator of $\ci_\gs$, and $\pl_{c_1, c_2}\left(\widetilde A_{(i<j)}\right) = \pl_{c_1, c_2}\left(A\right)=0$ since $A\in \cv(\ci_\gs)$. 

Now suppose $c_1< i <c_2\leq j-1$ and consider $u_{c_1} = I_n + E_{c_1, i} \in B$. Right multiplication of $A$ by $u_{c_1}$ replaces the $i$-th column $a_i$ of $A$ by the sum $a_i+a_{c_1}$. We have 
\[
\pl_{c_1,c_2}\left(\widetilde A_{(i<j)}\right) =  \pl_{i,c_2}^{(i<j)}\left(A\, u_{c_1}\right) -  \pl_{i,c_2}^{(i<j)}\left(A\right)=0-0=0
\]
since $\pl_{i,c_2} = \pl_{i,c_2}^{(i<j)}$ is a generator of $\ci_\gs$, $A\in \cv(\isigma)$, and $Au_{c_1}\in \cv(\ci_\gs)$ as well because $\ci_\gs$ is right $B$-invariant by Proposition~\ref{binv}.

For the rest of the proof, we consider the case of $1\leq c_1<c_2\leq i$.  We start by computing the form of the $c$-th column for some $c\leq i$. First, if $S_\gs(c) \leq \mi{i}$ then $N^{\mi{i}}a_{c} =0$ so the $c$-th column of $\widetilde A_{(i<j)}$ is $0$ and~\eqref{eqn.minorsextended} is obviously $0$ if $c$ is selected as $c_1$ or $c_2$. We may therefore assume that $S_\gs(c) > \mi{i}$. Because $N^{\mi{i}} a_{c}\in \ker N^{S_\gs(c) - \mi{i}}$, the $c$-th column of $\widetilde A_{(i<j)}$ is of the form 
\begin{align} \label{eqn.colform}
\left[ a_{\mi{i}+1, c}\  \right.,  a_{m(i)+2, c}\ , &\cdots ,  a_{\mi{i}+(S_\gs(c)-\mi{i}), c}\ , 0^{S_\gs(i) - S_\gs(c)} \ ,   \\  
\nonumber &  a_{\ell+\mi{i}+1, c} \ ,    a_{\ell + m(i)+2, c}\ , \cdots , a_{\ell + \mi{i}+(S_\gs(c)-\mi{i}), c}\ ,\left. 0^{S_\gs(i) - S_\gs(c)}  \right]^T.
\end{align}
Note that this vector has $2\al(i<j)+2$ many entries since $\mi{i} = S_\gs(i)-\alpha(i<j)-1$ and thus 
\[
(S_\gs(c)-\mi{i}) + (S_\gs(i) -  S_{\gs}(c) )  =  \alpha(i<j)+1.
\]
Now, for $c_1<c_2\leq i$, the equation~\eqref{eqn.minorsextended} becomes
\begin{eqnarray}\label{eqn.minorsextended2}
\pl_{c_1,c_2}\left(\widetilde A_{(i<j)}\right) = \sum_{k=0}^{\al(i<j)} \begin{vmatrix} a_{\mi{i}+k+1, c_1} & a_{\mi{i}+k+1, c_2}\\ a_{\ell+\mi{i}+\al(i<j) +1 - k, c_1} & a_{\ell+\mi{i}+\al(i<j) +1 - k, c_2}  \end{vmatrix} .
\end{eqnarray}

Suppose $S_\gs(c_1)+S_\gs(c_2)\leq i$. Under this condition, we claim that every minor in the sum~\eqref{eqn.minorsextended2} is $0$. Indeed, since the $c$-th column has form~\eqref{eqn.colform}, for all $0\leq k \leq \al(i<j)$ we have 
\begin{align}\label{eqn.ineq1}
a_{\mi{i}+k+1, c} \neq 0 & \Rightarrow  \mi{i}+1 \leq \mi{i} +k + 1 \leq \mi{i}+(S_\gs(c)-\mi{i}) \\
\nonumber &\Rightarrow k \leq S_\gs(c)-\mi{i}  -1
\end{align}
and 
\begin{align} \label{eqn.ineq2}
&a_{\ell+\mi{i}+\al(i<j) +1 - k, c} \neq 0 \\
\nonumber  & \qquad \qquad \Rightarrow  \mi{i}+1  \leq  \mi{i} + \alpha(i<j)+1 - k \leq  \mi{i} + (S_\gs(c)-\mi{i}) \\
\nonumber & \qquad \qquad \Rightarrow  \mi{i}+ \al(i<j) +1 - S_\gs(c) \leq k.
\end{align}
Thus, using~\eqref{eqn.ineq1} for $c=c_1$ and~\eqref{eqn.ineq2} for $c=c_2$ we have
\begin{align*}
a_{\mi{i}+k+1, c_1}&a_{\ell+\mi{i}+\al(i<j) +1 - k, c_2} \neq 0 \\ &\Rightarrow   \mi{i}+ \al(i<j) +1 - S_\gs(c_2)  \leq S_\gs(c_1)-\mi{i}  -1 \\
&\Rightarrow  2\mi{i} +\al (i<j) +1 \leq S_\gs(c_1)+S_\gs(c_2) -1 \\
&\Rightarrow i  < S_\gs(c_1)+S_\gs(c_2),
\end{align*}
where the last equality follows since $2\mi{i} +\al (i<j) +1 = 2S_\gs(i) -\al(i<j)-1 = i$ by Lemma~\ref{lem.leftend}. By the same reasoning with the roles of $c_1$ and $c_2$ reversed we have 
\[
a_{\mi{i}+k+1, c_2}a_{\ell+\mi{i}+\al(i<j) +1 - k, c_1} \neq 0 \Rightarrow i  < S_\gs(c_1)+S_\gs(c_2).
\]
Thus, if $S_\gs(c_1)+S_\gs(c_2)\leq i$, we have shown that  
\[
\begin{vmatrix} a_{\mi{i}+k+1, c_1} & a_{\mi{i}+k+1, c_2}\\ a_{\ell+\mi{i}+\al(i<j) +1 - k, c_1} & a_{\ell+\mi{i}+\al(i<j) +1 - k, c_2}  \end{vmatrix} =0
\]
for all $0\leq k \leq \al(i<j)$. This proves~\eqref{eqn.minorsextended2} is $0$ in this case since every minor in the sum is $0$, proving our claim.

Now assume $S_\gs(c_1) + S_\gs(c_2)>i$.  Using Lemma~\ref{lem.leftend} once more, we have 
\begin{align*}
i<S_\gs(c_1) + S_\gs(c_2) &\Leftrightarrow i + 1 \leq  S_\gs(c_1) + S_\gs(c_2) \\
& \Leftrightarrow 2 S_{\gs}(i)-\al(i<j) \leq S_\gs(c_1) + S_\gs(c_2) \\
& \Leftrightarrow 2 S_{\gs}(i)- S_\gs(c_1)-S_\gs(c_2) \leq \al(i<j). 
\end{align*}
Since $c_1<c_2\leq i$ we know $S_\gs(c_1)\leq S_\gs(c_2)\leq S_\gs(i)$ and thus $2 S_{\gs}(i)- S_\gs(c_1)-S_\gs(c_2)\geq 0$. For the remainder of the proof, let $(i'<j')$ be the unique cup above $(i<j)$ such that 
\begin{eqnarray}\label{eqn.i'<j'def}
\al(i'<j') = \al(i<j) - \left( 2 S_{\gs}(i)- S_\gs(c_1)-S_\gs(c_2) \right) \geq 0.
\end{eqnarray}
Since the cup $(i'<j')$ has $2 S_{\gs}(i)- S_\gs(c_1)-S_\gs(c_2)$ fewer cups nested above it than $(i<j)$, 
\begin{align*}
S_\gs(i') \leq  S_\gs(i) - \left(2 S_{\gs}(i)- S_\gs(c_1)-S_\gs(c_2) \right) \Rightarrow S_\gs(i')+S_\gs(i) \leq S_\gs(c_1)+S_\gs(c_2).
\end{align*}
Because $S_\gs(i)\geq S_\gs(c_2)$, the last inequality implies $S_\gs(i') \leq S_\gs(c_1)$. Furthermore, as $i'$ is the left endpoint of a cup, it is the minimal index with cup sequence value $S_\gs(i')$, and we conclude $i'\leq c_1$. Thus $c_1$ and $c_2$ are vertices of the cup diagram $\cupsigma$ nested inside the cup $(i'<j')$, and they index columns in the matrix $A_{(i'<j')}$.

Since $S_\gs(c)-\mi{i}-1 = \al(i<j) - \left( S_\gs(i) - S_\gs(c_1) \right)$ and $\mi{i}+\al(i<j) +1 -S_\gs(c_1) = S_\gs(i)-S_{\gs}(c_2)$, using  equations~\eqref{eqn.ineq1} and~\eqref{eqn.ineq2} for $c_1$ and $c_2$ yields the inequalities:
\begin{align}
\label{eqn.ineqc1A} a_{\mi{i}+k+1, c_1} \neq 0 &\Rightarrow k \leq   \al(i<j) - \left(S_\gs(i)-S_\gs(c_1)\right) \\
\label{eqn.ineqc2A} a_{\mi{i}+k+1, c_2} \neq 0 &\Rightarrow k \leq   \al(i<j) - \left(S_\gs(i)-S_\gs(c_2)\right) \\
\label{eqn.ineqc1B} a_{\ell + \mi{i} + \al(i<j)+1-k,c_1} \neq 0 &\Rightarrow S_\gs(i)-S_\gs(c_1) \leq k \\
\label{eqn.ineqc2B} a_{\ell + \mi{i} + \al(i<j)+1-k,c_2} \neq 0 &\Rightarrow S_\gs(i)-S_\gs(c_2) \leq k .
\end{align}
Note that $c_1< c_2$ implies $S_\gs(c_1)\leq S_\gs(c_2)$. By~\eqref{eqn.ineqc1A} and ~\eqref{eqn.ineqc2A}, we have 
\[
a_{\mi{i}+k+1, c_1} = a_{\mi{i}+k+1, c_2} = 0 \; \text{ when } \; k > \al(i<j) - \left( S_\gs(i)-S_{\gs(c_2)} \right)
\]
and by~\eqref{eqn.ineqc1B} and ~\eqref{eqn.ineqc2B}, we have  
\[
a_{\ell + \mi{i} + \al(i<j)+1-k,c_1} = a_{\ell + \mi{i} + \al(i<j)+1-k,c_2} = 0 \; \text{ when } \; k< S_{\gs(i)} - S_\gs(c_2).
\]
Thus, the only minors in the sum~\eqref{eqn.minorsextended2} which are nonzero occur for $k$ such that $S_\gs(i)-S_\gs(c_2) \leq k \leq \al(i<j)- \left( S_\gs(i)-S_\gs(c_2) \right)$. Equation~\eqref{eqn.minorsextended2} becomes
\begin{align}\label{eqn.minorsextended3}
\pl_{c_1, c_2}(\widetilde A_{(i<j)}) & =  \sum_{k=S_\gs(i)-S_\gs(c_2)}^{\al(i<j) - \left( S_\gs(i)-S_\gs(c_2) \right)} \begin{vmatrix} a_{\mi{i}+k+1, c_1} & a_{\mi{i}+k+1, c_2}\\ a_{\ell+\mi{i}+\al(i<j) +1 - k, c_1} & a_{\ell+\mi{i}+\al(i<j) +1 - k, c_2}  \end{vmatrix}. 
\end{align}

Let $t(c_1, c_2):= S_\gs(c_2)-S_{\gs}(c_1) \geq 0$. 
By~\eqref{eqn.ineqc1B}, $a_{\ell + \mi{i} + \al(i<j)+1-k,c_1}=0$ for all $k< S_\gs(i)-S_\gs(c_1)$, i.e., for all $k$ such that $S_\gs(i)-S_\gs(c_2)\leq k \leq S_\gs(i)-S_\gs(c_2) +t(c_1,c_2)-1$. Similarly, by~\eqref{eqn.ineqc1A}, $a_{\mi{i}+k+1, c_1}=0$ for all $k> \al(i<j) - \left( S_\gs(i) - S_\gs(c_1)\right)$, i.e., for all $k$ such that 
\begin{align*}
\al(i<j) - \left( S_\gs(i) - S_\gs(c_1)\right) + 1\leq k &\leq \al(i<j) - \left( S_\gs(i) - S_\gs(c_1)\right) +t(c_1, c_2)\\ & \qquad \qquad \qquad  =  \al(i<j) - \left( S_\gs(i) - S_\gs(c_2)\right).
\end{align*}
This means we can further decompose~\eqref{eqn.minorsextended3} as: 
\begin{align*}
\pl_{c_1, c_2}(\widetilde A_{(i<j)}) & =   \left( \sum_{k=S_\gs(i)-S_\gs(c_2)}^{S_\gs(i) -S_\gs(c_2) + t(c_1,c_2)-1}  a_{\mi{i}+k+1, c_1}a_{\ell+\mi{i}+\al(i<j) +1 - k, c_2} \right)  \\
& \quad \quad + \left( \ \sum_{k=S_\gs(i)-S_\gs(c_2) + t(c_1, c_2)}^{\al(i<j)-\left( S_\gs(i)-S_\gs(c_1) \right)} \begin{vmatrix} a_{\mi{i}+k+1, c_1} & a_{\mi{i}+k+1, c_2}\\ a_{\ell+\mi{i}+\al(i<j) +1 - k, c_1} & a_{\ell+\mi{i}+\al(i<j) +1 - k, c_2}  \end{vmatrix} \ \right) \\
& \qquad \quad \quad  -  \left( \sum_{k=\al(i<j) - \left( S_\gs(i)-S_\gs(c_1) \right)+1 }^{\al(i<j) - \left( S_\gs(i)-S_\gs(c_2) \right)} a_{\mi{i}+k+1, c_2}a_{\ell+\mi{i}+\al(i<j) +1 - k, c_1}  \right).
\end{align*}
If $t(c_1,c_2) = 0$, then the first and last expressions above are both $0$, otherwise there are $t(c_1,c_2)$ terms in each of the first and last sums. 
We simplify this expression by applying Lemma~\ref{pairingdeterminants} exactly $t(c_1, c_2)$ many times to obtain:
\begin{align*}\label{eqn.minorsextended4}
\pl_{c_1, c_2}(\widetilde A_{(i<j)}) & =   \sum_{k=S_\gs(i)-S_\gs(c_2)}^{\al(i<j)-\left( S_\gs(i)-S_\gs(c_1) \right)} \begin{vmatrix} a_{\mi{i}+k+1, c_1} & a_{\mi{i}+k+t(c_1,c_2)+1, c_2}\\ a_{\ell+\mi{i}+\al(i<j) +1 - \left( k + t(c_1,c_2) \right) , c_1} & a_{\ell+\mi{i}+\al(i<j) +1 - k, c_2}  \end{vmatrix}.
\end{align*}
Finally, we re-index the sum over $0\leq k \leq \al(i<j) - \left( 2S_\gs(i) - S_\gs(c_1)-S_\gs(c_2) \right) = \al (i'<j')$. Using the definition of $\al(i'<j')$ from~\eqref{eqn.i'<j'def} and the fact that $\mi{i} = S_\gs(i)-\al(i<j)-1$ we get  
\begin{align*}
S_\gs(c_1) - \al(i'<j')+k & = \mi{i} + k + S_\gs(i)-S_{\gs}(c_2)+1  \\
S_\gs(c_2) - \al(i'<j')+k &= \mi{i} + k + S_\gs(i)-S_{\gs}(c_2) + t(c_1, c_2)+1 \\
S_\gs(c_2)+k & = \mi{i} +\al(i<j) +1 - \left( k + S_\gs(i)-S_\gs(c_2) \right) \\
S_\gs(c_1)+k & = \mi{i} +\al(i<j) +1 - \left( k + S_\gs(i)-S_\gs(c_2) +t(c_1,c_2)  \right) \ ,
\end{align*}
and reindexing gives us 
\[
\pl_{c_1, c_2}(\widetilde A_{(i<j)})  = \sum_{k=0}^{\al(i'<j')} \begin{vmatrix} a_{S_\gs(c_1)-\al(i'<j')+k, c_1} &  a_{S_\gs(c_2)-\al(i'<j')+k, c_2} \\  a_{\ell-S_\gs(c_1)-k,c_1}  & a_{\ell + S_\gs(c_2) - k, c_2 }\end{vmatrix} = \pl_{c_1, c_2}^{(i'<j')} (A).
\]
Thus, $\pl_{c_1, c_2}(\widetilde A_{(i<j)}) = \pl_{c_1, c_2}^{(i'<j')} (A)=0$ since $\pl_{c_1, c_2}^{(i'<j')}$ is a generator of $\cm_{\gs, (i'<j')} \subseteq \ci_\gs$ and $A\in \cv(\ci_\gs)$. This completes the proof. 
\end{proof}

The following lemma is a straightforward application of the rank-nullity theorem. 

\begin{lemma}\label{lem.linearalg} Suppose $T: \C^n \to \C^n$ is a linear transformation and $A$ and $B$ are subspaces of $\C^n$ such that $B\subseteq A$ and $T(A)\subseteq B$. If 
\[
\dim A = \dim B + \dim\,\ker(T)
\]
then 
\begin{enumerate}[(i)]
\item \label{lem.linearalg1} $T(A)=B$ and 
\item \label{lem.linearalg2} $A=T^{-1}(B)$.
\end{enumerate}
\end{lemma}

The next two results use the linear algebraic statements of Lemma~\ref{lem.linearalg} to obtain two corollaries of Proposition~\ref{prop.aug.rank}. Recall that $F_i =  \C\{a_1, a_2, \ldots, a_i\}$ denotes the subspace spanned by the first $i$ columns of $A$.

\begin{corollary} \label{cor.preimage} Suppose $\gs$ is a rectangular tableau and $A\in \GL_n(\C)\cap \cv(\ci_\gs)$. 
For each cup $(i<j)$ of $\cupsigma$ with size $\gd(i<j)\geq 2$, $\dim N^{\mi{i}}\left(F_i\right) = \al(i<j)+1$ and 
\[
N^{-\mi{i}} \left(N^{\mi{i}}\left(F_i\right)\right)=F_i.
\]
\end{corollary}
\begin{proof} We claim $i = \dim F_i = \dim N^{\mi{i}}\left( F_i \right) + 2\,\mi{i}$. Given this claim, the second part of the corollary follows from Lemma~\ref{lem.linearalg}~\eqref{lem.linearalg2} with $A=F_i$, $B=N^{\mi{i}} \left( F_i \right)$, and $T=N^{\mi{i}}$ since $\dim \ker N^{\mi{i}} = 2\,\mi{i}$.

By Lemma~\ref{lem.leftend} we have 
\begin{eqnarray}\label{eqn.preimage1}
i-2\,\mi{i} =  i - 2S_\gs(i) +2\ga(i<j) +2 = \ga(i<j)+1.
\end{eqnarray}
We know $N^{\mi{i}} \left( F_i \right)$ is a subspace of the column space of the extended cup matrix $\widetilde{A}_{(i<j)}$. By Proposition~\ref{prop.aug.rank}, $\rk \widetilde A_{(i<j)}\leq \ga(i<j)+1$ and thus $\dim N^{\mi{i}} \left( F_i \right) \leq \ga(i<j)+1$. On the other hand, since $\dim F_i=i$ and $\dim \ker N^{\mi{i}} = 2\mi{i}$ we have 
\begin{align*}
\dim N^{\mi{i}} \left( F_i \right) &\geq i - 2\mi{i} = \al(i<j)+1
\end{align*}
where the last equality follows from~\eqref{eqn.preimage1}. This proves 
\begin{eqnarray}\label{eqn.preimage2}
\dim N^{\mi{i}}\left(F_i\right)=\ga(i<j)+1,
\end{eqnarray}
which proves the first part of the corollary.
The claim follows from~\eqref{eqn.preimage1} and~\eqref{eqn.preimage2}, and proves the second part.
\end{proof}

\begin{corollary} \label{cor.rank_condition} Suppose $(k^*<\E(k^*))$ is a cup in $\cupsigma$ with size $\gd(k^*<\E(k^*))\geq 2$ and $k^*<k< \E(k^*)$. Then 
\[
N^{\m{k^*}{k} - \mi{k^*}}(a_k)\in  A_{k^*}.
\]
In particular, if $(k^*<\E(k^*))$ is the unique cup above the cup containing $k$ (as either the left or right endpoint) then each of the following hold.
\begin{enumerate}[(i)]
\item \label{lem.rank.1} If $k$ is the left endpoint of the cup then
\[
N^{\frac{1}{2}(k-k^*+1)}(a_k) \in F_{k^*}.
\]
\item \label{lem.rank.2} If $k$ is the right end point of the cup and $\E^{-1}(k)$ denotes the left endpoint then 
\[
N^{\frac{1}{2}(\E^{-1}(k)-k^*-1)+\delta(\E^{-1}(k)<k)}(a_k) \in F_{k^*}.
\]
\end{enumerate}
\end{corollary}
\begin{proof} Consider the extended cup matrix for the cup $(k^* < \E(k^*))$,
\begin{eqnarray*}
 \widetilde A_{(k^*<\E(k^*))}=\begin{bmatrix}
| &  & | & | & | &  \\   
N^{\mi{k^*}}a_1 &  \cdots& N^{\mi{k^*}}a_{k^*}&  N^{\m{k^*}{k^*+1}}a_{k^*+1} & N^{\m{k^*}{k^*+2}} a_{k^*+2} & \cdots 
\\ | &  & | & | & | &
\end{bmatrix}.
\end{eqnarray*}
By Proposition~\ref{prop.aug.rank}, we have $\rk \widetilde A_{(k^*<\E(k^*))} = \al(k^*<\E(k^*))+1$. Thus, the rank of the submatrix spanned by the first $k^*$ columns of $\widetilde A_{(k^*<\E(k^*))}$ and any later column is at most $\al(k^*<\E(k^*))+1$.  Because $k^*<k < \E(k^*)$ we get
\begin{eqnarray} \label{eqn.rank1}
\rk \begin{bmatrix} | & & | & | \\ N^{\mi{k^*}}a_1 &  \cdots& N^{\mi{k^*}}a_{k^*} & N^{\m{k^*}{k}}a_k\\ | & & | & |  \end{bmatrix} \leq \al(k^*<\E(k^*))+1.
\end{eqnarray}
By Corollary~\ref{cor.preimage}, $\dim  N^{\mi{k^*}}\left(F_{k^*} \right) = \al(k^*<\E(k^*))+1$ and therefore~\eqref{eqn.rank1} tells us 
\begin{equation}\label{eqn.inclusion1}
N^{\m{k^*}{k}}(a_k) \in N^{\mi{k^*}}\left(F_{k^*} \right).
\end{equation}
Notice that 
\begin{align*}
\m{k^*}{k} &= (\m{k^*}{k} - \mi{k^*}) + \mi{k^*} = S_\gs(k) - S_\gs(k^*) + \mi{k^*}. 
\end{align*}
Using this, \eqref{eqn.inclusion1} becomes
\begin{align} \label{eqn.inclusion2}
\nonumber & N^{\mi{k^*}} \left( N^{S_\gs(k)-S_\gs(k^*)}(a_k) \right)\in N^{\mi{k^*}}(F_{k^*} )\\ 
 &\quad\quad \Rightarrow  N^{S_\gs(k)-S_\gs(k^*)}(a_k) \in N^{-\mi{k^*}} \left( N^{\mi{k^*}}(F_{k^*} )\right) \quad \text{(by definition of preimage)}  \nonumber \\
 & \quad\quad \Rightarrow  N^{S_\gs(k)-S_\gs(k^*)}(a_k) \in F_{k^*} \quad \text{(by Corollary~\ref{cor.preimage}).} 
\end{align}

If $k$ is the left endpoint of a cup and $(k^*<\E(k^*))$ is the unique cup directly above $(k<\E(k))$ then 
\[
S_\gs(k)-S_\gs(k^*) = \frac{1}{2}\left(k-k^*+1\right)
\]
by Lemma~\ref{lem.Sdiff}. Now~\eqref{eqn.inclusion2} yields statement (i).

On the other hand, if $k$ is the right endpoint of a cup $(\E^{-1}(k)<k)$ and $(k^*<\E(k^*))$ is the unique cup directly above it then combining Lemmas~\ref{lem.Sdiff} and~\ref{lem.length} gives us
\begin{align*}
S_\gs(k)-S_\gs(k^*) &= S_\gs(k)-S_{\gs}(\E^{-1}(k)) + S_{\gs}(\E^{-1}(k)) - S_\gs(k^*) \\
&= \gd(\E^{-1}(k)<k) -1 + \frac{1}{2}\left(\E^{-1}(k)-k^*+1\right) \\
&= \gd(\E^{-1}(k)<k) + \frac{1}{2}\left(\E^{-1}(k)-k^*-1\right). 
\end{align*}
Now~\eqref{eqn.inclusion2} implies statement (ii). 
\end{proof}

Our next proposition combines the previous two corollaries to prove our key technical result: any invertible matrix in $\cv(\ci_\gs)$ satisfies the conditions defining $\SF_\gs$ from Fung's Proposition~\ref{prop.Fung}.

\begin{prop}\label{prop.components} Let $\gs$ be a prime standard tableau of shape $(\ell,\ell)$ and suppose $A\in \cv(\ci_\gs) \cap GL_n(\C)$. Let $F_\bullet$ be the flag represented by $A$, so $F_i = \C\{a_1, \ldots, a_i\}$ for all $i=1, \ldots, n$ where $a_1, a_2, \ldots, a_n$ denote the columns of $A$. For each cup $(i<j)$ in $\cc(\sigma)$ we have:
\begin{enumerate}[(i)]
\item \label{prop:bigprop1} $N\left(a_k\right)\in F_{k-1}$ for all $k\leq j$ and
\item \label{prop:bigprop2} $N^{\gd(i<j)}\left(F_j\right) = F_{i-1}$.
\end{enumerate}
\end{prop}

\begin{proof} Notice that $F_{i-1}\subseteq F_j$ and
\begin{align*}
\dim F_{i-1} + \dim \ker N^{\gd(i<j)} &= (i-1) + 2\gd(i<j)\\
&= (i-1)+(j-i+1) = j = \dim F_j,
\end{align*}
Thus to prove~\eqref{prop:bigprop2} it suffices, by Lemma~\ref{lem.linearalg}, to show that $N^{\gd(i<j)}(F_j) \subseteq F_{i-1}$. Lemma~\ref{lem.linearalg} also implies that if $N^{\gd(i<j)}(F_j) \subseteq F_{i-1}$, then
\begin{eqnarray}\label{eqn.bigprop.preimage1}
F_j = N^{-\gd(i<j)}(F_{i-1}).
\end{eqnarray}
In particular, we get that~\eqref{eqn.bigprop.preimage1} holds whenever~\eqref{prop:bigprop2} does, which is useful in the proof below.

We proceed by induction on $j$, that is, by induction on the right endpoint of the cup $(i<j)$.  Our base case is the cup $(i<j)$ such that $j$ is minimal.  Any such cup must have size $\gd(i<j) = 1$, so $i=j-1$.  
The assumption that $j$ is minimal also implies that every $k$ with $1\leq k\leq i=j-1$ is the left endpoint of a cup and the unique cup above each such cup $(k<\E(k))$ is $(k-1<\E(k-1))$. In particular, by Corollary~\ref{cor.rank_condition}\eqref{lem.rank.1} we have 
\begin{align}\label{eqn.bigprop.base1}
N^{\frac{1}{2}(k-k^* +1)}(a_k)\in F_{k^*} \Rightarrow N(a_k ) \in F_{k-1} \quad \text{ for all $1\leq k \leq i=j-1$.}
\end{align}
The cup directly above $(j-1, j)$ is $(j-2<\E(j-2))$. Thus Corollary~\ref{cor.rank_condition}\eqref{lem.rank.2} implies
\begin{eqnarray}\label{eqn.bigprop.base2}
N^{\frac{1}{2}((j-1)-(j-2)-1) + \gd(j-1<j)}(a_j) \in F_{j-2} \Rightarrow N(a_j) \in F_{j-2}.
\end{eqnarray}
This also implies $N(a_j)\in F_{j-1}$ since $F_{j-2}\subset F_{j-1}$. Together with~\eqref{eqn.bigprop.base1} this gives us $N(a_k)\in F_{k-1}$ for all $k\leq j$, proving~\eqref{prop:bigprop1} in the base case. We also have $N(F_{j-1})\subseteq F_{i-1}= F_{j-2}$ by~\eqref{eqn.bigprop.base1} and $N(a_j)\in F_{j-2}$ by~\eqref{eqn.bigprop.base2} and thus $N(F_j) \subseteq F_{j-2}$, proving~\eqref{prop:bigprop2} in this case.  This concludes the proof of the base case.

Now suppose $(i<j)$ is an arbitrary cup in $\cupsigma$ and that statements~\eqref{prop:bigprop1} and \eqref{prop:bigprop2} hold for all cups $(i'<j')$ such that $j'<j$. Assume that $j<n$. (We address the case $j=n$, which is much more straightforward, below.) Since $\gs$ is prime, we may assume that there exists a cup $(i^*<j^*)$ directly above $(i<j)$.

We use the induction hypothesis to prove three auxiliary claims:
\
\begin{itemize}
\item \textbf{Claim 1:} If $1<k\leq i$ is the left endpoint of a cup and $(k^*<\E(k^*))$ is the unique cup directly above the cup $(k<\E(k))$, then $N^{\frac{1}{2}(k-k^*-1)} (F_{k-1}) = F_{k^*}$ and $F_{k-1} =  N^{-\frac{1}{2}(k-k^*-1)}( F_{k^*})$.

\

\item \textbf{Claim 2:} $N^{\gd(i<j)-1}(F_{j-1}) = F_i$ and $F_{j-1} = N^{-\gd(i<j)+1}(F_i)$.

\

\item \textbf{Claim 3:} $N^{\gd(i<j)} (a_j)\in F_{i-1}$.
\end{itemize}

\

\textbf{Proof of Claim 1:} Any  cup that appears between $k^*$ and $k$ is of the form $(k'<\E(k'))$ where $k^*<k'<\E(k')<k$. In particular, the induction hypothesis applies to the cup $(k'<\E(k'))$ and there are $\frac{1}{2}(k-k^*-1)$ cups with this property. Let $(i_1<j_1),(i_2<j_2),\cdots,(i_t<j_t)$ be the sequence of cups such that $(k^*<\E(k^*))$ the unique cup directly above each one and such that $i_1= k^*+1$, $i_r= j_{r-1}+1$ for all $2\leq r \leq t$, and $j_t = k-1$. By the induction hypothesis applied using statement~\eqref{prop:bigprop2} to each such cup we have
\begin{equation}\label{eqn.inductionstep1}
N^{\delta(i_r<j_r)} (F_{j_r}) = F_{i_r-1}  \quad \text{ for all $1\leq r \leq t$}.
\end{equation}
Since there are $\frac{1}{2}(k-k^*-1)$ total cups between $k^*$ and $k$, 
\[
\delta(i_1<j_1)+\delta(i_2<j_2)+ \cdots + \delta(i_t<j_t) = \frac{1}{2}(k-k^*-1)
\]
and thus~\eqref{eqn.inductionstep1} gives us the first equality of Claim 1. The second follows by applying  Lemma~\ref{lem.linearalg}~\eqref{lem.linearalg2}.

\

\textbf{Proof of Claim 2:} If $\gd(i<j)=1$ then the statement holds trivially, so we may assume $\gd(i<j)\geq 2$. Let $(i_1<j_1),(i_2<j_2),\cdots,(i_t<j_t)$ be the sequence of cups nested inside the cup $(i<j)$, with no other cup between them. That is, $i_1=i+1$, $i_r = j_{r-1}+1$ for all $2\leq r \leq t$, and $j_t=j-1$. 
Applying the induction hypothesis with statement~\eqref{prop:bigprop2} to each such cup we have
\begin{eqnarray}\label{eqn.nested1}
N^{\gd(i_r<j_r)}(F_{j_r})=F_{i_r-1} \; \text{ for all } 1\leq r \leq t.
\end{eqnarray}
Since
\[
\gd(i<j)-1=\gd(i_1<j_1)+\gd(i_2<j_2)+\cdots+\gd(i_t<j_t),
\] 
we obtain the first equality of Claim 2 from~\eqref{eqn.nested1}. The second follows by Lemma~\ref{lem.linearalg}~\eqref{lem.linearalg2}.

\

\textbf{Proof of Claim 3:} Recall that $(i^*<j^*)$ denotes the cup immediately above the cup $(i<j)$.
Since $j$ is the right endpoint of a cup with left endpoint $i$, by Corollary~\ref{cor.rank_condition}~\eqref{lem.rank.2} applied with $(\E^{-1}(k)<k )= (i<j)$ and $(k^*<\E(k^*)) = (i^*<j^*)$ we have
\[
N^{\frac{1}{2}(i-i^*-1)+\gd(i<j)}(a_j) \in F_{i^*}\ .
\]
By definition of the preimage and Claim 1 applied with $k=i$ we get
\begin{eqnarray*}
N^{\gd(i<j)} (a_j) \in N^{-\frac{1}{2}(i-i^*-1)}\left(F_{i^*}\right) = F_{i-1},
\end{eqnarray*}
as desired.

\

We can now prove the induction step for the cup $(i<j)$ with $j<n$:

\

\textbf{Proof of~\eqref{prop:bigprop1} for the cup $(i<j)$:}
Let $m<j$ be the largest right endpoint of a cup that appears before $j$. By the induction hypothesis applied to $(\E^{-1}(m)<m)$, $N(a_k)\in F_{k-1}$ for all $1\leq k \leq m$. To complete the proof of~\eqref{prop:bigprop1}, it suffices to prove that $N(a_k)\in F_{k-1}$ for all $k$ such that $m+1\leq k \leq j$. 

First, assume $\gd(i<j)= 1$, so $(i<j)$ has no other cups nested inside it. Thus, $m<i=j-1$ and each $k$ with $m+1\leq k \leq i=j-1$ is the left endpoint of a cup, by maximality of $m$. By Corollary~\ref{cor.rank_condition}\eqref{lem.rank.1} we get that 
\begin{align*}
    N^{\frac{1}{2}(k-k^*+1)}(a_k) \in F_{k^*}\; \text{ for all } \; m+1\leq k \leq i =j-1
\end{align*}
where $(k^*<\E(k^*))$ is the cup directly above the cup with left endpoint $k$. Thus for all $k$ with $m+1\leq k \leq j-1$ we have
\begin{align*}
N^{\frac{1}{2}(k-k^*-1)}\left(N(a_k)\right) \in F_{k^*} &\Rightarrow N(a_k) \in N^{-\frac{1}{2}(k-k^*-1)}(F_{k^*}) \quad \text{(by definition of preimage) }\\
&\Rightarrow N(a_k)\in F_{k-1} \quad \text{ (by Claim~1).}
\end{align*}
Now it only remains to show that $N(a_j)\in F_{j-1}$, but this follows immediately from Claim~3 since $\gd(i<j)=1$ and $F_{i-1} = F_{j-2}\subseteq F_{j-1}$. Thus, $N(a_k)\in F_{k-1}$ for all $k\leq j$ whenever $(i<j)$ has size $1$. 

Now assume $\gd(i<j)\geq 2$. There must be at least one cup nested inside of $(i<j)$, in particular, $j-1$ must be the right endpoint of such a cup. Thus $m=j-1$ in this case and to complete the proof of~\eqref{prop:bigprop1} we need only prove that $N(a_j)\in F_{j-1}$. Since $F_{i-1}\subseteq F_i$ we have 
\begin{align*}
N^{\gd(i<j)}(a_j)\in F_{i} &\Rightarrow N^{\gd(i<j)-1}\left(N(a_j) \right) \in F_{i} \quad \text{ (by Claim~3)} \\
&\Rightarrow  N(a_j) \in N^{-\gd(i<j)+1 } (F_{i}) \quad \text{ (by definition of preimage) }\\
& \Rightarrow N(a_j) \in F_{j-1} \quad \text{ (by Claim~2). }
\end{align*}
We have now established that $N(a_k)\in F_{k-1}$ for all $k\leq j$, proving statement~\eqref{prop:bigprop1} for the cup $(i<j)$. 

\

\textbf{Proof of~\eqref{prop:bigprop2} for the cup $(i<j)$:}
By statement~\eqref{prop:bigprop1} for the cup $(i<j)$, we have $N\left( F_i \right) \subseteq F_{i-1}$. Now Claim~2 yields
\[
N^{\gd(i<j)} \left( F_{j-1} \right) = N\left( F_i \right) \subseteq F_{i-1}.
\]
Thus, $N^{\gd(i<j)}(F_{j})\subseteq F_{i-1}$ follows immediately from $N^{\gd(i<j)}(a_j)\in F_{i-1}$, which is true by Claim~3. Thus statement~\eqref{prop:bigprop2} holds for the cup $(i<j)$ and our proof of the induction step for $j<n$ is complete.

\

Finally, we prove the induction step for the cup $(i<j)$ with $j=n$. Since $\gs$ is prime, we have $(i<j) = (1<n)$. Now $n-1$ must be the right endpoint of a cup, and $N(a_k) \in F_{k-1}$ for all $k\leq n-1$ by the induction hypothesis. This implies $F_{n-1}$ is an $N$-stable hyperplane. Because $N$ is nilpotent, we conclude $N(a_n) \in F_{n-1}$. Thus, statement~\eqref{prop:bigprop1} holds for the cup $(1<n)$. Since $\gd(1<n) = \ell$ and $N^\ell =0$ we also have $N^\ell (F_n)=\{0\}=F_0$, proving~\eqref{prop:bigprop2} is also true in this case. 

The proof is now complete.
\end{proof}

We can now prove that any matrix $A$ with $A\in \cv(\ci_\gs)\cap GL_n(\C)$ represents a flag in the component~$\SF_\gs$. 

\begin{corollary}\label{cor.containment1}
Let $\pi: GL_n(\C) \to GL_n(\C)/B$ denote the projection map from the group of $n\times n$ invertible matrices to the flag variety, $\gs$ be a standard tableau of shape $(n-\ell,\ell)$, and $\ci_\gs$ be the ideal associated to $\gs$. Then 
\[
\cv(\ci_{\gs})\cap GL_n(\C) \subseteq \pi^{-1}(\SF_\gs)
\]
where $\cb_\gs$ denotes the irreducible component of the Springer fiber $\cb_N$ indexed by $\gs$.
\end{corollary}

\begin{proof} By Lemma~\ref{lem.reduction}, it suffices to prove the statement when $\gs$ is prime. Consider the flag $F_\bullet$ represented by a matrix $A\in \cv(\ci_\gs)\cap GL_n(\C)$. Statement~\eqref{prop:bigprop1} of Proposition~\ref{prop.components} implies $N(F_i)\subseteq F_{i-1}$ for all $i$, so $\pi(A) \in \SF_N$. Statement~\eqref{prop:bigprop2} of Proposition~\ref{prop.components} and Proposition~\ref{prop.Fung} together imply $\pi(A)\in \SF_\gs$. 
\end{proof}

%%%%%%%%%%%%

\subsection{The containment $\pi^{-1}(\cb_\gs) \subseteq \cv(\ci_\gs)$}

The results in the previous subsection show that if $A$ is an invertible matrix in $ \cv(\ci_\gs)$  then $A\in \pi^{-1}(\cb_\gs)$. We now show that the reverse containment holds as well. The next two results focus on the case where $\gs$ is a rectangular tableau. 

\begin{lemma}\label{lemma.cellinJ}
For all $\gs\in \SYT(\ell, \ell)$, the Springer--Schubert cell $\cell$ is contained in $\cv(\cj_\gs)$, where $\cj_\gs$ is the monomial ideal from Definition~\ref{def:JSigma}.
\end{lemma}

\begin{proof} Recall that $z_{i,j}, z_{i+\ell, j}\in \cj_\gs$ if and only if $S_\gs(j)<\ell$ and 
\[
S_\gs(j)+1\leq i \leq \ell.
\]
Let $A\in \cell$ and given $p\in [n]$, let $v_p$ denote the $p$-th column of $A$. Assume $S_\gs(p)<\ell$ and let $i$ be such that $S_\gs(p)+1\leq i \leq \ell$. We aim to prove the $i$ and $i+\ell$ entries of $v_p$ are $0$.

Suppose $p$ is the left end point of a cup $(p<\E(p))$ in $\cupsigma$ and let $\ga=\ga(p,\E(p))$. By definition, the entries of $v_p$ in rows $r$ for $1\leq r\leq c(p)$ and $c(p)+\al+2 \leq r \leq c(p)+\al+\ell$ and $c(p)+\al+\ell+2 \leq r \leq n$ are all $0$. By Lemma~\ref{lemma.completedcups}, 
\begin{eqnarray}\label{eqn.Stoc}
S_\gs(p) = c(p)+\gd(p<\E(p))+\al
\end{eqnarray}
so
\[
i\geq S_\gs(p)+1  \Rightarrow   i \geq c(i)+\gd(p<\E(p))+\al +1  \Rightarrow  i \geq c(i)+\al+2
\]
where the last implication follows from the fact that $\gd(p<\E(p))\geq 1$. We also get that $i+\ell \geq c(i)+\al+\ell+2$. Using the description of the zeros in $v_p$ above, it follows that the entries of $v_p$ in rows $i$ and $i+\ell$ are both $0$.

Now suppose $p$ is the right end point of a cup in $\cupsigma$. Then $v_p = e_{c(p)}$ so the only nonzero entry of $v_p$ occurs in position $c(p)$. Note that $c(p)\leq \ell$. Thus, to argue the $i$ and $i+\ell$ entries of $v_p$ are $0$, it suffices to show that $c(p)\leq S_\gs(p)$. This is obviously true by~\eqref{eqn.Stoc} above.

We have shown the $(i,j)$ and $(i+\ell, j)$ entries of $A\in \cell$ are $0$ whenever $z_{ij}\in \cj_\gs$, proving $F_\gs\subseteq \cv(\cj_\gs)$.  
\end{proof}

\begin{prop}\label{lemma.cellinM} Let $\gs\in \SYT(\ell, \ell)$ and suppose $(i<j)$ is a cup of size $\gd(i<j)\geq 2$ in $\cupsigma$. The Springer--Schubert cell $\cell$ is contained in $\cv(\cm_{\gs, (i<j)})$ where $\cm_{\gs, (i<j)}$ is the ideal from Definition~\ref{def:MSigma}.  
\end{prop}

\begin{proof} We will show that the evaluation of the generators of the ideal $\cm_{\gs, (i<j)}$ at any matrix $A$ in $\cell$ is zero. For each $p\in [n]$, let $v_p$ denote the $p$-th column of the matrix $A$. We consider the matrix $M_{(i<j)}(A)$ obtained by evaluating $M_{(i<j)}$ from~\eqref{eqn.matrix} on the matrix $A$. The columns of $M_{(i<j)}(A)$ correspond to the columns $i, i+1, \ldots, j-1$ of $A$, in particular, the $p$-th column of $A$ for $i\leq p \leq j-1$ determines the $(p-i+1)$-st column of $M_{(i<j)}(A)$.

Suppose $p$ is the left end point of a cup in $\cupsigma$ with $i\leq p<\E(p)
\leq j$. By \eqref{eqn.matrix} the $(p-i+1)$-st column of $M_{(i<j)}(A)$ consists of the entries in $v_p$ appearing in rows $S_\gs(p)-\ga(i<j)+t$ and $S_\gs(p)-\ga(i<j)+t+\ell$ where $0\leq t\leq\ga(i<j)$. Using Lemmas~\ref{lem.length} and~\ref{lemma.completedcups} we get that
\begin{align*}
S_\gs(p)-\ga(i<j)+t &= S_\gs(\E(p)) -\delta(p<\E(p)) +1 - \al(i<j) +t \\
& =  c(p) + \al(p<\E(p)) +1 -\al(i<j)+t \\
&=c(p)+\left[\ga(p<\E(p))-\ga(i<j)\right]+t+1.
\end{align*}
Since $(p<\E(p))$ is nested inside $(i<j)$, the difference $\kappa:=\ga(p<\E(p))-\ga(i<j)$ counts the number of cups that lie strictly above $(p<\E(p))$ and below $(i<j)$, including $(i<j)$. In order to compute the $(p-i+1)$-st column of $M_{(i<j)}(A)$ we must identify the entries of $v_p$ in rows $c(p)+\kappa+t+1$ and $c(p)+\kappa +\ell +t+1$ for $0\leq t \leq \al(i<j)$.

Recall from Definition~\ref{def:cell} that the vector $v_p$ is defined as 
\[
v_p=\left(0^{c(p)},a_p,a_{p_1},\cdots,a_{p_{\ga(p<\E(p))}},0^{\ell-1},1,0^{\ell-c(p)-\ga(p<\E(p))-1}\right)^T.
\]
Thus, the entry of $v_p$ in row $c(p)+1$ is $a_p$, then entry in row $c(p)+2$ is $a_{p_1}$, then entry in row $c(p)+3$ is $a_{p_2}$, and so on. In particular, we see that the entry in row $c(p) + \kappa + t +1$ is $a_{p_{\kappa+t}}$ (where we take  $a_p=a_{p_0}$).

Since $(p<\E(p))$ is nested inside $(i<j)$, we know $a_i$ appears in the vector $v_p$ as some entry below $a_p$. By definition of $\kappa$, we get $a_i = a_{p_\kappa}$ and since $\ga(i<j)=\ga(p<\E(p))-\kappa$ we have that $a_{p_{\kappa+t}}=a_{i_{t}}$ (where we take $a_{i_0}=a_i$). Thus, the entries of $v_p$ in positions $k(p)+\kappa+t+1$ for $0\leq t \leq \al(i<j)$ are precisely: $a_i, a_{i_1}, \ldots, a_{i_{\ga(i<j)}}$ where $a_{p_{\al(p<\E(p))}} = a_{p_{\kappa+\al(i<j)}} = a_{i_{\al(i<j)}}$.

By definition of the vector $v_p$, the entry $a_{i_{\ga(i<j)}}$ appears in row $c(p)+\alpha(p<\E(p))+1 = c(p)+\kappa+\ga(i<j)+1$. The next nonzero entry is $1$, which occurs exactly $\ell$ entries after $a_{i_{\ga(i<j)}}$, in row $c(p)+\alpha(p<\E(p))+\ell+1 = c(p)+\kappa+\ga(i<j)+\ell+1$.   This shows that the entries of $v_p$ in the rows $c(p)+\kappa + \ell +t +1$ for $0\leq t \leq \ga(i<j)$ are all $0$, except for the last one, which is equal to $1$.

We showed that the column $(p-i+1)$-st column of $M_{(i<j)}$ when evaluated at a generic matrix in $\cell$ is precisely
\[
\left(a_{i},a_{i_1},\cdots,a_{i_{\al(i<j)}},0^{\ga(i<j)},1\right)^T.
\]

Next, suppose $i< p\leq j-1$ is the right end point of a cup in $\cupsigma$ that is nested inside $(i<j)$. As above, the $(p-i+1)$-st column of $M_{(i<j)}(A)$ consists of the entries in $v_p$ appearing in rows $S_\gs(p)-\ga(i<j)+t$ and $S_\gs(p)-\ga(i<j)+t+\ell$ for $0\leq t\leq\ga(i<j)$. By Lemmas~\ref{lem.rightend} and~\ref{lemma.completedcups}, 
\begin{align*}
S_\gs(p)-\ga(i<j)+t&= c(\E^{-1}(p))+\delta(\E^{-1}(p)<p) + \al(\E^{-1}(p)<p) - \al(i<j) +t \\
&=c(p)+\left[\ga(\E^{-1}(p)<p)-\ga(i<j)\right]+t  \\
&=c(p)+\kappa+t,
\end{align*}
where $\kappa:= \ga(\E^{-1}(p)<p)-\ga(i<j)$. Since $p<j$, we know that $\kappa\geq 1$ and thus $c(p)+\kappa + t \geq c(p)+1$ for all $0\leq t \leq \ga(i<j)$. Since $v_p = e_{c(p)}$, the entry of $v_p$ in row $c(p)+1$ is $0$, the entry in row $c(p)+2$ is $0$, and so on.  In particular, the entries of $v_p$ in rows $c(p)+\kappa+t$ and $c(p)+\kappa+t+\ell$ for all $0\leq t\leq \ga(i<j)$ are $0$. Thus, the $(p-i+1)$-st column of $M_{(i<j)} $ when evaluated at a generic matrix in $\cell$ is the zero vector.

Recall that for every cup $(i<j)$ in $\cupsigma$ with $\gd(i<j)\geq 2 $ the ideal $\cm_{(i<j)}$ is generated by sums of $2\times 2$ minors of the form $\pl_{\{k+1,2\ga(i<j)+2-k\},\{c_1,c_2\}} (M_{(i<j)})$ where $c_1$ and $c_2$ denote column indices of $M_{(i<j)}$. If either $c_1$ or $c_2$ is equal to $p-i+1$ where $p$ is a the right endpoint of a cup nested in $(i<j)$, then the corresponding column of $M_{(i<j)}(A)$ for $A\in \cell$ is the zero vector and each minor $\pl_{\{k+1,2\ga(i<j)+2-k\},\{c_1,c_2\}}(M_{(i<j)})$ evaluated on $A\in \cell$ is obviously zero. 

To complete the proof, it suffices to consider the case where both $c_1$ and $c_2$ are equal to $p-i+1$ for some $i\leq p<j$ indexing a left end point of a cup in $\cupsigma$. In this case, the submatrix of $M_{(i<j)}(A)$ formed by these two columns is 
\[
\begin{bmatrix}
    a_{i} & a_{i}\\
    a_{i_1}& a_{i_1}\\
    \vdots& \vdots \\
    a_{i_{\ga(i<j)}}&a_{i_{\ga(i<j)}} \\
    0& 0\\
    \vdots&\vdots \\
    0& 0\\
    1& 1 \\
\end{bmatrix}.
\]
The entries in the column indexed by $c_1$ are equal to the entries in the column indexed by in $c_2$ and therefore each of the minors $\pl_{\{k+1,2\ga(i<j)+2-k\},\{c_1,c_2\}}(M_{(i<j)})$ evaluated on $A\in \cell$ is zero.
\end{proof}

Recall that $\pi: GL_n(\C) \to GL_n(\C)/B$. We now return to the situation where our tablau $\gs$ is shape $(n-\ell, \ell)$. 

\begin{corollary} \label{cor.cellinI}
For each standard tableau $\gs$ of shape $(n-\ell, \ell)$, the preimage $\pi^{-1}(\cell)$ of the Springer--Schubert cell $\cell$ is contained in $\cv(\ci_\gs)$. 
\end{corollary}
\begin{proof} If $n= 2\ell$ then $\gs$ is rectangular and combining Lemma~\ref{lemma.cellinJ} and Lemma~\ref{lemma.cellinM} we get that $\cell\subseteq \cv(\ci_\gs)$.  If $n\neq 2\ell$, then $\gs$ is not prime and we consider the prime decomposition $\gs=\gs_1\circ \gs_2 \circ \cdots \circ \gs_k$. Let $A\in \cell$. By Lemma~\ref{lem.taushift}, there exists $A_i\in F_{\gs_i}$ for $i=1,2,\ldots, k$ such that $\dot\tau_\mu A = \iota(A_1, A_2, \ldots, A_k)$. Since each $\gs_i$ is rectangular, $A_i\in \cv(\ci_{\gs_i})$ for each $i$ and it follows that $\dot\tau_\mu A \in \cv(\ci_\gs')$ and thus $A\in \cv(\ci_{\gs})$. This shows $\cell \subseteq \cv(\ci_\gs)$ for all $\gs\in \SYT(n-\ell, \ell)$.

Finally, every element of $\pi^{-1}(\cell)$ is of the form $Ab$ for some $A\in \cell$ and $b\in B$. The result now follows from the fact that $\ci_\gs$ is right $B$-invariant by Proposition~\ref{binv}. 
\end{proof}

We can now complete the proof of our main result.

\begin{proof}[Proof of Theorem {\ref{thm.ideal}}] The fact that $\ci_\gs$ is right $B$-invariant is Proposition~\ref{binv}.  By Corollary~\ref{cor.containment1} in the previous subsection, $\cv(\ci_\gs)\cap GL_n(\C) \subseteq \pi^{-1}(\cb_\gs)$. On the other hand, by Corollary~\ref{cor.cellinI} above we have $\pi^{-1}(\cell) \subseteq \cv(\ci_\gs)\cap GL_n(\C)$ and because $\cell$ is a dense open subset of $\SF_\gs$, $\dim (\pi^{-1}(\SF_\gs)) \leq \cv(\ci_\gs)\cap GL_n(\C)$. Since $\pi^{-1}(\SF_\gs)$ is irreducible, the inclusion $\cv(\ci_\gs)\cap GL_n(\C) \subseteq \pi^{-1}(\cb_\gs)$ now implies $\cv(\ci_\gs)\cap GL_n(\C) = \pi^{-1}(\cb_\gs)$.
\end{proof}

%%%%%%%%%%%%%%%%%%%%%%%%%%%%%%%%%%%%%%

\section{Conjectures on the cohomology class of $\SF_\gs$} \label{sec:cohomology}

In this section, we use the ideal $\ci_\gs$ to conjecture two formulas for the cohomology class $[\SF_\gs]$ of the component $\SF_\gs$ whenever $\gs$ has two rows. We prove our conjectures for a particular family of noncrossing matchings (the ``one big cup'' family) and corresponding tableaux.

\subsection{Cohomology classes via combinatorial commutative algebra} 
Recall that for each subvariety $Y$ of the flag variety $\Fl_n(\C)$, there is a well defined cohomology class $[Y]\in H^*(\Fl_n(\C); \Z)$. Let $I^+$ be the ideal in $\Z[x_1, x_2, \ldots, x_n]$ generated by constant-free symmetric polynomials. 
The Borel presentation~\cite{Borel}: 
\[
H^*(\Fl_n(\C); \Z) \simeq \Z[x_1, x_2, \ldots, x_n]/I^+
\]
allows us to identify the class $[Y]$ as a polynomial in $x_1,\ldots, x_n$, defined uniquely up to the quotient by $I^+$ (see also~\cite{Manivel, BilleyGaoPawlowski25}). 

Let $\gs\in \SYT(n-\ell, \ell)$. We can use the ideal $\isigma$ to compute the cohomology class of each irreducible component $\cb_\gs$. To do this, we leverage combinatorial commutative algebra in the polynomial ring $\C[\mathbf{z}]$ using techniques of Knutson--Miller~\cite{KnutsonMiller2005}. Define a $\Z^d$-grading  on the polynomial ring $\C[\mathbf{z}]$ by $\operatorname{deg}\left(z_{i j}\right)=e_j$, where $e_j$ is the $j$-th coordinate vector in $\mathbb{Z}^n$. By construction, the ideal $\ci_\gs$ defined in Section \ref{sec:IdealIsig} is homogeneous with respect to this grading. This means that we can associate a graded Hilbert series to the scheme $\spec \C[\mathbf{z}]/\ci_\gs$. For a formal definition of the graded Hilbert series we refer the reader to \cite[\S 8.1, 8.2]{MillerSturmfels2005}.

The graded Hilbert series determines the \emph{multidegree} $\cc(\C[\mathbf{z}]/\ci_\gs; \mathbf{x})$ of $\spec \C[\mathbf{z}]/\ci_\gs$. The multidegree is a polynomial in $\mathbb{Z}[\mathbf{x}]=\mathbb{Z}\left[x_1, x_2, \ldots, x_d\right]$ of degree $\dim \left(\C[\mathbf{z}]/\ci_\gs\right)$, defined using a particular shift and truncation of the graded Hilbert series. A precise definition can be found in \cite[\S 8.5]{MillerSturmfels2005}. The following lemma tells us that the multidegree of $\spec \C[\mathbf{z}]/\ci_\gs$ determines the cohomology class $[\cb_\gs]$. Recall that $\pi: GL_n(\C) \rightarrow GL_n(\C) / B\simeq \Fl_n(\C)$ denotes the natural projection. The following is~\cite[Corollary 2.3.1]{KnutsonMiller2005}; see also~\cite[Proposition 3]{ITW18}.

\begin{lemma}[Knutson--Miller]\label{lemma.multidegree} Let $Y \subseteq \Fl_n(\C)$ be a closed scheme, and let $\tilde{Y}$ be a closed subscheme of the affine scheme $M_n=\spec \mathbb{C}[\mathbf{z}]$ of $n \times n$ matrices such that $\tilde{Y} \cap GL_n(\C)=\pi^{-1}(Y)$. Then the cohomology class $[Y]$ in $H^*(\Fl_n(\C);\Z)$ is represented by the multidegree $\cc(\tilde{Y};\mathbf{x})$ with respect to the grading $\operatorname{deg}\left(z_{i j}\right)=e_j$.
\end{lemma}

In the next section, we apply this lemma with $Y=\cb_\gs$ and $\tilde{Y} = \spec \C[\mathbf{z}] / \ci_\gs$. Lemma~\ref{lemma.multidegree} applies because Theorem~\ref{thm.ideal} shows that for every two-row tableaux $\gs$, the associated ideal $\isigma$ is such that $\pi^{-1}(\cb_{\gs}) = \cv(\ci_{\gs})\cap GL_n(\C)$. As a result, we can compute $[\cb_\gs]$ for two-row tableaux to using commutative algebra. The data so obtained motivates two conjectural formulas, which appear in the next section.

%%%%%%%%%%%%%%%%%%%%%

\subsection{Two conjectural formulas}
Using the combinatorial data coming from noncrossing partitions and the definition of the ideal $\ci_\gs$, we establish two conjectures computing $[\SF_\gs]$. We prove these conjectures for a specific family of two-row tableaux and have used SageMath~\cite{sagemath} and Macaulay2~\cite{M2} to check that they hold up to $n=8$.

Our first conjecture (a restatement of Conjecture~\ref{conj1.intro} from the introduction) computes the cohomology classes for two-row tableaux of shape $(\ell,\ell)$ as a sum of monomials. Recall that $S_\gs$ denotes the nesting sequence of the standard noncrossing matching $\cupsigma$.

\begin{conj}\label{conj:monomialproduct} 
Let $\sigma$ be a standard tableau of shape $(\ell, \ell)$. The cohomology class of $\cb_\gs$ is 
\[[\cb_\gs]=
\prod_{\substack{j\in [n]\\S_\gs(j)<\ell}} x_j^{2(\ell-S_\gs(j))}  \cdot  \prod_{\substack{(i<j)\\ \delta(i<j)\geq 2}} \left(  \sum_{k=i}^{j-1} x_ix_{i+1}\cdots \hat{x}_k \cdots x_{j-1} \right)
\]
where the product is taken over all cups $(i<j)$ in $\cupsigma$ with size at least $2$. 
\end{conj}

\begin{ex}\label{ex.min4} We compute the classes $[\SF_\gs]$ for the two standard tableaux of shape $(2,2)$; these examples also appear in Figure~\ref{fig:intro2} of the introduction.  Consider the following tableau $\gs$ of shape $(2,2)$ and its associated noncrossing matching $\cupsigma$ displayed to the right.
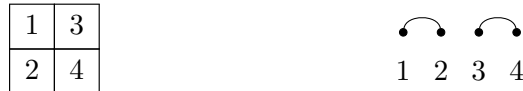
\begin{multicols}{2}
   \begin{center}
      \begin{ytableau}
        1&3\\
        2&4
    \end{ytableau}
   \end{center} 
\vspace*{-.1in}
\begin{tikzpicture}[scale=0.5]
	\begin{scope}[shift={(8,0)}]
		\foreach \i in {0,...,3}
		{
			\node[pnt] at (\i,0)(\i){};
		}

		\draw (0)  to [bend left=80](1);
		\draw(2)  to [bend left=80] (3);

		\draw (0,-1) node {$1$};
       
        \draw (1,-1) node {$2$};
      
        \draw (2,-1) node {$3$};
        \draw (3,-1) node {$4$};
        
	\end{scope}
\end{tikzpicture} 
\end{multicols}
\noindent We have $\ell=2$ and  $S_\gs=(1,1,2)$, so $\{j\in[4]\mid S_\gs(j)<\ell\}=\{1,2\}$. Furthermore, $\cupsigma$ has no cup of size 2 or greater. By Conjecture~\ref{conj:monomialproduct} we have that
\begin{equation*}
[\cb_\gs]=x_1^{2(\ell-S_\gs(j))}  x_2^{2(\ell-S_\gs(j))} =x_1^{2(2-1)}  x_2^{2(2-1)} =x_1^2 x_2^2.
\end{equation*}

Next, consider the following tableau $\gs$ of shape $(2,2)$ and its associated noncrossing matching $\cupsigma$ displayed to the right.
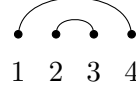
\begin{multicols}{2}
   \begin{center}
      \begin{ytableau}
        1&2\\
        3&4
    \end{ytableau}
   \end{center} 
\vspace*{-.1in}
\begin{tikzpicture}[scale=0.5]
	\begin{scope}[shift={(8,0)}]
		\foreach \i in {0,...,3}
		{
			\node[pnt] at (\i,0)(\i){};
		}
        
		\draw (0)  to [bend left=80](3);
		\draw(1)  to [bend left=80] (2);

		\draw (0,-1) node {$1$};
        \draw (1,-1) node {$2$};
        \draw (2,-1) node {$3$};
        \draw (3,-1) node {$4$};
	\end{scope}
\end{tikzpicture} 
\end{multicols}
\noindent We have $\ell=2$ and  $S_\gs=(1,2,2)$, so $\{j\in[n]\mid S_\gs(j)<\ell\}=\{1\}$. Note that $\gd(1<4)=2$. By Conjecture~\ref{conj:monomialproduct} we have,
\begin{equation*}
\begin{split}
[\cb_\gs]&=x_1^{2(\ell-S_\gs(1))} \left(x_1x_2+x_1x_3+x_2x_3\right)=x_1^{2(2-1)}\left(x_1x_2+x_1x_3+x_2x_3\right)\\
&=x_1^2\left(x_1x_2+x_1x_3+x_2x_3\right)=x_1^3x_2+x_1^3x_3+x_1^2x_2x_3 \ . 
\end{split}
\end{equation*}    

It is straightforward to confirm the validity of these formulas using Lemma~\ref{lemma.multidegree}.
\end{ex}

\begin{ex}\label{2and3rainbow}
Consider the following tableau $\gs$ of shape $(5,5)$ and its associated noncrossing matching $\cupsigma$ displayed to the right.
 \begin{multicols}{2}
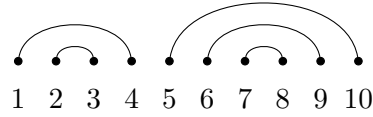
\begin{center}
     \begin{ytableau}
        1&2&3&4&5\\
        6&7&8&9&10
    \end{ytableau}
\end{center}
 \vspace*{-.3in}
\begin{tikzpicture}[scale=0.5]
	\begin{scope}[shift={(8,0)}]
		\foreach \i in {0,...,9}
		{
			\node[pnt] at (\i,0)(\i){};
		}

		\draw (0)  to [bend left=80](3);
		\draw(1)  to [bend left=80] (2);

		\draw (0,-1) node {$1$};
      
        \draw (1,-1) node {$2$};
        \draw (2,-1) node {$3$};
        \draw (3,-1) node {$4$};
       \draw(4)  to [bend left=80] (9);
		\draw(5)  to [bend left=80] (8);
		\draw(6)  to [bend left=80] (7);
		\draw (4,-1) node {$5$};
       
        \draw (5,-1) node {$6$};
        \draw (6,-1) node {$7$};
        \draw (7,-1) node {$8$};
        \draw (8,-1) node {$9$};
        \draw (9,-1) node {$10$};
	\end{scope}
\end{tikzpicture} 
\end{multicols}    
\noindent We have $\ell=5$ and  $S_\gs=(1,2,2,2,3,4,5,5,5)$, so $\{j\in[n]\mid S_\gs(j)<\ell\}=\{1,2,3,4,5,6\}$. There are three cups of size greater than one in $\cupsigma$: $(1<4)$, $(5<10)$, and $(6<9)$. By Conjecture~\ref{conj:monomialproduct} we have 
\begin{align*}
[\cb_\gs] = x_1^8 x_2^6 x_3^6 x_4^6 x_5^4 x_6^2 & (x_1x_2+x_1x_3+x_2x_3) (x_5x_6x_7x_8 + \cdots + x_6x_7x_8x_9) \\
&\qquad \cdot (x_5x_6x_7x_8 + x_5x_6x_7x_9 + \cdots + x_6x_7x_8x_9). \qedhere
\end{align*}
\end{ex}

The Schubert polynomials, defined by Lascoux and Schützenberger \cite{LascouxSchutzenberger1982}, form a basis for the polynomial ring  $H^*(\Fl_n(\C); \Z)\simeq \Z[\mathbf{x}]/I^+$. Given $w\in S_n$ we denote by $\mathfrak{S}_w(\mathbf{x})$ the Schubert polynomial of $w$.
The polynomial $\mathfrak{S}_w(\mathbf{x})$ can be computed using the recurrence
$$
\mathfrak{S}_w (\mathbf{x})= \begin{cases}x_1^{n-1} x_2^{n-2} \cdots x_{n-2}^2x_{n-1} & \text { if } w=w_0  \\ \partial_i \mathfrak{S}_{w s_i}(\mathbf{x}) & \text { if } w(i)<w(i+1),\end{cases}
$$
where $w_0 = [n, n-1, \ldots, 1]$ is the longest permutation, $s_i$ is the simple reflection exchanging $i$ and $i+1$, and the divided difference operator $\partial_i$ acts on $f\in \Z[\mathbf{x}]$ by
$$
\partial_i f(\mathbf{x})=\frac{f(\mathbf{x})-s_i f(\mathbf{x})}{x_i-x_{i+1}} .
$$
Note that $\partial_i^2=0$ and $\partial_i \partial_j = \partial_j\partial_i$ for all $i$ and $j$ with $|i-j|\geq 2$. 

Our second conjecture (a restatement of Conjecture~\ref{conj2.intro} from the introduction)  uses divided difference operators to give a combinatorial formula for the expansion of $[\SF_\gs]$ in the Schubert basis.

\begin{conj}\label{conj:divdiff} Let $\gs$ be a standard tableau of shape $(\ell,\ell)$ with corresponding cup diagram $\cupsigma$.  Let $(i_1<j_1), (i_2<j_2), \ldots, (i_\ell<j_\ell)$ be any total ordering of the cups in $\cupsigma$ such that when a cup appears, every cup nested inside of that one has already appeared. 
Then 
\[
[\SF_\gs] = \left( \partial_{i_1} + \cdots + \partial_{j_1-1} \right) \left( \partial_{i_2} + \cdots + \partial_{j_2-1} \right) \cdots  \left( \partial_{i_\ell} + \cdots + \partial_{j_\ell-1} \right)\mathfrak{S}_{w_0}(\mathbf{x}).
\]
\end{conj}

As we point out below (cf.~Corollary~\ref{conj:divdiff2}) the conjecture above can also be stated for any two-row tableau.

Table~\ref{table.n6} displays all standard tableau of shape $(3,3)$ next to their cohomology classes computed using Conjectures~\ref{conj:monomialproduct} and~\ref{conj:divdiff}. In each case, we have confirmed the validity of these results using Lemma~\ref{lemma.multidegree}.

\begin{table}[h!]
\centering
\setlength{\tabcolsep}{8pt}
\newcolumntype{C}{>{\centering\arraybackslash}m{2.5cm}}
\begin{tabularx}{\textwidth}{C >{\raggedright\arraybackslash}m{6cm} X}
\toprule
$\sigma$ & Monomial expansion of~$[\SF_\sigma]$ & Schubert expansion of $[\SF_\sigma]$ \\ 
\midrule

% Row 1
\begin{ytableau} 1&3&5 \\ 2&4&6 \end{ytableau} & 
$x_1^4x_2^4x_3^2x_4^2$ & 
$\mathfrak{S}_{563412}$ \\ 
\addlinespace[10pt]

% Row 2
\begin{ytableau} 1&2&5 \\ 3&4&6 \end{ytableau} & 
$x_1^4x_2^2x_3^2x_4^2(x_1x_2+x_1x_3+x_2x_3)$ & 
$\mathfrak{S}_{546312} + \mathfrak{S}_{635412}$ \\ 

\addlinespace[10pt]

% Row 3
\begin{ytableau} 1&3&4 \\ 2&5&6 \end{ytableau} & 
$x_1^4x_2^4x_3^2(x_3x_4+x_3x_5+x_4x_5)$ & 
$\mathfrak{S}_{563241} + \mathfrak{S}_{564132}$ \\ 
\addlinespace[10pt]

% Row 4
\begin{ytableau} 1&2&4 \\ 3&5&6 \end{ytableau} & 
$x_1^4x_2^2x_3^2(x_1x_2x_3x_4 + \dots + x_2x_3x_4x_5)$ & 
$\mathfrak{S}_{645132} + \mathfrak{S}_{546231} + \mathfrak{S}_{635241}$ \\ 
\addlinespace[10pt]

% Row 5
\begin{ytableau} 1&2&3 \\ 4&5&6 \end{ytableau} & 
$x_1^4x_2^2(x_1x_2x_3x_4 + \dots + x_2x_3x_4x_5) (x_1x_2+x_1x_3+x_2x_3)$ & 
$\mathfrak{S}_{543621} + \mathfrak{S}_{634521} + \mathfrak{S}_{562431} + 2\mathfrak{S}_{642531} + \mathfrak{S}_{651432} + \mathfrak{S}_{643512} + \mathfrak{S}_{652341}$ \\ 
\end{tabularx}
\caption{Cohomology classes $[\SF_\gs]$ for all standard tableau $\gs$ of shape $(3,3)$.}
\label{table.n6}
\end{table}

\begin{ex} This example computes the Schubert expansion for the tableaux $\gs$ appearing in Example~\ref{ex.min4} (cf.~Figure~\ref{fig:intro2}) and in Example~\ref{2and3rainbow}. When $\sigma_1$ and $\sigma_2$ are, respectively, the first and second tableau from Example~\ref{ex.min4} we apply the formula from Conjecture \ref{conj:divdiff} to obtain:
\begin{align*}
[\cb_{\gs_1}]&=\partial_1\partial_3\mathfrak{S}_{w_0} = \mathfrak{S}_{w_0s_1s_3} =x_1^2x_2^2 \\
[\cb_{\gs_2}]&= \partial_2(\partial_1+\partial_2+\partial_3)\mathfrak{S}_{w_0}=\partial_2(\mathfrak{S}_{w_0s_1} + \mathfrak{S}_{w_0s_2}+\mathfrak{S}_{w_0s_3}) \\
&\qquad \qquad = \mathfrak{S}_{w_0s_1s_2}+ \mathfrak{S}_{w_0s_3s_2} =  \mathfrak{S}_{3241}+\mathfrak{S}_{4132} = x_1^3x_2+x_1^3x_3+x_1^2x_2x_3,
\end{align*}
which confirms the calculations of Example~\ref{ex.min4} and Figure~\ref{fig:intro2}. Next, let $\gs$ be as in Example~\ref{2and3rainbow}. By Conjecture~\ref{conj:divdiff} we have 
\begin{align*}
[\cb_\gs]&=\partial_2\partial_7(\partial_1+\partial_2+\partial_3)(\partial_6+\partial_7+\partial_8) (\partial_5+\partial_6+\partial_7+\partial_8+\partial_9)\mathfrak{S}_{w_0}\\
&=\mathfrak{S}_{[9,8,10,7,5,4,3,6,2,1]}+ \mathfrak{S}_{[10,7,8,9,5,4,3,6,2,1]}+ \mathfrak{S}_{[9,8,10,7,6,3,4,5,2,1]}+ \mathfrak{S}_{[10,7,8,9,6,3,4,5,2,1]}\\
&\quad +\mathfrak{S}_{[9,8,10,7,5,6,2,4,3,1]}+ \mathfrak{S}_{[10,7,8,9,5,6,2,4,3,1]}+ 2\mathfrak{S}_{[9,8,10,7,6,4,2,5,3,1]}+ 2\mathfrak{S}_{[10,7,8,9,6,4,2,5,3,1]}\\
&\quad+\mathfrak{S}_{[9,8,10,7,6,5,1,4,3,2]} +\mathfrak{S}_{[10,7,8,9,6,5,1,4,3,2]}+\mathfrak{S}_{[9,8,10,7,6,4,3,5,1,2]} +\mathfrak{S}_{[10,7,8,9,6,4,3,5,1,2]}\\
& \quad+\mathfrak{S}_{[9,8,10,7,6,5,2,3,4,1]}+\mathfrak{S}_{[10,7,8,9,6,5,2,3,4,1]}\ . 
\end{align*}
It is straightforward to confirm that this is equal to the formula computed in Example~\ref{2and3rainbow}.
\end{ex}

%%%%%%%%%%%%%%%%

\subsection{The one big cup family}
From here on we focus our attention on particular family of two-row tableaux, and show that our conjectures hold in this case. We say a tableau $\gs\in \SYT(\ell, \ell)$ is the prime tableau with \emph{one big cup} if the associated perfect matching consists of the cup $(1<n)$ with only cups of size one nested inside. Formally, $\gs$ is of the form:
\[
\begin{array}{|c|c|c|c|c|c|}
\hline
1 & 2 & 4 & \dots & 2(n-3) & 2(n-2) \\ \hline
3 & 5 & 7 & \dots & n-1 & n \\ \hline 
\end{array} \ .
\]

Recall that for an ideal $\mathcal{K} \subseteq \C[\mathbf{z}]$ and a fixed monomial order, the initial ideal $\mathrm{in}(\mathcal{K})$ is the monomial ideal generated by the leading terms of all polynomials in $\mathcal{K}$. A Gr\"obner basis for $\mathcal{K}$ is a generating set whose leading terms generate $\mathrm{in}(\mathcal{K})$. The following lemma shows that for the one big cup case, the initial ideal of $\isigma$ takes a particularly simple form under any \emph{diagonal term order}, which is any monomial order such that the leading term of the determinant of any submatrix of $Z =(z_{i,j})_{1 \leq i,j \leq n}$ is the product of its diagonal entries.

\begin{lemma}\label{lemma.Grobner} Suppose $\gs\in \SYT(\ell,\ell)$ is the prime tableau with one big cup. Then $\ci_\gs$ is prime, and furthermore, the generators used in define $\ci_\gs$ in Section~\ref{sec:IdealIsig} form a Gr\"obner basis of $\ci_\gs$ relative to a diagonal term order. 
\end{lemma}
\begin{proof}  In this case we have $\ci_\gs = \cj_\gs+\cm_{\gs, (1<n)}$ where $\cm_{\gs,(1<n)}$ is the ideal generated by all $2\times 2$ minors of the matrix:
\begin{align*}
M_{(1<n)} &= \begin{bmatrix} z_{S_\gs(1), 1} & z_{S_{\gs}(2), 2} & \cdots & z_{S_{\gs}(n-1), n-1}  \\
z_{S_\gs(1)+\ell, 1} & z_{S_{\gs}(2)+\ell, 2} & \cdots & z_{S_{\gs}(n-1)+\ell, n-1} 
\end{bmatrix} \\
&=  \begin{bmatrix} z_{1, 1} & z_{2, 2} & z_{2, 3} & z_{3,4} & z_{3,5} & \cdots & z_{\ell, n-2} & z_{\ell, n-1}  \\
z_{1+\ell, 1} & z_{2+\ell, 2} & z_{2+\ell, 3} & x_{3+\ell, 4} & z_{3+\ell, 5} & \cdots & z_{n, n-2} & z_{n, n-1} 
\end{bmatrix} \ .
\end{align*}
In particular, since both $\cj_\gs$ and $\cm_{\gs, (1<n)}$ are prime and their generating sets have no variables in common, the ideal $\isigma$ is prime.

Sturmfels~\cite[Theorem 1]{Sturmfels1990} showed that the ideal generated by $(r+1)\times(r+1)$ minors of an $m\times n$ matrix form a Gr\"obner basis for the ideal defining the variety of $m\times n$ matrices of rank at most $r$ with respect to any diagonal term order. In particular, this implies that the generating set of $\cm_{\gs,(1<n)}$ forms a Gr\"obner basis with respect to any diagonal term order. Once again, since the generating sets of $\cm_{\gs,(1<n)}$ and $\cj_\gs$ have no variables in common, we get that the generators given in Section~\ref{sec:IdealIsig} are a Gr\"obner basis of $\ci_\gs$ relative to any diagonal term order. 
\end{proof}

\begin{remark} The results of Lemma \ref{lemma.Grobner} do not generalize to all two-row prime tableaux. For example, consider the case where $\sigma$ is the standard tableau whose diagram is the ``rainbow'' from Example \ref{1component.a}. Although this ideal $\ci_\gs$ is prime, the specified generating set fails to be a Gr\"obner basis relative to any term order.  The challenge of identifying a computationally efficient Gr\"obner basis for ideals associated to larger two-row prime tableaux remains an open problem, and a significant roadblock to applying techniques from commutative algebra to prove  Conjectures~\ref{conj:monomialproduct} and~\ref{conj:divdiff}.
\end{remark}

\begin{prop}\label{mononebigcup}
Conjecture~\ref{conj:monomialproduct} holds for the prime tableau $\gs\in \SYT(\ell, \ell)$ with one big cup, that is, 
\[
[\cb_\gs] = x_1^{n-2}x_2^{n-4}x_3^{n-4}x_4^{n-6}x_5^{n-6} \cdots x_{n-4}^2x_{n-3}^2  \left(  \sum_{k=1}^{n-1} x_1x_{2}\cdots \hat{x}_k \cdots x_{n-1} \right)  .
\]
\end{prop}

\begin{proof} We prove the desired formula using Lemma~\ref{lemma.multidegree}. We leverage three key properties of the multidegree; recall that we are using the degree function $\deg(z_{i,j})=e_j$.
\begin{enumerate}
\item \cite[Corollary 8.47]{MillerSturmfels2005}: The multidegree is a degenerative function, and in particular, $\cc(\C[\mathbf{x}]/\ci_\gs; \mathbf{x}) = \cc(\C[\mathbf{z}]/\mathrm{in}(\ci_\gs); \mathbf{x})$.
\item \cite[Theorem 8.44]{MillerSturmfels2005}: If $\cp = \left< z_{i_1,j_1}, z_{i_2,j_2} , \ldots, z_{i_t,j_t} \right>$ is a prime monomial ideal then
\[
\mathcal{C}\left(\C[\mathbf{z}] /\cp ; \mathbf{x}\right)= x_{j_1}x_{j_2}\cdots x_{j_t}.
\]
\item \cite[Theorem 8.53]{MillerSturmfels2005}: Suppose $\ci$ is a reduced ideal with maximum dimensional associated primes $\cp_1, \cp_2, \ldots, \cp_k$. Then 
\[
\cc(\C[\mathbf{z}] / \ci; \mathbf{x}) = \sum_{i=1}^k \cc(\C[\mathbf{z}]/ \cp_i; \mathbf{x}).
\]
\end{enumerate}
Fix a diagonal term order on the variable set $\mathbf{z}$. Combining Lemma~\ref{lemma.multidegree} with Fact (1), it suffices to prove
\[
\cc(\C[\mathbf{z}] / \mathrm{in}(\ci_\gs);
\mathbf{x}) = x_1^{n-2}x_2^{n-4}x_3^{n-4}x_4^{n-6}x_5^{n-6} \cdots x_{n-4}^2x_{n-3}^2  \left(  \sum_{k=1}^{n-1} x_1x_{2}\cdots \hat{x}_k \cdots x_{n-1} \right)  ,
\]
where, by Lemma~\ref{lemma.Grobner}, 
\begin{eqnarray}\label{eqn.leadterms}
\mathrm{in}(\ci_\gs) = \cj_\gs + \langle z_{S_\gs(i),i}\,z_{S_\gs(k)+\ell,k} \mid 1\leq i<k\leq n-1\rangle.
\end{eqnarray}
To complete the proof, we apply Stanley--Reisner theory for square-free monomial ideals~\cite[Chapter 1]{MillerSturmfels2005}.  Since $\mathrm{in}(\ci_\gs)$ is square-free, it is equal to the ideal $\left< \mathbf{z}^\tau \mid \tau \notin \Delta \right>$ generated by monomials corresponding to nonfaces of a simplicial complex $\Delta$. To complete the proof, we compute $\Delta$ and use it to find the associated primes of $\mathrm{in}(\ci_\gs)$.

By definition, the simplicial complex $\Delta$ has faces equal to subsets of the variable set asssociated to monomials which are not in the initial ideal of $\ci_\gs$. Using~\eqref{eqn.leadterms}, the monomials which are not in $\mathrm{in}(\ci_\gs)$ are precisely those which are not divisible by 
\[
z_{S_\gs(i),i}\,z_{S_\gs(k)+\ell,k} \;\; \text{ for any } 1\leq i <k \leq n-1.
\]
In particular, we get 
\[
\Delta = \bigcup_{i=1}^{n-1} \left\{ z_{S_\gs(1),1}, \ldots, z_{S_\gs(n-i), n-i}, z_{S_\gs(1)+\ell,1}, \ldots, z_{S_\gs(i)+\ell, i} \right\}.
\]
By~\cite[Theorem 1.7]{MillerSturmfels2005}, each maximum dimensional associated prime of $\mathrm{in}(\ci_\gs)$ is of the form: 
\begin{align*}
\cp_0 &= \cj_\gs+\left< z_{S_\gs(1),1}, z_{S_\gs(2),2},\, \ldots,\, z_{S_\gs(n-2),n-2} \right> \\
\cp_i &= \cj_\gs+\left< z_{S_\gs(1),1}, \ldots, z_{S_\gs(i),i}, z_{S_\gs(i+2)+\ell, i+2}, \ldots,  z_{S_\gs(n-1)+\ell, n-1} \right> \; \text{ for $1\leq i \leq n-3$} ,\; \\
\cp_{n-2} &= \cj_\gs+\left< z_{S_\gs(2)+\ell,2}, z_{S_\gs(3)+\ell,3}, \ldots, z_{S_\gs(n-1)+\ell,n-1} \right> .
\end{align*} 
Since $S_\gs=(1,2,2,3,3,\ldots, \ell, \ell )$ we have by Fact (2) that  
\[
\cc(\C[\mathbf{z}]/\cj_\gs; \mathbf{x}) = x_1^{n-2}x_2^{n-4}x_3^{n-4}x_4^{n-6}x_5^{n-6} \cdots x_{n-4}^2x_{n-3}^2.
\]
Fact (2) also implies 
\begin{align*}
\cc(\C[\mathbf{z}]/\cp_{0} ; \mathbf{x}) &= \cc(\C[\mathbf{z}]/\cj_\gs; \mathbf{x}) (x_1x_2 \cdots x_{n-2})\\
\cc(\C[\mathbf{z}]/\cp_{i} ; \mathbf{x}) &= \cc(\C[\mathbf{z}]/\cj_\gs; \mathbf{x}) (x_1x_2 \cdots \hat{x}_{i+1} \cdots x_{n-1})\; \text{ for $1\leq i \leq n-3$} ,\; \\
\cc(\C[\mathbf{z}]/\cp_{n-1} ; \mathbf{x}) &= \cc(\C[\mathbf{z}]/\cj_\gs; \mathbf{x}) (x_2x_3 \cdots x_{n-2}).
\end{align*}
The desired formula now follows directly from Fact (3).
\end{proof}

\begin{prop}\label{prop.divdiff.onebigcup}
Conjecture ~\ref{conj:divdiff} holds for the prime tableau $\gs\in \SYT(\ell, \ell)$ with one big cup, that is, 
\begin{equation}\label{conj:divbigcup}
[\cb_\gs]=\partial_2\partial_4\cdots\partial_{2\ell-2}\left(\partial_1+\partial_2+\partial_3+\cdots+\partial_{2\ell-2}+\partial_{2\ell-1}\right)\mathfrak{S}_{w_0} (\mathbf{x}).
\end{equation}
\end{prop}

\begin{proof} Since $\partial_i^2=0$, it suffices to prove
\[
[\cb_\gs]=\partial_2\partial_4\cdots\partial_{2\ell-2}\left(\partial_1+\partial_3+\cdots+\partial_{2\ell-1}\right)\mathfrak{S}_{w_0} (\mathbf{x}).
\]
We compute the effect of applying an odd indexed divided difference operator followed by a sequence of even indexed divided difference operators. 

For any $1\leq i\leq \ell$ the divided difference operator $\partial_{2i-1}$ acts on $\mathfrak{S}_{w_0} = x_1^{n-1}x_2^{n-2}\cdots x_{n-1}$ by decreasing the power of $x_{2i-i}$ by $1$:
\[
\partial_{2 i-1}  \mathfrak{S}_{w_0}(\mathbf{x}) =x_1^{n-1} x_2^{n-2}\cdots x_{2(i-1)}^{n-2(i-1)} x_{2 i-1}^{n-2 i} x_{2 i}^{n-2 i} \cdots x_{n-1}.
\]
Thus, $x_{2i-1}$ and $x_{2i}$ have the same power in $\partial_{2 i-1}  \mathfrak{S}_{w_0}$, and the power of $x_{2(i-1)}$ is now two more than that of $x_{2i-1}$.

Next, we study the effect of applying the chain of even indexed divided difference operators $\partial_2\partial_4\cdots\partial_{2\ell-2}$ to  $\partial_{2 i-1} \mathfrak{S}_{w_0}$. Since $\partial_2, \partial_4, \ldots, \partial_{2\ell -2}$ all commute, it suffices to compute $\partial_{2k} \partial_{2i-1} \mathfrak{S}_{w_0}$ for each $k=1, \ldots, \ell -1$. 
    
If $k \neq i-1$, then applying $\partial_{2k}$ to $\partial_{2i-1} \mathfrak{S}_{w_0}$ simply reduces the exponent of $x_{2k}$ by one.  On the other hand, if $k=i-1$ (which occurs when $i>1$), then applying $\partial_{2k} = \partial_{2(i-1)}$ to $\partial_{2i-1} \mathfrak{S}_{w_0}$ breaks the monomial expression into the sum of two: a monomial where the exponent of $x_{2(i-1)}$ is reduced by one and another where the exponent of $x_{2i-1}$ is reduced by one. In particular, note that $\partial_{2}\partial_4\cdots \partial_{2\ell-2} \partial_{2i-1} \mathfrak{S}_{w_0}$ is divisible by $x_1^{n-2}x_2^{n-4}x_3^{n-4}x_4^{n-6}x_5^{n-6}\cdots x_{n-4}^2x_{n-3}^2$ for all $i$.

In general, when $i=1$ we have
\begin{align*}\label{eqn:divdiffone}
\partial_2  \partial_4  \cdots \partial_{n-2} \partial_{2 i-1}\mathfrak{S}_{w_0} (\mathbf{x})
&=\partial_2  \partial_4  \cdots  \partial_{n-2} \partial_1 \mathfrak{S}_{w_0}(\mathbf{x}) \\
&=x_1^{n-2} x_2^{n-3} x_3^{n-3}x_4^{n-5}x_5^{n-5} \cdots x_{n-4}^3x_{n-3}^3x_{n-2}^1 x_{n-1}^1 \\
&=x_1^{n-2}x_2^{n-4}x_3^{n-4}x_4^{n-6}x_5^{n-6}\cdots x_{n-4}^2x_{n-3}^2 (x_2x_3\cdots x_{n-1})
\end{align*}
If $2\leq i\leq \ell$, then
\begin{align*}
\partial_2  \partial_4  \cdots \partial_{n-2}  \partial_{2 i-1}\mathfrak{S}_{w_0}(\mathbf{x}) &=  x_1^{n-1} x_2^{n-3} x_3^{n-3} \cdots x_{2 i-4}^{n-(2 i-3)} x_{2 i-3}^{n-(2 i-3)} x_{2( i-1)}^{n-2 i} x_{2 i-1}^{n-2 i} \\ & \qquad\qquad\qquad \cdot x_{2 i}^{n-(2 i+1)} x_{2 i+1}^{n-(2 i+1)} \cdots x_{n-2} x_{n-1}\left(x_{2( i-1)}+x_{2 i-1}\right) \\
&=x_1^{n-2}x_2^{n-4}x_3^{n-4}x_4^{n-6}x_5^{n-6}\cdots x_{n-4}^2x_{n-3}^2 \\
& \qquad\qquad\qquad \cdot (x_1x_2\cdots \hat{x}_{2i-2}\cdots x_{n-1}+ x_1x_2\cdots \hat{x}_{2i-1}\cdots x_{n-1}).
\end{align*}
We conclude 
\begin{align*}
& \sum_{i=1}^{\ell} \partial_2  \partial_4  \cdots \partial_{n-2} \partial_{2 i-1}\mathfrak{S}_{w_0} (\mathbf{x}) \\
& \qquad\qquad\qquad  = x_1^{n-2}x_2^{n-4}x_3^{n-4}x_4^{n-6}x_5^{n-6}\cdots x_{n-4}^2x_{n-3}^2 \left( \sum_{k=1}^{n-1}  x_1x_2\cdots \hat{x}_k \cdots x_{n-1}\right),
\end{align*}
and the result now follows from Proposition~\ref{mononebigcup}.
\end{proof}

%%%%%%%%%%%%%%%%%%%%%%%
\subsection{Extension to the non-rectangular case}

As with all of our results, we can compute the cohomology class $[\SF_\gs]$ for $\gs$ any two-row standard tableau if we know the cohomology classes of the components indexed by its associated primes. The next example demonstrates this for a rectangular tableau, and the following lemma states the formula for general two-row tableaux.

\begin{ex}\label{ex.copiesof2rainbow}
Let $\cupsigma$ be the following noncrossing matching.
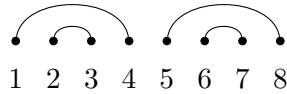
\begin{figure*}[h!]
\centering
\begin{tikzpicture}[scale=0.5]
\begin{scope}[shift={(8,0)}]
		\foreach \i in {0,...,7}
		{
			\node[pnt] at (\i,0)(\i){};
		}
        
		\draw (0)  to [bend left=80](3);
		\draw(1)  to [bend left=80] (2);
        \draw(4)  to [bend left=80] (7);
        \draw(5)  to [bend left=80] (6);
		\draw (0,-1) node {$1$};
        \draw (1,-1) node {$2$};
        \draw (2,-1) node {$3$};
        \draw (3,-1) node {$4$};
        \draw (4,-1) node {$5$};
        \draw (5,-1) node {$6$};
        \draw (6,-1) node {$7$};
        \draw (7,-1) node {$8$};
\end{scope}
\end{tikzpicture} 
\end{figure*}

\noindent In this case, we have that $\gs=\gs_1\circ\gs_2$ is the concatenation of two copies of 
\[
\begin{ytableau} 1 & 2\\ 3 & 4
\end{ytableau} \ .
\]
Using the computations from Example~\ref{ex.min4}, we have 
\[
[\cb_{\gs_1}]=[\cb_{\gs_2}]= x_1^2 (x_1x_2+x_1x_3+x_2x_3).
\]
We have that $\mu=(4,4)$ is the composition recording the sizes of the prime components of $\gs$ and
\[
\ci_\gs =  \ci_{\gs_1}'+\ci_{\gs_2}' + \left\langle  z_{5,j}, z_{6,j}, z_{7,j}, z_{8,j} \mid 1\leq j \leq 4  \right\rangle 
\]
where $\ci_{\gs_1}'=\ci_{\gs_1}$ and $\ci_{\gs_2}'$ is the image of the ideal $\ci_{\gs_2}$ under the shift $z_{a,b}\mapsto z_{a+4, b+4}$. It is straightforward to check using Macaulay2 that $\ci_\gs$ is prime and compute its multidegree. Using Lemma~\ref{lemma.multidegree} we have:
\begin{align*}
[\SF_\gs] &= x_1^4x_2^4x_3^4x_4^4 \cdot x_1^2 (x_1x_2+x_1x_3+x_2x_3) \cdot x_5^2 (x_5x_6+x_5x_7+x_6x_7) \\
&= x_1^4x_2^4x_3^4x_4^4 \cdot [\SF_{\gs_1}] \cdot [\SF_{\gs_2}](\mathbf{x}+4)
\end{align*}
where $[\SF_{\gs_2}](\mathbf{x}+4)$ is the image of $[\SF_{\gs_2}]$ after shifting $x_i\mapsto x_{i+4}$. We can also compute the Schubert expansion.  We have $[\SF_{\gs_1}]=[\SF_{\gs_2}] = \mathfrak{S}_{3241}(\mathbf{x}) + \mathfrak{S}_{4132}(\mathbf{x})$ and 
\[
[\cb_\gs] = \mathfrak{S}_{76853241}(\mathbf{x})+\mathfrak{S}_{85763241}(\mathbf{x})+\mathfrak{S}_{76854132}(\mathbf{x})+\mathfrak{S}_{85764132}(\mathbf{x}).
\]
This shows that $[\cb_\gs] =  \partial_2 \partial _6(\partial_1+\partial_2+\partial_3)(\partial_5+\partial_6+\partial_7)\mathfrak{S}_{w_0}$, confirming the results of Conjecture~\ref{conj:divdiff}.
\end{ex}

\begin{lemma}\label{lemma.genformula}
Let $\gs$ be a two-row tableau of shape $(n-\ell,\ell)$ with prime decomposition $\gs=\gs_1\circ\gs_2\circ\cdots\circ\gs_k$. For each $i$ with $1\leq i \leq k$, we set $\mu_i = |\gs_i|$ and $\parsum_i = \mu_1+\mu_2+\cdots +\mu_{i}$ denotes the $i$-th partial sum of the composition $\mu=(\mu_1, \mu_2, \ldots, \mu_k)$.
Then
\begin{equation}\label{eqn.genmultideg}
[\cb_\gs]=\left( \prod_{p=1}^{k-1}~ \prod_{j=\parsum_{p-1}+1}^{\parsum_p} x_j^{(n-\parsum_p)}\right) \left( \prod_{i=1}^k[\cb_{\gs_i}](\mathbf{x}+\parsum_{i-1}) \right)
\end{equation}
where $[\cb_{\gs_i}](\mathbf{x}+\parsum_{i-1})$ indicates that we shift the indices of the variables $x_j$ appearing in a polynomial representative of $[\SF_{\gs_i}]$ by $x_j\mapsto x_{j+\parsum_{i-1}}$
\end{lemma}

\begin{proof} By Definition \ref{def.general}, we have  $\isigma = \tau_\mu^{-1}\isigma'$. Since the action of $\tau_\mu^{-1}$ is given by $z_{a,b}\mapsto z_{\tau_\mu^{-1}(a),b}$ and our degree function is $\deg(z_{i,j}) = e_j$, we have
\[
\cc(\C[\mathbf{z}]/\isigma; \mathbf{x}) = \cc(\C[\mathbf{z}]/\isigma'; \mathbf{x}).
\]
The ideal $\isigma'$ is the sum of the shifted ideals $\ci_{\gs_i}'$ associated to every prime component $\gs_i$ of $\gs$ plus the prime monomial ideal
\begin{equation}\label{eqn.extragenI}
\mathcal{K}_\mu=\left< z_{i,j}\mid \parsum_{p-1} +1\leq j \leq \parsum_p <i \, \text{ for all }\, p= 1, \ldots ,k-1 \right>.
\end{equation}
Since each of $\ci_{\gs_1}', \ldots, \ci_{\gs_k}', \mathcal{K}_\mu$ have generators defined in distinct variable sets, any associated prime of $\ci_\gs$ is of the form $\cp_1+\cdots  + \cp_k+\mathcal{K}_\mu$ where each $\cp_i$ is an associated prime of $\ci_{\gs_i}'$. We also have $\cc(\C[\mathbf{z}]/\ci_{\gs_i}'; \mathbf{x}) = \cc(\C[\mathbf{z}]/\ci_{\gs_i}; \mathbf{x}+t_{i-1})$.
Using the properties of the multidegree (as in the proof of Proposition~\ref{mononebigcup}) and Lemma~\ref{lemma.multidegree} we get  
\begin{align*}
[\SF_{\gs}] &= \cc(\C[\mathbf{z}]/\mathcal{K}_\mu; \mathbf{x}) \cdot \prod_{i=1}^k \cc(\ci_{\gs_i}'; \mathbf{x}) \\
& = \left( \prod_{p=1}^{k-1}~ \prod_{j=\parsum_{p-1}+1}^{\parsum_p} x_j^{(n-\parsum_p)}\right)\left( \prod_{i=1}^k[\cb_{\gs_i}](\mathbf{x}+\parsum_{i-1}) \right),
\end{align*}
as desired.
\end{proof}

Using this formula we extend Conjecture~\ref{conj:divdiff} to arbitrary two-row tableaux. 

\begin{corollary} \label{conj:divdiff2} Let $\gs$ be a standard tableau of shape $(n-\ell, \ell)$ with prime decomposition $\gs= \gs_1\circ\gs_2\circ\cdots \gs_k$.  If Conjecture~\ref{conj:divdiff} holds for $\gs_1, \gs_2, \ldots, \gs_k$, then the same formula holds for $\gs$. 
\end{corollary}
\begin{proof} Let $\mu=(\mu_1, \mu_2, \ldots, \mu_k)$ be the composition determined by the sizes of the prime components of $\gs$. For each $i=1, 2, \ldots ,k$, let $y_{i}$ denote the longest permutation group $S_{\mu_i}$. For example if $\mu = (3,4)$, then  $y_1=321$ and $y_2=4321$.   It is straightforward to check that 
\[
\mathfrak{S}_{w_0}(\mathbf{x}) = \left( \prod_{p=1}^{k-1}~ \prod_{j=\parsum_{p-1}+1}^{\parsum_p} x_j^{(n-\parsum_p)} \right) \mathfrak{S}_{y_1}(\mathbf{x}) \mathfrak{S}_{y_2}(\mathbf{x}+\parsum_1)\cdots \mathfrak{S}_{y_k}(\mathbf{x}+t_{k-1}).
\]
If Conjecture~\ref{conj:divdiff} holds for $\gs_1 \ldots, \gs_k$, then the desired result follows immediately from Lemma~\ref{lemma.genformula} using the fact that $\partial_i\partial_j = \partial_j \partial_i$ whenever $|i-j|\geq 2$. 
\end{proof}

By Lemma~\ref{lemma.genformula} and Corollary~\ref{conj:divdiff2}, we can use Propositions~\ref{mononebigcup} and~\ref{prop.divdiff.onebigcup} to compute the cohomology class of $[\SF_\gs]$ for any two-row tableau $\gs$ whose prime components are rays and prime tableaux with one big cup.

\begin{ex}
Let $\cupsigma$ be the noncrossing matching from Example \ref{ex.disconnected3}, pictured below.
\begin{figure*}[h!]
\begin{tikzpicture}[scale=0.5] \begin{scope}[shift={(8,0)}]
\foreach \i in {0,...,9}{\node[pnt] at (\i,0)(\i){};}
\draw (0)  to [bend left=80](3);
\draw(1)  to [bend left=80] (2);
\draw (4,0) -- (4,1.5);
\draw (5,0) -- (5,1.5);
\draw (6)  to [bend left=80](9);
\draw(7)  to [bend left=80] (8);
     
\draw (0,-1) node {$1$};
\draw (1,-1) node {$2$};
\draw (2,-1) node {$3$};
\draw (3,-1) node {$4$};
\draw (4,-1) node {$5$};
\draw (5,-1) node {$6$};
\draw (6,-1) node {$7$};
\draw (7,-1) node {$8$};
\draw (8,-1) node {$9$};
 \draw (9,-1) node {$10$};
\end{scope}
\end{tikzpicture} 
\end{figure*}
\noindent The tableau $\gs$ has four prime components: $\gs_1, \gs_2, \gs_3, \gs_4$, where $\gs_1=\gs_4$ is the prime tableau appearing in Example~\ref{ex.min4} and $\gs_2=\gs_3$ is a single box. We have
\begin{align*}
[\cb_\gs]&=\partial_2\partial_8 (\partial_1+\partial_2+\partial_3)(\partial_7+\partial_8+\partial_9)\left(\mathfrak{S}_{w_0}\right) \\
&=(\partial_2\partial_1+\partial_2\partial_3)(\partial_8\partial_7+\partial_8\partial_9)\left(\mathfrak{S}_{w_0}\right)\\
&=(\partial_2\partial_1+\partial_2\partial_3)(\mathfrak{S}_{[10,9,8,7,6,5,4,1,3,2]}+ \mathfrak{S}_{[10,9,8,7,6,5,3,2,4,1]}) \\
&=\mathfrak{S}_{[9,8,10,7,6,5,3,2,4,1]}+\mathfrak{S}_{[10,7,9,8,6,5,3,2,4,1]}+\mathfrak{S}_{[9,8,10,7,6,5,4,1,3,2]}+\mathfrak{S}_{[10,7,9,8,6,5,4,1,3,2]}. \qedhere
\end{align*}
\end{ex}

%%%%%%%%%%

\end{document}